\documentclass[reqno]{amsart}
\usepackage{graphicx,xcolor,cite,mathtools,bbold,bbding}
\usepackage{amssymb,amsmath,amsthm,mathrsfs,paralist,bm,esint,setspace}
\usepackage[ngerman,french,english]{babel}
\RequirePackage[colorlinks,citecolor=blue,urlcolor=blue]{hyperref}
\usepackage{cite}
\usepackage{cases}
\usepackage{xfrac}
\usepackage{tcolorbox}
\usepackage{tikz}
\usepackage{pgfplots}
\usepackage{subcaption}
\usetikzlibrary{arrows.meta}
\pgfplotsset{compat=1.14}
\pgfplotsset{every x tick label/.append style={font=\footnotesize, yshift=0.6ex}}
\pgfplotsset{every y tick label/.append style={font=\footnotesize, xshift=0.5ex}}

\DeclareMathOperator{\Var}{Var}

\DeclareMathOperator{\supp}{\operatorname{supp}}

\newcommand{\N}{\mathbb{N}}
\newcommand{\Z}{\mathbb{Z}}

\newcommand{\R}{\mathbb{R}}

\renewcommand{\P}{\mathrm{P}}

\newcommand{\E}{\mathrm{E}}

\newcommand{\T}{\mathbb{T}}
\newcommand{\1}{\mathbb{1}}
\renewcommand{\d}{{\rm d}}
\renewcommand{\Re}{\mathrm{Re}}
\newcommand{\e}{{\rm e}}

\renewcommand{\ge}{\geqslant}
\renewcommand{\le}{\leqslant}
\definecolor{CYL}{rgb}{0.3,0.1,0.1}
\author[J. Hu and C.Y. Lee]{Jingwu Hu \and Cheuk Yin Lee}
\address{School of Science and Engineering, The Chinese University of Hong Kong, 
	Shenzhen, Guangdong 518172, China} 
\email{jingwuhu@link.cuhk.edu.cn}
\address{School of Science and Engineering, The Chinese University of Hong Kong, 
	Shenzhen, Guangdong 518172, China}
\email{leecheukyin@cuhk.edu.cn}
\title[Spatio-temporal increments of nonlinear parabolic SPDEs]{On the spatio-temporal increments of nonlinear parabolic SPDEs and the open KPZ equation}
\newtheorem{stat}{Statement}[section]
\newtheorem{proposition}[stat]{Proposition}

\newtheorem{corollary}[stat]{Corollary}
\newtheorem{theorem}[stat]{Theorem}
\newtheorem{lemma}[stat]{Lemma}
\theoremstyle{definition}

\newtheorem{remark}[stat]{Remark}
\newtheorem{OP}[stat]{Open Problem}
\newtheorem{question}[stat]{Question}

\numberwithin{equation}{section}
\begin{document}
\maketitle
\setcounter{tocdepth}{3}
\let\oldtocsection=\tocsection
\let\oldtocsubsection=\tocsubsection
\let\oldtocsubsubsection=\tocsubsubsection

\renewcommand{\tocsection}[2]{\hspace{0em}\oldtocsection{#1}{#2}}
\renewcommand{\tocsubsection}[2]{\hspace{2.5em}\oldtocsubsection{#1}{#2}}
\begin{abstract}
We study spatio-temporal increments of the solutions to nonlinear parabolic SPDEs on a bounded interval with Dirichlet, Neumann, or Robin boundary conditions. We identify the exact local and uniform spatio-temporal moduli of continuity for the sample functions of the solutions. These moduli of continuity results imply the existence of random points in space-time at which spatio-temporal oscillations are exceptionally large. We also establish small-ball probability estimates and Chung-type laws of the iterated logarithm for spatio-temporal increments. 
Our method yields extension of some of these results to the open KPZ equation on the unit interval with inhomogeneous Neumann boundary conditions.
Our key ingredients include new strong local non-determinism results for linear stochastic heat equation under various types of boundary conditions, and detailed estimates for the errors in linearization of spatio-temporal increments of the solution to the nonlinear equation.
\end{abstract}

\section{Introduction}

Fix $L>0$ and consider the solution $u = \{u(t\,,x)\}_{t \ge 0, x\in[0,L]}$ to the stochastic partial differential equation (SPDE, for short):
\begin{equation}\label{she}\left\{\begin{split}
	&\partial_t u = \tfrac12 \partial_x^2 u + b(u)+ \sigma(u) \xi
		&\text{on $\R_+\times(0\,,L)$},\\
	&u(0\,,x) = u_0(x) & \text{for all }x\in[0\,,L],
\end{split}\right.\end{equation}
where $\xi = \{\xi(t\,,x)\}_{t \ge 0, x \in [0\,,L]}$ is a space-time white noise defined on a complete probability space $(\Omega\,,\mathscr{F}\,,\P)$, $\sigma: \R \to \R$ and $b: \R \to \R$ are both non-random, globally Lipschitz functions, and $u_0\in L^2([0\,,L])$ is a non-random function.
Throughout we assume one of the following boundary conditions:
\begin{itemize}
\item Dirichlet boundary condition:
\begin{align}\label{D:BC}\tag{D}
		u = 0 \quad \text{at} \quad x = 0,\ x = L;
\end{align}
\item Neumann boundary condition:
\begin{align}\label{N:BC}\tag{N}
		\partial_x u = 0 \quad \text{at} \quad x=0,\ x = L;
\end{align}
\item Robin boundary condition:
\begin{align}\label{R:BC}\tag{R}
	\begin{cases}
		\partial_x u + \alpha u  = 0 \quad \text{at}\quad x = 0,\\
		\partial_x u + \beta u = 0 \quad\text{at}\quad x = L,
	\end{cases}
\end{align}
where $\alpha\,, \beta \in \R$ are constants.
\end{itemize}

Equations of the type \eqref{she} are sometimes referred to as reaction-diffusion equations \cite{Fife,KKM25, Cerrai, KKMS23}.
A special case of \eqref{she} is the stochastic heat equation with $b=0$ and $\sigma(u) = u$, which is also known as the parabolic Anderson model \cite{CM94, BC95, DK}.
The stochastic heat equation is closely related to the Kardar-Parisi-Zhang (KPZ) equation, which was originally introduced by \cite{KPZ} where the spatial domain is $\R$ or $\R^n$, and has deep connections to different systems and models in mathematical physics \cite{Hairer13, Corwin12, Quastel, HHT15}.
The open KPZ equation (see \eqref{kpz} below), introduced by Corwin and Shen \cite{CS18}, models stochastic interface growth on a bounded interval with inhomogeneous Neumann boundary conditions and arises from the open asymmetric simple exclusion process under a scaling limit \cite{CS18}.
The reader may refer to \cite{Corwin22, CK24, BKWW23, Yang22, KM22} for recent developments.

The primary goal of this paper is to study spatio-temporal regularities of the sample functions of solutions to \eqref{she} and the open KPZ equation \eqref{kpz}, and to establish detailed descriptions regarding local spatio-temporal increments.

In order to present our main results, let us define the parabolic-type metric $\rho$ on $[0\,,\infty) \times [0\,,L]$ by $\rho((t\,,x)\,,(s\,,y)) = \max\{ |t-s|^{1/4}\,,|x-y|^{1/2}\}$, and define
\begin{align*}
	&B_\rho((t\,,x)\,,r) = \{ (s\,,y) \in [0\,,\infty)\times[0\,,L] : \rho((t\,,x)\,,(s\,,y)) \le r \},\\
	&B^*_\rho((t\,,x)\,,r) = \{ (s\,,y) \in [0\,,\infty)\times[0\,,L] : 0<\rho((t\,,x)\,,(s\,,y)) \le r \}.
\end{align*}
Also, recall that when $b=0$ and $\sigma=0$, the weak solution to \eqref{she} is $G\ast u_0$, which is defined for any $z = (t\,,x) \in [0\,,\infty) \times [0\,,L]$ by
\begin{align}\label{G*u}
	G\ast u_0(z) := G_t \ast u_0(x) = \int_0^L G_t(x\,,y) u_0(y) \, \d y,
\end{align}
where $G$ is the heat kernel (see Section \ref{s:pre} below).
As is commonly done \cite{DPZ, Walsh}, the SPDE \eqref{she} is interpreted in its mild form
\begin{align*}
	u(t\,,x) = (G_t \ast u_0)(x) &+ \int_{(0,t) \times [0,L]} G_{t-s}(x\,,y) b(u(s\,,y))\, \d s\, \d y\\
	& + \int_{(0,t)\times [0,L]} G_{t-s}(x\,,y) \sigma(u(s\,,y))\, \xi(\d s \, \d y)
\end{align*}
for any $(t\,,x)\in (0\,,\infty) \times [0\,,L]$, where the last integral is a stochastic integral which can be defined in the sense of Walsh, and the existence of a unique solution is well known; see \cite{Walsh}.

\subsection{Main results}
Our main results apply to any one of the boundary conditions \eqref{D:BC}, \eqref{N:BC}, \eqref{R:BC}.
The first result identifies the exact local modulus of continuity for the spatio-temporal increments relative to a fixed base point in space-time, which exhibit a Khinchine-type law of the iterated logarithm (LIL).

\begin{theorem}[Law of the iterated logarithm]\label{th:she:lil}
For every fixed point $z_0=(t_0\,,x_0)\in(0\,,\infty)\times[0\,,L]$, there exists a constant $K_0\in(0\,,\infty)$ such that
\begin{align}\label{u:lil}
	\lim_{\varepsilon\to0^+}\sup_{z\in B^*_\rho(z_0,\varepsilon)} \frac{|u(z)-u(z_0)|}{\rho(z\,,z_0)\sqrt{\log\log(1/\rho(z\,,z_0))}} = K_0|\sigma(u(z_0))| \quad \text{a.s.}
\end{align}
The preceding continues to hold for $t_0=0$ with $u(z_0) = u_0(x_0)$ if in addition
\begin{align}\begin{split}\label{G*u:z_0}\textstyle
	&u_0 \text{ is bounded and, for some $r>0$,}\\
	&|G_t\ast u_0(x)-u_0(x_0)| \lesssim \rho((t\,,x)\,,(0\,,x_0)) \quad \forall (t\,,x) \in B_\rho((0\,,x_0)\,,r).
\end{split}\end{align}
\end{theorem}

As is customary, ``$f(a) \lesssim g(a)$'' means that there exists $C \in (0\,,\infty)$ such that $f(a) \le C g(a)$ for all $a$.
Theorem \ref{th:she:lil} says that for every fixed $z_0 \in (0\,,\infty)\times[0\,,L]$, there is a $\P$-null set (depending on $z_0$) off which \eqref{u:lil} holds.
See \cite{FKM15} for spatial LILs and \cite{WX24} for temporal LILs in a similar context of nonlinear SPDEs.

The next result complements the above by identifying the exact uniform modulus of continuity for the spatio-temporal increments.
Let us recall that a Borel set $A \subset \R$ is said to be \emph{polar} for $u$ if $\P\{ \exists (t\,,x) \in [0\,,\infty)\times[0\,,L]\,, u(t\,,x) \in A\} = 0$.

\begin{theorem}[Exact uniform modulus of continuity]\label{th:she:mc}
Assume that $\sigma^{-1}\{0\}$ is polar for $u$.
Then, for every fixed interval $I=[a\,,T]\times[c\,,d]$ with $0<a<T$ and $0\le c<d \le L$, there exists a constant $K \in(0\,,\infty)$ such that
\begin{align}\label{u:mc}
	\lim_{\varepsilon\to0^+}\sup_{\substack{z,z'\in I:0<\rho(z,z')\le\varepsilon}} \frac{|u(z')-u(z)|}{|\sigma(u(z))|\rho(z\,,z')\sqrt{\log(1/\rho(z\,,z'))}} = K \quad \text{a.s.}
\end{align}
The above statement extends to $a=0$ under the additional assumption that
\begin{align}\label{G*u:I}\textstyle
	u_0 \text{ is bounded and } |G\ast u_0(z')-G\ast u_0(z)| \lesssim \rho(z\,,z') \text{ on $[0\,,T]\times[c\,,d]$.}
\end{align}
\end{theorem}

\begin{remark}
When $\sigma$ is bounded away from 0, the polarity condition is satisfied tautologically.
When $b=0$ and $\sigma(u)=u$, under boundary condition \eqref{N:BC} or \eqref{R:BC}, the polarity condition is satisfied if $u_0$ is strictly positive on $[0\,,L]$, 
due to the fact that if $u_0>0$ then $u>0$ on $\R_+ \times [0\,,L]$; see \cite[Proposition 2.7]{CS18} and \cite[Proposition 4.2]{P19}; see also \cite{M91, DP93}.
\end{remark}

\begin{remark}
It is natural to require conditions \eqref{G*u:z_0} and \eqref{G*u:I} on $u_0$ in order to obtain exact moduli of continuity results near $t=0$ because the regularity of $G \ast u_0$ on $[0,T] \times [c\,,d]$ depends on that of $u_0$.
As is well known, if $u_0 \in C^\alpha([0\,,L])$ and $0<\alpha<1$, then $G \ast u_0 \in C^{\alpha/2,\alpha}([0\,,T] \times [0\,,L])$; see \cite[p.200, Theorem 5.1.17]{Lunardi}.
In particular, if $u_0 \in C^{1/2}([0\,,L])$, then \eqref{G*u:z_0} and \eqref{G*u:I} both hold. 
The bounds in the second parts of \eqref{G*u:z_0} and \eqref{G*u:I} are not optimal and can be replaced respectively, for instance, by conditions \eqref{u:lil:G*u} and \eqref{u:mc:G*u} below.
\end{remark}

Let us emphasize that the constants $K_0$ in \eqref{u:lil} and $K$ in \eqref{u:mc} are both finite and strictly positive, hence the modulus functions in \eqref{u:lil} and \eqref{u:mc} are exact.
Because of the presence of the logarithmic factor in \eqref{u:mc}, the sample functions $(t\,,x) \mapsto u(t\,,x)$ only belong to the space $C^{\sfrac{1}{4}-,\sfrac{1}{2}-}(I) = \textstyle{\bigcap_{0<\alpha<\sfrac{1}{4}} \bigcap_{0<\beta<\sfrac{1}{2}} C^{\alpha,\beta}(I)}$ but not $C^{\sfrac{1}{4},\sfrac{1}{2}}(I)$. 
This demonstrates the optimality of the H\"older regularity of $u$.

In the case that \eqref{she} is the linear stochastic heat equation with additive noise, i.e., when $b= 0$ and $\sigma=$ constant, the solution to \eqref{she} is Gaussian.
Exact local and uniform moduli of continuity are known for a large class of Gaussian random fields; see \cite{MR,MWX13,LX23}.
The results apply to the solutions to a family of linear SPDEs on $\R_+ \times \R^d$ with additive spatially homogeneous Gaussian noise \cite{LX23, HSWX20}.
Our results extend those results to the solutions to \eqref{she} which are non-Gaussian random fields when $\sigma$ or $b$ is non-constant.
In particular, \eqref{u:lil} states that the spatio-temporal increments of $u$ at a fixed point $z_0$ are locally of order $|\sigma(u(z_0))| \rho(z\,,z_0) \sqrt{\log\log(1/\rho(z\,,z_0))}$, but \eqref{u:mc} shows that the uniform modulus for the increments is of a larger order, at a logarithmic level.
The moduli of continuity results imply the existence of random exceptional points at which the spatio-temporal increments are larger than those at a fixed point, as stated below.

\begin{corollary}[Exceptional increments]\label{cor:she:ex}
Assume that $\sigma^{-1}\{0\}$ is polar for $u$.
Fix an interval $I=[a\,,T]\times[c\,,d]$, where $0<a<T$ and $0\le c<d \le L$. Let $K$ be the constant in \eqref{u:mc}. 
For every $\theta>0$, define the random set
\begin{align*}
	F(\theta)=\left\{ z \in I : \lim_{\varepsilon\to0^+} \sup_{z' \in B_\rho^*(z,\varepsilon)}\frac{|u(z')-u(z)|}{\rho(z\,,z')\sqrt{\log(1/\rho(z\,,z'))}} \ge \theta |\sigma(u(z))| \right\}.
\end{align*}
If $\theta>K$, then $F(\theta)=\varnothing$ a.s.;
if $\theta\in (0\,,K]$, then $F(\theta)$ has Lebesgue measure 0 a.s.;
and there exists $K' \in (0\,,K]$ such that if $0<\theta<K'$, then $F(\theta)$ is nonempty and dense in $I$ a.s. Consequently, the random set
\begin{align}\label{u:sing}
	\left\{ z \in I : \lim_{\varepsilon\to0^+} \sup_{z' \in B_\rho^*(z,\varepsilon)}\frac{|u(z')-u(z)|}{\rho(z\,,z')\sqrt{\log\log(1/\rho(z\,,z'))}} =\infty \right\}
\end{align}
has Lebesgue measure 0 and is dense in $I$ a.s.
\end{corollary}

The first result of this kind was proved for Brownian motion by Orey and Taylor \cite{OT74}, who also computed the Hausdorff dimension of fast points -- the set of times where Brownian increments fail to satisfy LIL and are exceptionally large.
Similar results are known for fractional Brownian motion \cite{KS00} and a class of Gaussian processes \cite{KPX00}.
The Hausdorff dimension of the set of exceptional spatial points for the stochastic heat equation on $\R_+ \times \R$ at which temporal increments fail to satisfy LIL is obtained in \cite{HK17}.
Let us mention that exceptional points of the type similar to \eqref{u:sing} are also studied for Brownian motion \cite{OT74}, Brownian sheet \cite{Walsh82, Walsh}, and stochastic wave equations \cite{CN88, LX22}, and are called singularities in the context of Brownian sheet and stochastic wave equations.
We leave some questions that appear to lie beyond the scope of this paper. An adaptation of the method of limsup random fractals \cite{KPX00, HK17} may lead to answers to some of these questions.

\begin{question}
Let $K^* = \sup\{ \theta \ge 0 : F(\theta) \ne \varnothing \text{ a.s.} \}$. Then $0<K^* \le K$, where $K$ is the constant in \eqref{u:mc}.
Is $K^* = K$? 
Is $F(K) \ne\varnothing$ a.s.?
Is it possible to compute or obtain formulas for these constants?
(See Theorem \ref{th:w:lil:mc} below for upper and lower bounds on $K$.)
\end{question}

\begin{question}
What are the (Hausdorff, Minkowski, packing) dimensions of $F(\theta)$ for $0<\theta\le K$? 
\end{question}

Our next set of results concerns small-ball probabilities and $\liminf$-type behavior of spatio-temporal increments.

\begin{theorem}[Small-ball probability]\label{th:she:sb}
Fix any point $z_0 = (t_0\,,x_0) \in[0\,,\infty)\times[0\,,L]$.
Assume that $b$ and $\sigma$ are bounded.
Let $\phi: (0\,,1] \to [1\,,\infty)$ be a function such that 
\begin{align}\label{phi}
	\phi(\varepsilon) = O(|\log \varepsilon|) \quad \text{as $\varepsilon \to 0^+$.}
\end{align}
When $t_0>0$, if $\inf_{x\in \R} |\sigma(x)| > 0$, then
there exist $\varepsilon_0 \in (0\,,1]$ and $C_0\,,C_1 \in (0\,,\infty)$ such that for all $\varepsilon \in (0\,,\varepsilon_0]$,
\begin{align}\label{u:sb}
	\e^{-C_1 \phi(\varepsilon)} \le \P\left\{ \sup_{z\in B_\rho(z_0,\varepsilon)}|u(z)-u(z_0)| \le \frac{\varepsilon}{(\phi(\varepsilon))^{1/6}} \right\} \le \e^{- C_0 \phi(\varepsilon)};
\end{align}
When $t_0=0$, if $\sigma(u_0(x_0)) \ne 0$, and if
\begin{align}\begin{split}\label{G*u:chung}
	&u_0 \text{ is bounded and, for some $r>0$ and $q>1$,}\\
	&|G_t\ast u_0(x)-u_0(x_0)| \lesssim [\rho((t\,,x)\,,(0\,,x_0))]^{q} \quad \forall (t\,,x) \in B_\rho((0\,,x_0)\,,r),
\end{split}\end{align}
then there exist $\varepsilon_0 \in (0\,,1]$ and $C_0\,,C_1 \in (0\,,\infty)$ such that for all $\varepsilon \in (0\,,\varepsilon_0]$,
\begin{align}\label{u:sb:t=0}
	\e^{-C_1 |\sigma(u_0(x_0))|^6 \phi(\varepsilon)} \le \P\left\{ \sup_{z\in B_\rho(z_0,\varepsilon)}|u(z)-u(z_0)| \le \frac{\varepsilon}{(\phi(\varepsilon))^{1/6}} \right\} \le \e^{- C_0 |\sigma(u_0(x_0))|^6 \phi(\varepsilon)}.
\end{align}
\end{theorem}

Small-ball probability estimates are known to imply Chung's LIL \cite{Chung, LS01}.
The following result holds regardless of whether or not $b$ and $\sigma$ are bounded.

\begin{theorem}[Chung-type LIL]\label{th:she:chung}
For every fixed $z_0 = (t_0\,,x_0)\in(0\,,\infty)\times[0\,,L]$, there exists a constant $C_2 \in (0\,,\infty)$ such that
\begin{align}\label{u:chung}
	\liminf_{\varepsilon\to0^+} \frac{(\log\log(1/\varepsilon))^{1/6}}{\varepsilon} \sup_{z\in B_\rho(z_0,\varepsilon)} |u(z)-u(z_0)| = C_2|\sigma(u(z_0))| \quad \text{a.s.}
\end{align}
The statement extends to $t_0 = 0$ under the additional assumption \eqref{G*u:chung}.
\end{theorem}

Similar small-ball probability and Chung-type LIL results for SPDEs such as \eqref{she} but on spatial domain $\T$ or $\R$ can be found in \cite{AJM21, LX23, Chen23, Chen24, Chen25, KKM24}.
Moreover, the existence of a small-ball constant for $t \mapsto u(t\,,x_0)$, where $u$ solves the stochastic heat equation on $\R_+ \times \T$, is established by Khoshnevisan et al \cite{KKM24}.
Theorem \ref{th:she:sb} is a spatio-temporal version of the result of \cite{KKM24} in a weaker form.

\begin{OP}
Let $\phi: (0\,,1] \to [1\,,\infty)$ be a function such that $\phi(\varepsilon) = O(|\log \varepsilon|)$ as $\varepsilon \to 0^+$. 
Does the limit (small-ball constant)
\[
	\lim_{\varepsilon\to0^+} \frac{1}{\phi(\varepsilon)} \log \P \left\{ \sup_{z \in B_\rho(z_0,\varepsilon)}|u(z)-u(z_0)| \le \frac{\varepsilon}{(\phi(\varepsilon))^{1/6}} \right\}
	\text{ exist?}
\]
\end{OP}

Our method yields similar temporal results and spatial results for \eqref{she}, which we state below without proof.
Also, our method continues to apply when the spatial domain is $\T$ or $\R$.

\begin{corollary}\label{cor:she:s/t}
For any fixed $(t_0\,,x_0) \in (0\,,\infty) \times [0\,,L]$, there exist constants $K_0\,,K_0'$, $C_0\,,C_0'$, $C_1\,,C_1'\,,C_2\,,C_2'\in (0\,,\infty)$ such that
\begin{gather*}
	\limsup_{\varepsilon\to0^+}\frac{|u(t_0+\varepsilon\,,x_0)-u(t_0\,,x_0)|}{\varepsilon^{1/4} \sqrt{\log\log(1/\varepsilon)}} = K_0|\sigma(u(t_0\,,x_0))| \quad \text{a.s.},\\
	\limsup_{\varepsilon\to0^+}\frac{|u(t_0\,,x_0+\varepsilon)-u(t_0\,,x_0)|}{\sqrt{\varepsilon\log\log(1/\varepsilon)}} = K_0'|\sigma(u(t_0\,,x_0))| \quad \text{a.s.},\\
	\liminf_{\varepsilon\to0^+} \left( \frac{\log\log(1/\varepsilon)}{\varepsilon} \right)^{1/4} \sup_{t: |t-t_0| \le \varepsilon} |u(t\,,x_0)-u(t_0\,,x_0)| = C_2 |\sigma(u(t_0\,,x_0))|\quad \text{a.s.},\\
	\liminf_{\varepsilon\to0^+} \left( \frac{\log\log(1/\varepsilon)}{\varepsilon} \right)^{1/2} \sup_{x: |x-x_0| \le \varepsilon} |u(t_0\,,x)-u(t_0\,,x_0)| = C_2' |\sigma(u(t_0\,,x_0))| \quad \text{a.s.},\\
	\e^{-C_1\phi(\varepsilon)} \le \P\left\{ \sup_{t: |t-t_0| \le \varepsilon} |u(t\,,x_0)-u(t_0\,,x_0)| \le \left(\frac{\varepsilon}{\phi(\varepsilon)} \right)^{1/4} \right\} \le \e^{-C_0 \phi(\varepsilon)},\\
	\e^{-C_1'\phi(\varepsilon)} \le \P\left\{ \sup_{x: |x-x_0| \le \varepsilon} |u(t_0\,,x)-u(t_0\,,x_0)| \le \left(\frac{\varepsilon}{\phi(\varepsilon)} \right)^{1/2} \right\} \le \e^{-C_0' \phi(\varepsilon)},
\end{gather*}
where the last two small-ball estimates hold under the additional conditions that $b$ is bounded and $|\sigma|$ is bounded above and away from 0, and that $\phi:(0\,,1] \to [1\,,\infty)$ satisfies $\phi(\varepsilon) = O(|\log \varepsilon|)$ as $\varepsilon \to 0^+$.

If $\sigma^{-1}\{0\}$ is polar for $u$, then for any fixed $0<a<T$ and $0\le c<d \le L$, there exist constants $K\,,K'$ such that
\begin{gather*}
	\lim_{\varepsilon\to0^+}\sup_{t,t'\in [a,T]: 0< |t-t'| \le \varepsilon}\frac{|u(t',x_0)-u(t\,,x_0)|}{|\sigma(u(t\,,x_0))| |t'-t|^{1/4}\sqrt{\log(1/|t'-t|)}} = K\quad \text{a.s.},\\
	\lim_{\varepsilon\to0^+}\sup_{x,x'\in [c,d]: 0< |x-x'| \le \varepsilon}\frac{|u(t_0\,,x')-u(t_0\,,x)|}{|\sigma(u(t_0\,,x))| \sqrt{|x'-x|\log(1/|x'-x|)}} = K' \quad \text{a.s.}
\end{gather*}
\end{corollary}

\subsection{The open KPZ equation}

As an application of the method of this paper, we study spatio-temporal increments for the open KPZ equation 
\begin{equation}\label{kpz}\left\{\begin{split}
	&\partial_t h = \tfrac12 \partial_x^2 h + \tfrac12 (\partial_x h)^2 + \xi
		&\text{on $\R_+\times(0\,,1)$},\\
	&h(0\,,x) = \log u_0(x) & \forall x\in [0\,,1],
\end{split}\right.\end{equation}
with inhomogeneous Neumann boundary condition
\begin{align}\label{kpz:bc}
	\partial_x h(t\,,0) = \mu, \qquad \partial_x h(t\,,1)=-\nu
	\qquad \forall t > 0,
\end{align}
where $\xi$ is a space-time white noise, $u_0 \in C([0\,,1])$ is a strictly positive continuous non-random function, and $\mu\,,\nu\in\R$ are constants.
The Hopf-Cole solution to \eqref{kpz} is given by
\begin{align}\label{hopf-cole}
	h(t\,,x) = \log u(t\,,x) \quad \ \forall t>0,\,x \in [0\,,1],
\end{align}
where $u$ is the solution to the stochastic heat equation
\begin{equation}\label{pam}\left\{\begin{split}
	&\partial_t u = \tfrac12 \partial_x^2 u + u \xi
		&\text{on $\R_+\times(0\,,1)$},\\
	&u(0\,,x) = u_0(x) &\forall x \in [0\,,1],
\end{split}\right.\end{equation}
with the Robin boundary condition
\begin{align}
	\partial_x u(t\,,0) = (\mu-\tfrac12)u(t\,,0), \quad \partial_x u(t\,,1) = -(\nu-\tfrac12) u(t\,,1) \qquad \forall t > 0.
\end{align}
Owing to strict positivity of $u$ (see \cite[Proposition 2.7]{CS18}), the logarithm in \eqref{hopf-cole} is well defined.
For the justification of the Hopf-Cole solution to \eqref{kpz}, see \cite{GH19}.

The theorem below identifies the exact local and uniform moduli of continuity for the spatio-temporal increments of the open KPZ equation, which extends the temporal result of Das \cite{Das24} and the spatial result of Foondun et al \cite{FKM15} for the KPZ equation on $\R_+ \times \R$.

\begin{theorem}\label{th:kpz:lil:mc}
For every fixed point $z_0=(t_0\,,x_0)\in(0\,,\infty)\times[0\,,1]$,
\begin{align}\label{kpz:lil}
	\lim_{\varepsilon\to0^+}\sup_{z\in B^*_\rho(z_0,\varepsilon)} \frac{|h(z)-h(z_0)|}{\rho(z\,,z_0)\sqrt{\log\log(1/\rho(z\,,z_0))}} = K_0
	\qquad \text{a.s.}
\end{align}
where $0<K_0<\infty$ is the same constant as in \eqref{u:lil}.
Moreover, for every fixed interval $I=[a\,,T]\times[c\,,d]$ with $0 < a < T$ and $0\le c<d\le 1$,
\begin{align}\label{kpz:mc}
	\lim_{\varepsilon\to0^+}\sup_{{z,z'\in I: 0<\rho(z,z')\le\varepsilon}} \frac{|h(z')-h(z)|}{\rho(z\,,z')\sqrt{\log(1/\rho(z\,,z'))}} = K_1 \qquad \text{a.s.}
\end{align}
where $0<K_1<\infty$ is the same constant as in \eqref{u:mc}.
Furthermore, \eqref{kpz:lil} and \eqref{kpz:mc} continue to hold when $t_0=0$ and $a=0$ under \eqref{G*u:z_0} and \eqref{G*u:I}, respectively.
\end{theorem}

Theorem \ref{th:kpz:lil:mc} implies the existence of exceptional spatio-temporal increments for the open KPZ equation:

\begin{corollary}\label{cor:kpz:ex}
Fix $I=[a\,,T]\times[c\,,d]$, where $0 < a < T$ and $0\le c<d\le 1$.
Let $K$ be the constant in \eqref{u:mc}
and \eqref{kpz:mc}. 
For every $\theta>0$, define the random set
\begin{align*}
	E(\theta)=\left\{ z \in I : \lim_{\varepsilon\to0^+} \sup_{z'\in B_\rho^*(z,\varepsilon)}\frac{|h(z')-h(z)|}{\rho(z\,,z')\sqrt{\log(1/\rho(z\,,z'))}} \ge \theta \right\}.
\end{align*}
If $\theta>K$, then $E(\theta)=\varnothing$ a.s.;
if $\theta\in (0\,,K]$, then $E(\theta)$ has Lebesgue measure 0 a.s.;
and there exists $K' \in (0\,,K]$ such that if $0<\theta<K'$, then $E(\theta)$ is nonempty and dense in $I$ a.s.
Consequently, the random set
\[
	\left\{ z \in I : \lim_{\varepsilon\to0^+} \sup_{z'\in B_\rho^*(z,\varepsilon)}\frac{|h(z')-h(z)|}{\rho(z\,,z')\sqrt{\log\log(1/\rho(z\,,z'))}} = \infty \right\}
\]
has Lebesgue measure 0 and is dense in $I$ a.s.
\end{corollary}

Moreover, we obtain a Chung-type LIL for the open KPZ equation:

\begin{theorem}\label{th:kpz:chung}
Fix $z_0 = (t_0\,,x_0) \in(0\,,\infty)\times[0\,,1]$. 
Then
\begin{align}
	\liminf_{\varepsilon\to0^+} \frac{(\log\log(1/\varepsilon))^{1/6}}{\varepsilon} \sup_{z\in B_\rho(z_0,\varepsilon)} |h(z)-h(z_0)| = C_2 \quad \text{a.s.}
\end{align}
where $C_2$ is the same constant as in \eqref{u:chung}.
This continues to hold when $t_0=0$ under the additional assumption \eqref{G*u:chung}.
\end{theorem}

Finally, we document the corresponding spatial results and temporal results, which can be obtained using the same proofs that lead to the above results for the open KPZ equation.

\begin{corollary}
For any fixed point $(t_0\,,x_0) \in (0\,,\infty) \times [0\,,1]$, and fixed numbers $0<a<T$, $0\le c<d\le 1$,
\begin{gather*}
	\limsup_{\varepsilon\to0^+} \frac{|h(t_0+\varepsilon\,,x_0)-h(t_0\,,x_0)|}{\varepsilon^{1/4}\sqrt{\log\log(1/\varepsilon)}} = K_0 \quad \text{a.s.},\\
	\limsup_{\varepsilon\to0^+} \frac{|h(t_0\,,x_0+\varepsilon)-h(t_0\,,x_0)|}{\sqrt{\varepsilon \log\log(1/\varepsilon)}} = K_0' \quad \text{a.s.},\\
	\lim_{\varepsilon\to0^+} \sup_{t,t'\in[a,T]: 0<|t-t'|\le\varepsilon} \frac{|h(t',x_0)-h(t\,,x_0)|}{|t'-t|^{1/4}\sqrt{\log(1/|t'-t|)}} = K \quad \text{a.s.},\\
	\lim_{\varepsilon\to0^+} \sup_{x,x'\in[c,d]: 0<|x-x'|\le\varepsilon} \frac{|h(t_0\,,x')-h(t_0\,,x)|}{\sqrt{|x'-x|\log(1/|x'-x|)}} = K' \quad \text{a.s.},\\
	\liminf_{\varepsilon\to0^+} \left( \frac{\log\log(1/\varepsilon)}{\varepsilon} \right)^{1/4} \sup_{t:|t-t_0| \le \varepsilon} |h(t\,,x_0)-h(t_0\,,x_0)| = C_2 \quad \text{a.s.},\\
	\liminf_{\varepsilon\to0^+} \left( \frac{\log\log(1/\varepsilon)}{\varepsilon} \right)^{1/2} \sup_{x:|x-x_0| \le \varepsilon} |h(t_0\,,x)-h(t_0\,,x_0)| = C_2' \quad \text{a.s.},
\end{gather*}
where $K_0\,,K_0'\,,K\,,K'\,,C_2\,,C_2'$ are the same constants as in Corollary \ref{cor:she:s/t}.
\end{corollary}

\begin{OP}
What are optimal bounds for the small-ball probabilities
\begin{gather*}
\P\left\{ \sup_{z \in B_\rho(z_0,r)}|h(z)-h(z_0)| \le \varepsilon \right\},\\
\P\left\{ \sup_{t:|t-t_0|\le r}|h(t\,,x_0)-h(t_0\,,x_0)| \le \varepsilon \right\}, \
\P\left\{ \sup_{x: |x-x_0|\le r}|h(t_0\,,x)-h(t_0\,,x_0)| \le \varepsilon \right\}?
\end{gather*}
\end{OP}

\subsection{Proof ideas and contributions}

Similar spatial and temporal LILs and moduli of continuity results for SPDEs of the type \eqref{she} but on spatial domain $\T$ or $\R$ are established in \cite{Das24, HSWX20, KKM24}.
Their arguments build on either the Lei-Nualart decomposition \cite{LN09} or the Mueller-Tribe pinned string method \cite{MT02} for the linear equation, which essentially states that the solution can be decomposed into $u_1+u_2$, where $u_2$ has smooth sample paths and $u_1$ is a fractional Brownian motion or a Gaussian random field with stationary increments.
These results or methods do not seem to carry over directly to the case of bounded interval domains with Robin boundary condition and when $u$ is treated as a spatio-temporal process.
In the case of Dirichlet or Neumann boundary condition, the heat kernel can be decomposed as $G=\Gamma +H$, where $\Gamma$ is the heat kernel on the full line $\R$ and $H$ is a smooth function, which can be derived using the method of images or Poisson summation formula \cite{Walsh, K09, DS24}, but this decomposition method does not seem to apply readily to the case of Robin boundary conditions either.
Even in Dirichlet and Neumann cases, the issue with this decomposition is that the component $u_1(t\,,x) := \int_0^t \int_0^L \Gamma_{t-s}(x-y) \xi(\d s\, \d y)$ does not have stationary increments for $(t\,,x) \in \R_+ \times [0\,,L]$ due to the presence of boundaries.
In order to circumvent these technical obstacles, we appeal to a different approach using the strong local non-determinism (SLND) method for the linear equation \cite{LX19, LX23} and combine it with the method of linearization of the nonlinear equation \cite{KSXZ13, FKM15, KKM24, HP15, Hairer13}.

It might help to recall that a Gaussian random field $\{X(z)\}_{z \in I}$ with $I \subset \R^N$ is said to be \emph{strongly locally non-deterministic} \cite{Berman, Pitt, CD82, MP87, Xiao08} if
there exists $C>0$ such that
\[
	\Var(X(z) \mid X(z_1)\,,\dots,X(z_n)) \ge C \min_{1 \le i \le n} \Var(X(z)-X(z_i))
\]
uniformly for all $n \in \N_+$ and for all $z\,,z_1\,,\dots,z_n \in I$.
Under Dirichlet or Neumann condition, we prove the spatio-temporal SLND property for the linear equation with additive noise (see Section \ref{s:w} below) by adopting the method of \cite{LX19, L22, LX23, LX25} based on Fourier transform.
The case of Robin condition \eqref{R:BC} requires a separate treatment because the heat kernel is not amenable to Fourier transform in the spatial variable $x$.
We devise a proof that bypasses the use of Fourier transform in $x$ and uses instead the orthonormal basis of eigenfunctions to establish the spatio-temporal SLND property under \eqref{R:BC}, which is more natural and adaptable to the domain and its boundary condition.
This idea appears to be new in the context of SLND for SPDEs and may make it possible to study SPDEs on general bounded domains or fractal domains with various boundary conditions (see, e.g., \cite{CCL24, HY18, BCHOT}) and to investigate their optimal H\"older regularities, exact moduli of continuity, small-ball probabilities, etc.
Our SLND results are sharp and yield matching bounds for the conditional variance (see Lemmas \ref{lem:SLND:DN} and \ref{lem:SLND:R}).
As a result, we also obtain matching upper and lower bounds for the variance of spatio-temporal increments under \eqref{D:BC} and \eqref{N:BC}, which improve the bounds in \cite{DS24}, and obtain new matching bounds under \eqref{R:BC} (see Proposition \ref{pr:var:w-w:op}).
All of these matching bounds are valid up to $t=0$ and up to the boundaries of the interval.

The SLND property allows us to apply the framework in \cite{LX23} to the case of additive noise.
In particular, \cite{LX23} establishes exact uniform and local moduli of continuity, sharp small-ball probability bounds, and Chung's law of the iterated logarithm for a large class of anisotropic Gaussian random fields under the SLND property together with a few other assumptions.
We verify those assumptions and apply the results of \cite{LX23} to obtain our main results in the Gaussian case.
But for spatio-temporal increments at $t=0$ or at the boundaries under Dirichlet condition, due to the boundary effects, the variance bounds have a different form than in the framework of \cite{LX23}, so the results of \cite{LX23} cannot be directly applied and these cases need to be treated separately.
Since spatio-temporal SLND implies spatial SLND and temporal SLND, our method also yields spatial results and temporal results. 

In order to go from the Gaussian case to the non-Gaussian case, we adopt the idea of linearization of the SPDE and localization of heat kernel in \cite{KSXZ13, FKM15}, but without Fourier transform, and obtain detailed estimates for the spatio-temporal linearization errors (see Section \ref{s:E} below).
Our work demonstrates that a crude heat kernel bound (Lemma \ref{lem:G} below) is enough for carrying out the spatio-temporal localization analysis without the use of Gaussian bounds for heat kernel, making it possible for extensions to SPDEs with more general differential operators (see, e.g., \cite{N18, FN17}).
Finally, the local and uniform moduli of continuity and Chung-type LIL for the open KPZ equation can be obtained through linearization of the Hopf-Cole solution, which relates the spatio-temporal increments to those of the stochastic heat equation with multiplicative noise and allows application of our results for \eqref{she}. To the best of our knowledge, our results for the open KPZ equation are new.

%

\subsection{An outline of the paper}

In Section \ref{s:pre}, we gather some basic spectral properties of eigenpairs under various boundary conditions, and present a heat kernel estimate.
In Section \ref{s:w}, we investigate the constant-coefficient case $b=0$ and $\sigma = 1$ in \eqref{she}, establish variance estimates and the SLND property, and obtain exact local and uniform spatio-temporal moduli of continuity, small-ball probability estimates, and a Chung-type LIL for the solution.
In Section \ref{s:E}, we consider linearization of the nonlinear equation \eqref{she} and establish detailed estimates for the linearization error for the spatio-temporal increments.
In Section \ref{s:pf}, we present the proofs of the main results, namely, Theorems \ref{th:she:lil}, \ref{th:she:mc}, Corollary \ref{cor:she:ex}, and Theorems \ref{th:she:sb} and \ref{th:she:chung}.
Finally, in Section \ref{s:kpz}, we prove Theorem \ref{th:kpz:lil:mc}, Corollary \ref{cor:kpz:ex}, and Theorem \ref{th:kpz:chung} for the open KPZ equation.

\subsection{Notations}

Let us end the Introduction with a list of notations that will be used throughout the paper:
$\N_+ = \{1\,,2\,,\dots\}$; 
$\N_0 = \{0\,,1\,,2\,,\dots\}$;
$\R_+ = (0\,,\infty)$;
$\# A$ denotes cardinality of a set $A$;
$\1_A$ denotes indicator function of the set $A$;
$a\wedge b = \min\{a\,,b\}$; 
$a\vee b = \max\{a\,,b\}$; 
$\log_+(x) = \log(x \vee \e)$; 
``$f(x) \lesssim g(x)$'' means that there exists $C\in (0\,,\infty)$ such that $f(x) \le C g(x)$ for all $x$;
``$f(x) \asymp g(x)$'' means that there exist $C_1\,, C_2 \in (0\,,\infty)$ such that $C_1 g(x) \le f(x) \le C_2 g(x)$ for all $x$;
``$f(x) \sim g(x)$ as $x \to a$'' means that $f(x)/g(x)\to 1$ as $x \to a$;
``$f(x)=O(g(x))$'' means that there exists $C\in(0\,,\infty)$ such that $|f(x)| \le C|g(x)|$;
``$f(x) \propto g(x)$'' means that there exists $C\in (0\,,\infty)$ such that $f(x) = Cg(x)$ for all $x$;
For any $p \in [1\,,\infty)$, $\|\cdot\|_p$ denotes $L^p(\Omega\,,\mathscr{F}\,,\P)$-norm, i.e., $\|X\|_p = (\E|X|^p)^{1/p}$ for any random variable $X$.

\section{Preliminaries}\label{s:pre}

Let $\{ (\lambda_n\,, f_n) \}_{n\in\N_+}$ denote the eigenpairs of the Laplace operator $-\tfrac12\partial_x^2$ on $(0\,,L)$ with any one of the boundary conditions \eqref{D:BC}, \eqref{N:BC}, \eqref{R:BC}.
In other words, each $f_n$ satisfies $-\tfrac12 f_n'' = \lambda_n f_n$ on $(0\,,L)$ with the prescribed boundary condition.
We always assume that the eigenvalues are arranged in ascending order $\lambda_1 \le \lambda_2 \le \dots$ and each $f_n$ is normalized to have $\|f_n\|_{L^2}=1$.

\begin{lemma}\label{lem:eigen}
The following properties hold:
\begin{enumerate}
\item Under Dirichlet boundary condition \eqref{D:BC},
\begin{align}\label{eigen:D}
	\qquad\qquad \lambda_n = \tfrac12\big( \tfrac{\pi n}{L} \big)^2, \qquad f_n(x)=\sqrt{\tfrac{2}{L}}\sin\big( \tfrac{n\pi x}{L} \big) \qquad\quad \text{ for $n \in \N_+$.}
\end{align}
\item Under Neumann boundary condition \eqref{N:BC},
\begin{align}\begin{split}\label{eigen:N}
	&\qquad\qquad \lambda_n = \tfrac12\big( \tfrac{\pi (n-1)}{L} \big)^2 \hskip140pt \text{ for $n \in \N_+$,}\\
	&\qquad\qquad f_1(x) = \sqrt{\tfrac{1}{L}}\quad\text{and}\quad
	f_n(x)=\sqrt{\tfrac{2}{L}}\cos\big( \tfrac{(n-1)\pi x}{L} \big) \quad \text{for $n \ge 2$.}
\end{split}\end{align}
\item Under Robin boundary condition \eqref{R:BC}, $0$ is an eigenvalue iff $\alpha(1+\beta L)=\beta$.
There are at most finitely many negative eigenvalues $\{\lambda_n \}_{1\le n < n_0}$, where $1 \le n_0 < \infty$, and the other eigenvalues $\{\lambda_n \}_{n \ge n_0}$ are nonnegative.
\begin{enumerate}[i.]
\item If $\alpha(1+\beta L)=\beta$, then 
\[
	\lambda_n = \begin{cases}
	-\frac12 \kappa_n^2, & 1 \le n < n_0,\\
	0, & n = n_0,\\
	\frac12 \eta_n^2, & n > n_0,
	\end{cases}
\]
where $\kappa_n$ and $\eta_n$ are the positive roots of the equations
\begin{align}\label{eq:tan}
	&\tanh(\kappa_n L) = \frac{(\alpha-\beta)\kappa_n}{\kappa_n^2 - \alpha\beta},\qquad \tan(\eta_n L) = \frac{(\beta-\alpha)\eta_n}{\eta_n^2+\alpha \beta},
\end{align}
and $f_n = \|e_n\|_{L^2}^{-1}\, e_n$, where $e_{n_0}(x) = 1-\alpha x$ and 
\begin{align}\label{e_n}
	e_n(x) = \begin{cases}
	\cosh(\kappa_n x) - \frac{\alpha}{\kappa_n}\sinh(\kappa_n x), & 1 \le n < n_0,\\
	\cos(\eta_n x) - \frac{\alpha}{\eta_n} \sin(\eta_n x), & n > n_0.
	\end{cases}
\end{align}
\item If $\alpha(1+\beta L)\ne\beta$, then 
\[
	\lambda_n = \begin{cases}
	-\frac12 \kappa_n^2, & 1 \le n < n_0,\\
	\frac12 \eta_n^2, & n \ge n_0,
	\end{cases}
\]
where $\kappa_n$ and $\eta_n$ are the positive roots of the equations \eqref{eq:tan} and $f_n = \|e_n\|_{L^2}^{-1}\, e_n$, where $e_n$ is given by \eqref{e_n}.\smallskip
\end{enumerate}
In particular, there exists $n_1 \in \Z$ such that
\begin{align}\label{eta}
	\qquad\quad \eta_n = \tfrac{\pi (n_1+n)}{L} + O(\tfrac{1}{n}) \quad \text{and}\quad
	\|e_n\|_{L^2}^{-2} = \tfrac{2}{L}(1+O(\tfrac{1}{n^2})) \quad \text{as $n \to \infty$.}
\end{align}
\end{enumerate}
In all cases,
\begin{align}\label{lambda}
	\lambda_n \asymp n^2, \qquad
	&0 \le \lambda_{n+1}-\lambda_n \lesssim n \quad \text{as $n \to \infty$,}\\
	\label{f_n:bd}
	\textstyle\sup_{n\ge 1, x \in [0,L]}|f_n(x)| < \infty, \qquad
	&\textstyle\sup_{n \ge 1, x\in [0,L]}|n^{-1} f_n'(x)| < \infty,
\end{align}
and $\{f_n\}_{n\ge 1}$ is an orthonormal basis for $L^2([0\,,L])$ under ${\langle f,g\rangle}_{L^2} = \int_0^L f(x)g(x)\d x$.
\end{lemma}

\begin{proof}
Cases 1 and 2 are standard and routine eigenvalue problems, so we omit the proof.
As for case 3, note that $0$ is an eigenvalue iff $e_1(x) = A+Bx$, where $(A\,,B)\ne(0\,,0)$, is an eigenfunction satisfying condition \eqref{R:BC}. It is easy to see that the last condition is satisfied iff $\alpha(1+\beta L)=\beta$, in which case $e_1(x) = 1-\alpha x$ is an eigenfunction. From the equation $-\tfrac12 e''=\lambda e$, any other eigenpair $(\lambda\,, e)$ must have the form $e(x) = A\cos(\eta x) + B \sin(\eta x)$ with $\lambda=\tfrac12\eta^2 \ge 0$, or $e(x) = A\cosh(\kappa x) + B \sinh(\kappa x)$ with $\lambda=-\tfrac12\kappa^2 < 0$.
Then, from the boundary condition \eqref{R:BC}, one can readily deduce \eqref{eq:tan} and \eqref{e_n}.
The first equation in \eqref{eq:tan} has at most finitely many solutions because as $\kappa \to \infty$, $\tanh(\kappa L)\to1$  while the right-hand side tends to 0 and has at most one singularity on $(0\,,\infty)$.
The function $\eta \mapsto (\beta-\alpha)\eta/(\eta^2+\alpha\beta)$ has at most one singularity on $(0\,,\infty)$, is eventually increasing or decreasing to 0, and is $\asymp (\beta-\alpha)\eta^{-1}$ as $\eta \to \infty$.
It is easy to deduce from these properties that there is $n_1\in \Z$ such that for sufficiently large $n \in \N$, every interval $I_n:=(\pi(k_n-1/2)/L\,,\pi(k_n+1/2)/L)$, where $k_n=n_1+n$, contains exactly one solution $\eta_n$ to equation \eqref{eq:tan}.
Hence, $\eta_n \sim n\pi/L$ as $n \to \infty$.
Also, since $\tan(z L) \asymp z$ for $|z|$ small, it follows that
\begin{align*}
	\left|\eta_n - \frac{\pi k_n}{L}\right| \lesssim |\tan(\eta_n L - \pi k_n)| = |\tan(\eta_n L)| = \frac{|\beta-\alpha| \eta_n}{|\eta_n^2 + \alpha \beta|} \lesssim \eta_n^{-1} \lesssim n^{-1} \text{ as $n \to \infty$,}
\end{align*}
which shows the first property in \eqref{eta}.
This together with \eqref{e_n} implies that
\begin{align*}
	\|e_n\|_{L^2}^2 = \frac{\beta(\eta_n^2+\alpha^2)}{2\eta_n^2(\eta_n^2 + \beta^2)} - \frac{\alpha}{2\eta_n^2} + \frac{L}{2}\left( 1+ \frac{\alpha^2}{\eta_n^2} \right) = \frac{L}{2}( 1+ O(n^{-2}) ) \quad \text{as $n \to \infty$.}
\end{align*}
The property \eqref{lambda} and uniform bound \eqref{f_n:bd} follow readily.
Finally, the last assertion follows from general spectral theory for elliptic operators; see, e.g., \cite[Theorem 5.11]{Teschl} or \cite[Theorem 4.12]{McLean}.
\end{proof}

We frequently use the following Parseval's identity, which is a direct consequence of $\{f_n\}_{n\in\N_+}$ being an orthonormal basis for $L^2([0\,,L])$:
For all $\phi \in L^2([0\,,L])$,
\begin{align}\label{parseval}\textstyle
	\|\phi\|^2_{L^2}=\sum_{n=1}^\infty \left|{\langle \phi\,,f_n\rangle}_{L^2}\right|^2.
\end{align}
The heat kernel for $\partial_t - \tfrac12 \partial_x^2$ under the respective boundary condition \eqref{D:BC}, \eqref{N:BC} or \eqref{R:BC} is given by
\begin{align}\label{G}
	G_t(x\,,y) = \sum_{n=1}^\infty \e^{-\lambda_n t} f_n(x) f_n(y), \quad t > 0,\, x, y \in [0\,,L].
\end{align}
A measurable process $u=\{u(t\,,x)\}_{t \ge 0, x \in [0,L]}$ is called a mild solution to \eqref{she} if it is adapted to the filtration $\{\mathscr{F}_t\}_{t\ge 0}$ of the noise $\xi$ and satisfies the integral equation
\begin{align}\begin{split}\label{she:mild}
	u(t\,,x) = (G_t \ast u_0)(x) &+ \int_{(0,t) \times [0,L]} G_{t-s}(x\,,y) b(u(s\,,y))\, \d s\, \d y\\
	& + \int_{(0,t)\times [0,L]} G_{t-s}(x\,,y) \sigma(u(s\,,y))\, \xi(\d s \, \d y)
\end{split}\end{align}
for any $(t\,,x)\in (0\,,\infty) \times [0\,,L]$.
It follows from standard existence and uniqueness theory that \eqref{she} has a unique mild solution \cite{Walsh, DS24}; see also \cite[Proposition 2.7]{CS18}.
Some moment estimates for the solution and its spatial and temporal increments will be established in Sections \ref{s:w} and \ref{s:E}.

The next lemma states a heat kernel estimate.
It follows from known Gaussian bounds on the heat kernel, but our main results and methods do not rely on the Gaussian bounds.
The estimate \eqref{G:bd} below will be enough for our purposes.

\begin{lemma}\label{lem:G}
Under \eqref{D:BC}, \eqref{N:BC} or \eqref{R:BC}, for any $T>0$, there exists $C>0$ such that
\begin{align}\label{G:bd}
	|G_t(x\,,y)| \le C\left( \frac{1}{\sqrt{t}} \wedge \frac{t}{|x-y|^{3}} \right) \quad \text{for all $t \in (0\,,T]$ and $x\,,y\in [0\,,L]$.}
\end{align}
\end{lemma}

\begin{proof}
Under \eqref{D:BC}, there exist $C_1\,,C_2>0$ such that
\[
	0 \le G_t(x\,,y) \le \frac{C_1}{\sqrt{t}} \exp\left(- \frac{(x-y)^2}{C_2 t}\right) \quad \forall t>0\,, x\,,y \in [0\,,L];
\]
see \cite[Corollary 3.2.8]{Davies}.
Under \eqref{N:BC}, there exist $C_3\,,C_4>0$ such that
\[
	0 \le G_t(x\,,y) \le C_3 \left( \frac{1}{\sqrt{t}} \vee 1 \right) \exp\left( - \frac{(x-y)^2}{C_4t} \right) \quad \forall t>0\,, x\,,y \in [0\,,L];
\]
see \cite[Theorem 3.2.9]{Davies} or \cite[Proposition 3.6]{CCL24}.
Under \eqref{R:BC}, for any $T>0$, there exist $C_5\,,C_6>0$ such that
\[
	0 \le G_t(x\,,y) \le \frac{C_5}{\sqrt{t}}\exp\left( - \frac{(x-y)^2}{C_6t} \right) \quad \forall t \in (0\,,T]\,,x\,,y\in [0\,,L];
\]
see \cite[Lemma 4.3]{CS18}.
The inequality \eqref{G:bd} follows from these estimates and the elementary property that $\sup_{z > 0} z^{3/2} \exp(-z^2)<\infty$.
\end{proof}

\begin{lemma}\label{lem:G*u}
For any $0<a<b$, there exists $C>0$ such that
\begin{enumerate}
\item $|(G_t \ast u_0)(x)- (G_t\ast u_0)(x')| \le C |x'-x|$,
\item $|(G_{t'} \ast u_0)(x)- (G_t\ast u_0)(x)| \le C |t'-t|$
\end{enumerate}
uniformly for all $t\,,t' \in [a\,,b]$ and $x\,,x'\in [0\,,L]$.
\end{lemma}

\begin{proof}
Recall that $u_0\in L^2([0\,,L])$.
By \eqref{G}, \eqref{lambda}, mean value theorem, and \eqref{f_n:bd},
\begin{align*}
&\textstyle|(G_t \ast u_0)(x)- (G_t\ast u_0)(x')|
=| \sum_{n=1}^\infty \e^{-\lambda_n t} (f_n(x)-f_n(x')) \langle f_n\,,u_0\rangle_{L^2} |\\
&\textstyle\lesssim  \sum_{n=1}^\infty \e^{-cn^2t} n |x-x'| \|u_0\|_{L^2} \lesssim |x-x'| \int_0^\infty \e^{-cz^2 t} z\, \d z \propto  \tfrac{1}{t} |x-x'| \le \tfrac{1}{a} |x-x'|
\end{align*}
uniformly for all $t\in [a\,,b]$ and $x\,,x'\in[0\,,L]$.
Similarly,
\begin{align*}
	&\textstyle|(G_{t'} \ast u_0)(x)- (G_t\ast u_0)(x)|
	=| \sum_{n=1}^\infty (\e^{-\lambda_n t'}-\e^{-\lambda_n t})f_n(x)\langle f_n\,,u_0\rangle_{L^2} |\\
	&\textstyle\lesssim  \sum_{n=1}^\infty \e^{-\lambda_n t} n^2 |t'-t| \lesssim |t'-t| \int_0^\infty \e^{-cz^2 t} z^2\, \d z \propto t^{-3/2} |t'-t| \le a^{-3/2} |t'-t|
\end{align*}
uniformly for all $t<t'$ in $[a\,,b]$ and $x\in [0\,,L]$.
This completes the proof.
\end{proof}

\section{The Gaussian case}\label{s:w}

In this section, we study the special case of \eqref{she} where $b\equiv 0$ and $\sigma \equiv 1$. In other words,
\begin{equation}\label{she:add}\left\{\begin{split}
	&\partial_t w = \tfrac12 \partial_x^2 w+ \xi
		&\text{on $\R_+\times(0\,,L)$},\\
	&w(0\,,x) = 0 &\text{for all $x\in[0\,,L]$}
\end{split}\right.\end{equation}
with boundary condition \eqref{D:BC}, \eqref{N:BC}, or \eqref{R:BC}.
The unique mild solution to \eqref{she:add} is the centered Gaussian random field
\begin{align}\label{w}
	w(t\,,x) = \int_{(0,t)\times[0,L]} G_{t-s}(x\,,y)\, \xi(\d s\, \d y),
	\quad t > 0,\, x \in [0\,,L],
\end{align}
where $G$ is given by \eqref{G}.

\subsection{Basic estimates}

\begin{lemma}\label{lem:int:G}
For any $T>0$, $\int_0^L [G_t(x\,,y)]^2 \d y \lesssim t^{-1/2}$ and $\int_0^t \d s \int_0^L \d y \, [G_s(x\,,y)]^2 \lesssim \sqrt{t}$ uniformly for all $t\in (0\,,T]$ and $x \in [0\,,L]$.
\end{lemma}

\begin{proof}
By Parseval's identity, \eqref{lambda}, and \eqref{f_n:bd},
\begin{align*}\textstyle
	\int_0^L [G_t(x\,,y)]^2 \d y = \sum_{n=1}^\infty \e^{-2\lambda_n t} |f_n(x)|^2 \lesssim \int_0^\infty \e^{-cz^2 t} \d z \propto t^{-1/2}.
\end{align*}
Replace $t$ by $s$, and then integrate to finish the proof.
\end{proof}

\begin{lemma}\label{lem:int:G:x}
Fix $T>0$. There exists a constant $c>0$ such that
\begin{enumerate}
\item $\int_0^L [G_{t}(x\,,y) - G_t(x',y)]^2\, \d y \lesssim \sum_{n=1}^\infty (|x-x'|^2 n^2 \wedge 1) \, \e^{-cn^2 t}$,
\item $\int_0^t \d s \int_0^L \d y\, [G_{s}(x\,,y) - G_{s}(x',y)]^2 \lesssim |x'-x|$
\end{enumerate}
uniformly for all $t \in (0\,,T]$ and $x\,,x'\in[0\,,L]$.
\end{lemma}

\begin{proof}
The first inequality can be derived by applying Parseval's identity, mean value theorem, \eqref{lambda}, and \eqref{f_n:bd}:
\begin{align*}
	&\textstyle\int_0^L [G_{t}(x\,,y) - G_t(x',y)]^2\, \d y
	= \sum_{n=1}^\infty \e^{-2\lambda_n t} |f_n(x)-f_n(x')|^2\\
	&\textstyle\lesssim \sum_{n=1}^\infty \e^{-2\lambda_n t} \left[\left({\|f_n'\|}_{L^\infty}|x-x'|\right) \wedge \left(2{\|f_n\|}_{L^\infty}\right)\right]^2 \lesssim \sum_{n=1}^\infty \e^{-cn^2 t} (n^2|x-x'|^2 \wedge 1).
\end{align*}
It follows that
\begin{align*}
	&\textstyle\int_0^{t} \d s \int_0^L \d y\, [G_{s}(x',y) - G_{s}(x\,,y)]^2\lesssim \int_0^{t} \d s \int_0^\infty \d z \, (|x-x'|^2 z^2 \wedge 1) \, \e^{-cz^2 s} \\
	&\textstyle\le \int_0^\infty \d z \, (|x-x'|^2 z^2 \wedge 1) \int_0^\infty \d s\, \e^{-cz^2 s}
	\lesssim \int_0^\infty \d z\, (|x-x'|^2 \wedge z^{-2})\\
	&\textstyle\le \int_0^{|x-x'|^{-1}} |x-x'|^2 \d z + \int_{|x-x'|^{-1}}^\infty z^{-2} \d z
	\lesssim |x-x'|.
\end{align*}
This completes the proof.
\end{proof}

\begin{lemma}\label{lem:int:G:t}
Fix $T>0$.
There exists a constant $c>0$ such that
\begin{enumerate}
\item $\int_0^L [G_{t'}(x\,,y) - G_t(x\,,y)]^2\, \d y \lesssim 
	\sum_{n=1}^\infty (|t'-t|^2 n^4 \wedge 1)\, \e^{-cn^2 t}$,
\item $\int_0^t \d s \int_0^L \d y [G_{t'-s}(x\,,y) - G_{t-s}(x\,,y)]^2 \lesssim (t'-t)^{1/2}$,
\item $\int_t^{t'} \d s \int_0^L \d y \, [G_{t'-s}(x\,,y)]^2
	\lesssim (t'-t)^{1/2}$
\end{enumerate}
uniformly for all $0<t<t' \le T$ and $x\in [0\,,L]$.
\end{lemma}

\begin{proof}
Thanks to Parseval's identity, \eqref{f_n:bd}, the elementary inequality $\e^{-a}-\e^{-b} \le \e^{-a}((b-a) \wedge 1)$ for all $0<a<b$, and property \eqref{lambda}, we obtain:
\begin{align*}
	&\textstyle\int_0^L [G_{t'}(x\,,y) - G_t(x\,,y)]^2\, \d y 
	= \sum_{n=1}^\infty (\e^{-\lambda_n t'} - \e^{-\lambda_n t})^2 |f_n(x)|^2\\
	&\quad\textstyle\lesssim \sum_{n=1}^\infty \e^{-2\lambda_n t}(|t'-t|^2 \lambda_n^2 \wedge 1)
	\lesssim \sum_{n=1}^\infty \e^{-cn^2 t}(|t'-t|^2 n^4 \wedge 1).
\end{align*}
We use the preceding to continue the computation:
\begin{align*}
	&\textstyle\int_0^t \d s \int_0^L \d y [G_{t'-s}(x\,,y) - G_{t-s}(x\,,y)]^2
	\lesssim \int_0^t \d s \int_0^\infty \d z \, ((t'-t)^2 z^4 \wedge 1) \, \e^{-cz^2 s}\\
	&\quad\textstyle\lesssim \int_0^\infty \d z \, ((t'-t)^2 z^4 \wedge 1) \int_0^t \d s \, \e^{-cz^2 s}
	\lesssim \int_0^\infty \d z\, ((t'-t)^2 z^2 \wedge z^{-2})\\
	&\quad\textstyle \lesssim \int_0^{(t'-t)^{-1/2}} (t'-t)^2 z^2\, \d z + \int_{(t'-t)^{-1/2}}^\infty z^{-2}\, \d z
	\lesssim (t'-t)^{1/2}.
\end{align*}
Finally, we may use Parseval's identity, \eqref{f_n:bd}, and the inequality $1-\e^{-x} \le 1 \wedge x$ for all $x\ge 0$ to deduce the last estimate:
\begin{align*}
	&\textstyle\int_t^{t'} \d s \int_0^L \d y \, [G_{t'-s}(x\,,y)]^2
	\lesssim  \int_t^{t'} \d s \sum_{n=1}^\infty \e^{-\lambda_n (t'-s)}|f_n(x)|^2\\
	&\quad\textstyle\lesssim \int_0^\infty \d z \int_t^{t'} \d s \,\e^{-cz^2 (t'-s)}
	\lesssim \int_0^\infty \d z \, z^{-2} (1-\e^{-cz^2 (t'-t)})\\
	&\quad\textstyle \lesssim \int_0^{(t'-t)^{-1/2}} (t'-t)\, \d z + \int_{(t'-t)^{-1/2}}^\infty z^{-2}\, \d z
	\lesssim (t'-t)^{1/2}.
\end{align*}
This completes the proof.
\end{proof}

\begin{lemma}\label{lem:w-w}
For any $T>0$, there exists $C_0>0$ such that
\begin{align}\label{var:w}
	&\Var(w(t\,,x)) \le C_0 \sqrt{t} \quad \text{and}\\
	&\Var(w(t',x')-w(t\,,x)) \le C_0 \left[ \rho^2((t\,,x)\,,(t',x'))\wedge \sqrt{t\vee t'}\right]
	\label{var:w-w}
\end{align}
uniformly for all $t\,,t'\in [0\,,T]$ and $x\,,x' \in [0\,,L]$.
\end{lemma}

\begin{proof}
Wiener isometry and Lemma \ref{lem:int:G} yield \eqref{var:w}.
Next, by Lemmas \ref{lem:int:G:x} and \ref{lem:int:G:t}, there exists $c_1>0$ such that for all $t\,,t'\in [0\,,T]$ and $x\,,x' \in [0\,,L]$, 
\begin{align}\label{w-w}
	\Var(w(t',x')-w(t\,,x)) \le c_1 \rho^2((t\,,x)\,,(t',x')).
\end{align}
Since $\Var(w(t',x')-w(t\,,x)) \le 2\Var(w(t',x')) + 2 \Var(w(t\,,x))$,
we may use \eqref{var:w} to finish the proof.
\end{proof}

\begin{lemma}\label{lem:w-w:DN}
Under \eqref{D:BC} or \eqref{N:BC}, for any $T>0$, there exists $C>0$ such that
\begin{align}
	&\Var(w(t\,,x)) \le C\, (\sqrt{t} \wedge f_1(x)) \quad \text{and}\\
	&\Var(w(t',x')-w(t\,,x)) \le C \left[ \rho((t\,,x)\,,(t',x'))\wedge \sqrt{t\vee t'} \wedge (f_1(x)\vee f_1(x'))\right]
	\label{w-w:DN}
\end{align}
uniformly for all $t\,,t'\in [0\,,T]$ and $x\,,x' \in [0\,,L]$.
\end{lemma}

\begin{proof}
Thanks to Lemma \ref{lem:w-w}, there is nothing to prove under \eqref{N:BC} since $f_1$ is constant; see \eqref{eigen:N}.
It remains to prove that $\Var(w(t\,,x)) \lesssim f_1(x)$ under \eqref{D:BC}. Indeed, by Wiener isometry and Parseval's identity,
\begin{align*}
	\Var(w(t\,,x)) &\textstyle= \int_0^t \d s\int_0^L \d y \, G_s^2(x\,,y)\\
	&\textstyle= \int_0^t \d s  \sum_{n=1}^\infty \e^{-2\lambda_n s} |f_n(x)|^2
	= \sum_{n=1}^\infty (2\lambda_n)^{-1}(1-\e^{-2\lambda_n t}) |f_n(x)|^2.
\end{align*}
Using \eqref{eigen:D} and $|\sin(a)| \le a $ for $a \ge 0$, we deduce that
\begin{align*}
	\textstyle\Var(w(t\,,x)) \lesssim \sum_{n=1}^\infty n^{-2}\sin^2(\pi nx/L)
	\lesssim \sum_{1 \le n \le L/(\pi x)} x^2 + \sum_{n \ge L/(\pi x)} n^{-2} \lesssim x.
\end{align*}
By symmetry, $\Var(w(t\,,x)) \lesssim L-x$. Use $f_1(x) \asymp x\wedge (L-x)$ to finish the proof.
\end{proof}

\subsection{Strong local non-determinism}

In this part, we prove that the Gaussian random field $w$ which solves \eqref{she:add} is strongly locally non-deterministic (SLND).

We start with conditions \eqref{D:BC} and \eqref{N:BC}.
Let us first recall that the Fourier transform of a function $f:\R^d \to \R$ is defined by $\hat{f}(\zeta) = \int_{\R^d} \e^{-i\zeta\cdot x} f(x) \d x$ for $\zeta \in \R^d$, 
and the inverse Fourier transform of $g:\R^d \to \R$ is $\check{g}(x) = (2\pi)^{-d} \int_{\R^d} \e^{i\zeta\cdot x} g(\zeta) \d \zeta$ for $x \in \R^d$.
We identify the torus as $\T \cong [-\pi\,,\pi]$. 
The Fourier transform of a function $\Phi: \T \to \R$ is defined by $\hat{\Phi}(n) = \int_{-\pi}^\pi \e^{-in\theta} \Phi(\theta) \d \theta$ for $n \in \Z$, 
and the inverse Fourier transform of $\Psi:\Z \to \R$ is $\check{\Psi}(\theta) = (2\pi)^{-1}\sum_{n\in\Z}\e^{in\theta} \Psi(n)$ for $\theta \in \T$.


\begin{lemma}\label{lem:SLND:DN}
Fix $T>0$. Then, under \eqref{D:BC} or \eqref{N:BC},
\begin{align*}
	\mathrm{Var}(w(t\,,x) \mid w(t_1\,,x_1)\,,\dots,w(t_m\,,x_m))
	\asymp \min_{1\le j \le m}\rho^2((t\,,x)\,,(t_j\,,x_j)) \wedge \sqrt{t} \wedge f_1(x)
\end{align*}
where $f_1$ is the principal eigenfunction under \eqref{D:BC} or \eqref{N:BC}, respectively, given in Lemma \ref{lem:eigen}, and the implied constants
do not depend on $m \in \N_+$ nor $(t\,,x)\,,(t_1\,,x_1)\,,$
$\dots,(t_m\,,x_m) \in [0\,,T] \times [0\,,L]$.
\end{lemma}

\begin{proof}
The upper bound follows from Lemma \ref{lem:w-w:DN} and the fact that
\[\textstyle
	\Var(X|X_1\,,\dots,X_m) = \inf\limits_{a_1,\dots,a_m\in \R} \E\left[ (X-\sum_{j=1}^m a_j X_j)^2 \right]
\]
for any centered Gaussian vector $(X\,,X_1\,,\dots,X_m)$.
To prove the lower bound, it suffices to show the existence of $C>0$ such that
\begin{align*}
	\textstyle{\E\left[ \left( w(t\,,x) - \sum_{j=1}^m a_j w(t_j\,,x_j) \right)^2 \right]} \ge C \min\limits_{1\le j \le m}\rho^2((t\,,x)\,,(t_j\,,x_j)) \wedge \sqrt{t} \wedge f_1(x)
\end{align*}
uniformly for all $m\in\N_+$, for all $(t\,,x)\,,(t_1\,,x_1)\,,\dots,(t_m\,,x_m) \in [0\,,T] \times [0\,,L]$, and for all $a_1\,,\dots, a_m\in \R$. To this end, we first use \eqref{w}, Wiener isometry, and \eqref{G} to write
\begin{align*}
	&\textstyle\E\left[ \left( w(t\,,x) - \sum_{j=1}^m a_j w(t_j\,,x_j) \right)^2 \right]\\
	&\textstyle=\int_{-\infty}^\infty \d s \int_0^L \d y \left[ G_{t-s}(x\,,y) \1_{[0,t]}(s) - \sum_{j=1}^m a_j G_{t_j-s}(x_j\,,y) \1_{[0,t_j]}(s) \right]^2\\
	&\textstyle=\sum_{n=1}^\infty \int_{-\infty}^\infty \d s \left[ \e^{-\lambda_n(t-s)}f_n(x) \1_{[0,t]}(s) - \sum_{j=1}^m a_j \e^{-\lambda_n(t_j-s)} f_n(x_j) \1_{[0,t_j]}(s) \right]^2\\
	&\textstyle=\frac{1}{2\pi} \sum_{n=1}^\infty \int_{-\infty}^\infty \frac{\d \tau}{\lambda_n^2+\tau^2} \left| (\e^{-i\tau t} - \e^{-\lambda_n t})f_n(x) - \sum_{j=1}^m a_j (\e^{-i\tau t_j} - \e^{-\lambda_n t_j}) f_n(x_j) \right|^2,
\end{align*}
where the last equality follows from Plancherel's theorem and the simple fact that the Fourier transform of $s\mapsto \e^{-\lambda_n(t-s)}\1_{[0,t]}(s)$ is $\tau \mapsto (\e^{-i\tau t} - \e^{-\lambda_n t})/(\lambda_n - i\tau)$.

{\bf Case 1:} Neumann boundary condition \eqref{N:BC}.
By \eqref{eigen:N},
\begin{align*}
	&\textstyle\E\left[ \left( w(t\,,x) - \sum_{j=1}^m a_j w(t_j\,,x_j) \right)^2 \right]\\
	&\textstyle=\frac{1}{4\pi L} \sum_{n\in \Z} \int_{-\infty}^\infty \frac{\d \tau}{\lambda_n^2+\tau^2} \Big| (\e^{-i\tau t} - \e^{-\lambda_n t})(\e^{in\pi x/L} + \e^{-in\pi x/L})\\
	&\textstyle \hskip1in  - \sum_{j=1}^m a_j (\e^{-i\tau t_j} - \e^{-\lambda_n t_j}) (\e^{in\pi x_j/L} + \e^{-in\pi x_j/L}) \Big|^2.
\end{align*}
Let $\phi:\R\to\R$ and $\psi:\R\to\R$  be two smooth, nonnegative functions with $\supp{\phi}=[-\pi/2\,,\pi/2]$, $\supp{\psi}=[-T/2\,,T/2]$ and $\phi(0)=\psi(0)=1$.
For any $r \in (0\,,1]$, define $\phi_r:\R\to\R$ by $\phi_r(x) = r^{-1}\phi(r^{-1}x)$ and $\psi_r:\R\to\R$ the same way.
Define $\Phi_r:\T\to\R$ as the restriction of $\phi_r$, i.e., $\Phi_r(\theta) = \phi_r(\theta)$ for $\theta \in (-\pi\,,\pi] \cong \T$.
Let
\begin{align}\label{eps:SLND:N}
	\textstyle
	\varepsilon = \min\limits_{1 \le j \le m}\left( \sqrt{\frac{|t-t_j|}{T}} \vee \frac{|x-x_j|}{L} \right) \wedge \sqrt{\frac{t}{T}}.
\end{align}
Note that $\varepsilon \in [0\,,1]$. If $\varepsilon = 0$, there is nothing to prove, so we may assume that $\varepsilon \in (0\,,1]$.
Define
\begin{align*}\textstyle
	I:= &\textstyle\sum_{n\in\Z} \int_{-\infty}^\infty \d \tau \Big[ (\e^{-i\tau t} - \e^{-\lambda_n t})(\e^{in\pi x/L} + \e^{-in\pi x/L})\\
	&\textstyle  - \sum_{j=1}^m a_j (\e^{-i\tau t_j} - \e^{-\lambda_n t_j}) (\e^{in\pi x_j/L} + \e^{-in\pi x_j/L})\Big] \e^{-in\pi x/L} \e^{i\tau t} \hat{\Phi}_\varepsilon(n) \hat{\psi}_{\varepsilon^2}(\tau).
\end{align*}
By Fourier inversion,
\begin{align*}
	I &\textstyle= 2\pi \sum_{n \in \Z} \Big[ (\psi_{\varepsilon^2}(0) - \e^{-\lambda_n t} \psi_{\varepsilon^2}(t))(1 + \e^{-2i n\pi x/L})\\
	&\textstyle \quad - \sum_{j=1}^m a_j (\psi_{\varepsilon^2}(t-t_j) - \e^{-\lambda_n t_j} \psi_{\varepsilon^2}(t))(\e^{in\pi (x_j-x)/L} - \e^{-in\pi(x_j+x)/L}) \Big] \hat{\Phi}_\varepsilon(n)\\
	&\textstyle= 4\pi^2 \Big[ (\psi_{\varepsilon^2}(0)-\e^{-\lambda_n t} \psi_{\varepsilon^2}(t))(\Phi_\varepsilon(0)+\Phi_\varepsilon(-\tfrac{2\pi x}{L})) \\
	&\textstyle\quad- \sum_{j=1}^m a_j  (\psi_{\varepsilon^2}(t-t_j)-\e^{-\lambda_n t_j} \psi_{\varepsilon^2}(t)) (\Phi_\varepsilon(\tfrac{\pi (x_j-x)}{L})+\Phi_\varepsilon(-\tfrac{\pi (x_j+x)}{L}))\Big].
\end{align*}
Note that $\psi_{\varepsilon^2}(0)=\varepsilon^{-2}$, $\Phi_\varepsilon(0) = \phi_\varepsilon(0)=\varepsilon^{-1}$.
Observe from the definition of $\varepsilon$ in \eqref{eps:SLND:N} that $\varepsilon^{-2}t \ge T$, which implies $\psi_{\varepsilon^2}(t) = 0$ since $\supp{\psi}=[-T/2\,,T/2]$.
Similarly, owing to \eqref{eps:SLND:N}, for each $j \in \{1\,,\dots,m\}$, we have $\varepsilon \le \sqrt{|t-t_j|/T}$ or $\varepsilon \le |x-x_j|/L$, which implies that at least one of $\psi_{\varepsilon^2}(t-t_j)$ or $\Phi_\varepsilon(\pi(x_j-x)/L)$ is 0, hence $\psi_{\varepsilon^2}(t-t_j)\Phi_\varepsilon(\pi(x_j-x)/L)=0$.
Since $\phi\ge 0$, we have $\Phi_\varepsilon(-2\pi x/L) \ge 0$.
Moreover, we observe that $\Phi_\varepsilon(-\pi(x_j+x)/L) = 0$.
Indeed, by the definition of $\Phi_\varepsilon$,
\begin{align*}
\textstyle
	\Phi_\varepsilon(-\frac{\pi (x_j+x)}{L}) = \begin{cases}
		\phi_\varepsilon(-\frac{\pi(x_j+x)}{L})
		 & \text{if } x_j+x \in [0\,,L],\\
		\phi_\varepsilon(\frac{\pi(L-x_j + L-x)}{L}) 
		& \text{if } x_j+x \in (L\,,2L].
	\end{cases}
\end{align*}
Since $x_j+x = |x_j-x| + 2(x_j\wedge x)$ and $L-x_j + L-x = |x_j-x| + 2(L-(x_j \vee x))$ are at least $\min_{1\le j \le m}|x_j-x|$, this and \eqref{eps:SLND:N} imply that $\psi_{\varepsilon^2}(t-t_j) \Phi_\varepsilon(-\pi(x_j+x)/L) = 0$.
The above observations imply that $I \ge 4\pi^2 \varepsilon^{-3}$.
Therefore, by Cauchy-Schwarz inequality,
\begin{align*}\textstyle
	\varepsilon^{-6}\lesssim |I|^2
	\lesssim \E\left[ \left( w(t\,,x) - \sum_{j=1}^m a_j w(t_j\,,x_j) \right)^2 \right] \times J,
\end{align*}
where 
\[\textstyle
	J:=\sum_{n\in\Z} \int_{-\infty}^\infty (\lambda_n^2 + \tau^2) |\hat{\Phi}_\varepsilon(n) \hat{\psi}_{\varepsilon^2}(\tau)|^2 \d \tau.
\]
By $\hat{\psi}_{\varepsilon^2}(\tau) = \hat{\psi}(\varepsilon^2 \tau)$, $\hat{\Phi}_\varepsilon(n) = \hat{\phi}_\varepsilon(n) = \hat{\phi}(\varepsilon n)$, and by \eqref{lambda},
\begin{align*}\textstyle
	J \lesssim \int_0^\infty \d z \int_{-\infty}^\infty \d \tau \, (z^4 + \tau^2) |\hat{\phi}(\varepsilon z) \hat{\psi}(\varepsilon^2 \tau)|^2 \propto \varepsilon^{-7},
\end{align*}
where the last relation is due to scaling, and the proportionality constant is finite since $\hat{\phi}$ and $\hat{\psi}$ are rapidly decreasing functions.
It follows that
\begin{align*}\textstyle
	\E\left[ \left( w(t\,,x) - \sum_{j=1}^m a_j w(t_j\,,x_j) \right)^2 \right]\gtrsim \varepsilon,
\end{align*}
which yields the desired lower bound since $f_1$ is constant; see \eqref{eigen:N}.

{\bf Case 2:} Dirichlet boundary condition \eqref{D:BC}.
By \eqref{eigen:D},
\begin{align*}
	&\textstyle\E\left[ \left( w(t\,,x) - \sum_{j=1}^m a_j w(t_j\,,x_j) \right)^2 \right]\\
	&\textstyle=\frac{1}{4\pi L} \sum_{n\in \Z} \int_{-\infty}^\infty \frac{\d \tau}{\lambda_n^2+\tau^2} \Big| (\e^{-i\tau t} - \e^{-\lambda_n t})(\e^{in\pi x/L} - \e^{-in\pi x/L})\\
	&\textstyle \hskip1in - \sum_{j=1}^m a_j (\e^{-i\tau t_j} - \e^{-\lambda_n t_j}) (\e^{in\pi x_j/L} - \e^{-in\pi x_j/L}) \Big|^2.
\end{align*}
Note that $f_1(x) \asymp x(L-x)/L^2$. We let
\begin{align}\label{eps:SLND:D}
	\textstyle
	\varepsilon = \min\limits_{1 \le j \le m}\left( \sqrt{\frac{|t-t_j|}{T}} \vee \frac{|x-x_j|}{L} \right) \wedge \sqrt{\frac{t}{T}} \wedge \frac{x(L-x)}{L^2}.
\end{align}
Note that $\varepsilon \in [0\,,1]$.
Without loss of generality, assume $\varepsilon>0$.
Define $\psi\,,\phi\,,\Phi$ and their scaled versions $\psi_r\,,\phi_r\,,\Phi_r$ as in {\bf Case 1}.
Define
\begin{align*}\textstyle
	I:= &\textstyle\sum_{n\in\Z} \int_{-\infty}^\infty \d \tau \Big[ (\e^{-i\tau t} - \e^{-\lambda_n t})(\e^{in\pi x/L} - \e^{-in\pi x/L})\\
	&\textstyle - \sum_{j=1}^m a_j (\e^{-i\tau t_j} - \e^{-\lambda_n t_j}) (\e^{in\pi x_j/L} - \e^{-in\pi x_j/L})\Big] \e^{-in\pi x/L} \e^{i\tau t} \hat{\Phi}_\varepsilon(n) \hat{\psi}_{\varepsilon^2}(\tau).
\end{align*}
By Fourier inversion,
\begin{align*}
	I &\textstyle= 4\pi^2 \Big[ (\psi_{\varepsilon^2}(0)-\e^{-\lambda_n t} \psi_{\varepsilon^2}(t))(\Phi_\varepsilon(0)-\Phi_\varepsilon(-\tfrac{2\pi x}{L})) \\
	&\textstyle\quad- \sum_{j=1}^m a_j  (\psi_{\varepsilon^2}(t-t_j)-\e^{-\lambda_n t_j} \psi_{\varepsilon^2}(t))(\Phi_\varepsilon(\tfrac{\pi (x_j-x)}{L})-\Phi_\varepsilon(-\tfrac{\pi (x_j+x)}{L})) \Big].
\end{align*}
Again, $\psi_{\varepsilon^2}(0)=\varepsilon^{-2}$, $\Phi_\varepsilon(0)=\varepsilon^{-1}$, $\psi_{\varepsilon^2}(t)=0$ and $\psi_{\varepsilon^2}(t-t_j) \Phi_\varepsilon(\pi(x_j-x)/L) = 0$ by the definition of $\varepsilon$ in \eqref{eps:SLND:D}.
By the definition of $\Phi_\varepsilon$,
\begin{align*}\textstyle
	\Phi_\varepsilon(-\frac{2\pi x}{L}) = \begin{cases}
		\phi_\varepsilon(-\frac{2\pi x}{L}) & \text{if } x \in [0\,,L/2),\\
		\phi_\varepsilon(\frac{2\pi(L-x)}{L}) & \text{if } x \in [L/2\,,L].
	\end{cases}
\end{align*}
In either case, we may use $\varepsilon \le x(L-x)/L^2$ and $\supp{\phi}=[-\pi/2\,,\pi/2]$ to deduce that $\Phi_\varepsilon(-\frac{2\pi x}{L}) = 0$.
Moreover, as in {\bf Case 1}, we have
$\psi_{\varepsilon^2}(t-t_j) \Phi_\varepsilon(-\pi(x_j+x)/L) = 0$.
It follows that $I = 4\pi^2 \varepsilon^{-3}$. The rest of the proof is the same as in {\bf Case 1}.
\end{proof}

We turn to the SLND property under \eqref{R:BC}.
The proof requires the lemma below.

\begin{lemma}\label{lem:evaluation}
Let $\{f_n\}_{n\in\N_+}$ be the orthonormal basis of eigenfunctions given by Lemma \ref{lem:eigen}.
If $\phi \in C^2[0\,,L]$, then the following holds in the sense of pointwise convergence:
\begin{align}\label{phi:sum}
	\phi(x) = \sum_{n=1}^\infty \langle \phi\,,f_n \rangle f_n(x)\quad \text{for all $x\in[0\,,L]$.}
\end{align}
\end{lemma}

\begin{proof}
This is standard. For completeness, we give a short proof.
Since $\phi \in C^2[0,L]$, we may integrate by parts twice and use the boundary condition \eqref{R:BC} for $f_n$ to deduce that
\begin{align*}
	\langle \phi\,,f_n'' \rangle 
	&\textstyle = \phi(L)f_n'(L)-\phi(0)f_n'(0) - \phi'(L)f_n(L) + \phi'(0)f_n(0) + \langle \phi'', f_n \rangle\\
	&\textstyle = -\beta \phi(L) f_n(L)+ \alpha\phi(0) f_n(0) - \phi'(L)f_n(L) + \phi'(0)f_n(0) + \langle \phi'', f_n \rangle.
\end{align*}
Then, it follows from \eqref{eta} and \eqref{f_n:bd} that
\begin{align}\label{sum:<phi,f>}\textstyle
	\sum_{n=1}^\infty |\langle \phi\,,f_n \rangle| \lesssim C+ \sum_{n=n_0+1}^\infty (2\lambda_n)^{-1} |\langle \phi\,,f_n'' \rangle| \lesssim \sum_{n=1}^\infty n^{-2} < \infty.
\end{align}
From Lemma \ref{lem:eigen}, we see that for each $N \in \N_+$, $S_N := \sum_{n=1}^N\langle \phi\,,f_n \rangle f_n$ is continuous on $[0\,,L]$, which converges uniformly to $S_\infty := \sum_{n=1}^\infty\langle \phi\,,f_n \rangle f_n$ because \eqref{sum:<phi,f>} and \eqref{f_n:bd} imply that for $M>N$,
\begin{align*}\textstyle
	\sup_{x\in [0,L]}|S_M(x)-S_N(x)| \le \sum_{N<n\le M} |\langle \phi\,,f_n\rangle| \sup_{n \in \N_+, x \in [0,L]}|f_n(x)| \to 0
\end{align*}
as $M\,,N\to\infty$.
This shows that $S_N$ converges pointwise to the limit $\sum_{n=1}^\infty\langle \phi\,,f_n \rangle f_n$ which is also continuous, but $S_N$ also converges to the limit $\phi$ in $L^2$ since $\{f_n\}_{n\in\N_+}$ is an orthonormal basis.
Hence, both limits must agree. This and continuity of $\phi$ ensure the pointwise convergence in \eqref{phi:sum}.
\end{proof}

\begin{lemma}\label{lem:SLND:R}
Fix $T>0$. Then, under \eqref{R:BC},
\begin{align*}
	\mathrm{Var}(w(t\,,x) \mid w(t_1\,,x_1)\,,\dots,w(t_m\,,x_m))
	\asymp \min_{1\le j \le m}\rho^2((t\,,x)\,,(t_j\,,x_j)) \wedge \sqrt{t},
\end{align*}
where the implied constants do not depend on
$m \in \N_+$ nor $(t\,,x)\,,(t_1\,,x_1)\,,\dots,(t_m\,,x_m) \in [0\,,T] \times [0\,,L]$.
\end{lemma}

\begin{proof}
The upper bound follows from Lemma \ref{lem:w-w}.
To prove the lower bound, it suffices to prove the existence of $C=C(T\,,L)>0$ such that
\begin{align}\label{SLND:R:goal}
	\textstyle{\E\left[ \left( w(t\,,x) - \sum_{j=1}^m a_j w(t_j\,,x_j) \right)^2 \right]} \ge C \min\limits_{1\le j \le m}\rho^2((t\,,x)\,,(t_j\,,x_j)) \wedge \sqrt{t}
\end{align}
uniformly for all $m\in\N_+$, for all $(t\,,x)\,,(t_1\,,x_1)\,,\dots,(t_m\,,x_m) \in [0\,,T] \times [0\,,L]$, and for all $a_1\,,\dots, a_m\in \R$. 
As in the proof of Lemma \ref{lem:SLND:DN}, we first write
\begin{align*}
	&\textstyle\E\left[ \left( w(t\,,x) - \sum_{j=1}^m a_j w(t_j\,,x_j) \right)^2 \right]\\
	&\textstyle=\frac{1}{2\pi} \sum_{n=1}^\infty \int_{-\infty}^\infty \frac{\d \tau}{\lambda_n^2+\tau^2} \left| (\e^{-i\tau t} - \e^{-\lambda_n t})f_n(x) - \sum_{j=1}^m a_j (\e^{-i\tau t_j} - \e^{-\lambda_n t_j}) f_n(x_j) \right|^2.
\end{align*}
Choose and fix two smooth nonnegative functions $\phi:\R\to\R$ and $\psi: \R\to\R$ with $\supp{\phi} = [-T/2\,,T/2]$, $\supp{\psi}=[-1/2\,,1/2]$, and $\phi(0)=\psi(0)=1$.
For every $r\in(0\,,1]$ and $x \in [0\,,L]$, define $\phi_r$ and $\psi_{x,r}$ by
\begin{align*}
	\phi_r(\tau) = r^{-1}\phi(r^{-1}\tau) \quad \text{and} \quad
	\psi_{x,r}(y) = r^{-1}\psi(r^{-1}(y-x)).
\end{align*}
Also, choose and fix a smooth, nonnegative, even function $\chi:\R \to \R$ such that
$\chi \equiv 1$ on $[0\,,1/2]$, $0 \le \chi \le 1$ on $[1/2\,,3/4]$, and $\chi \equiv 0$ on $[3/4\,,\infty)$.
Set
\begin{align}\label{eps:SLND:R}\textstyle
	\varepsilon :=
		\left(1 \wedge \tfrac{L}{2} \wedge \tfrac{1}{2a}\right)\left[\min\limits_{1\le j\le m}
		\left( \sqrt{\frac{|t-t_j|}{T}}\vee \frac{|x-x_j|}{L}\right)
		\wedge \sqrt{\frac{t}{T}}\,\right],
\end{align}
where $a:=|\alpha|\vee|\beta|$. 
We may assume that $a > 0$, because $a = 0$ corresponds to the Neumann case, which has already been proved in Lemma \ref{lem:SLND:DN}.
We may also assume that $\varepsilon>0$, for otherwise \eqref{SLND:R:goal} holds trivially with right-hand side being 0.
Note that $\supp{\phi_{\varepsilon^2}}=[-\varepsilon^2T/2\,,\varepsilon^2T/2]$ and $\supp{\psi_{x,\varepsilon}} = [x-\varepsilon/2\,,x+\varepsilon/2]$. 
Let us define the test function $\eta_{x,\varepsilon}$ according to the following three cases:\smallskip

\noindent {\bf Case 1:}
If $x\in[\varepsilon/2\,,L-\varepsilon/2]$, define
\begin{align*}\textstyle
	\eta_{x,\varepsilon}(y):=\psi_{x,\varepsilon}(y),
	\qquad y\in[0\,,L].
\end{align*}

\noindent
{\bf Case 2:}
If $x\in[0\,,\varepsilon/2)$, define
\[\textstyle
\eta_{x,\varepsilon}(y):=
\varepsilon^{-1}\chi\!\left(\frac{y}{\varepsilon}\right)\frac{1-\alpha y}{1-\alpha x},
\qquad y\in[0\,,L].
\]

\noindent
{\bf Case 3:}
If $x\in(L-\varepsilon/2\,,L]$, define
\[\textstyle
\eta_{x,\varepsilon}(y):=
\varepsilon^{-1}\chi\!\left(\frac{L-y}{\varepsilon}\right)\frac{1+\beta(L-y)}{1+\beta(L-x)},
\qquad y\in[0\,,L].
\]
Next, we are going to make three claims about some key properties of the test functions.\smallskip

{\it Claim 1.} In all three cases, $\eta_{x,\varepsilon} \in C^\infty([0\,,L])$ and satisfies Robin condition \eqref{R:BC}, i.e.,
\begin{align*}
\eta_{x,\varepsilon}'(0)+\alpha\eta_{x,\varepsilon}(0)=0,
\qquad
\eta_{x,\varepsilon}'(L)+\beta\eta_{x,\varepsilon}(L)=0.
\end{align*}
{\it Proof of Claim 1.} In {\bf Case 1}, $\eta_{x,\varepsilon}=\psi_{x,\varepsilon}$ is compactly supported in $(0\,,L)$, so the claim clearly holds. In {\bf Case 2}, the definition of $\varepsilon$ in \eqref{eps:SLND:R} ensures that $x < \varepsilon/2 \le 1/(4a)$, which implies $|1-\alpha x|\ge 1-|\alpha|x \ge 3/4 > 0$, so $g_x(y):=\frac{1-\alpha y}{1-\alpha x}$ is well defined and smooth, thus $\eta_{x,\varepsilon} \in C^\infty([0\,,L])$.
Since $g_x(y)$ satisfies Robin condition \eqref{R:BC} at $y=0$ (which can be verified easily) and since $\chi \equiv 1$ on $[0\,,1/2]$, $\eta_{x,\varepsilon}(y)$ also satisfies Robin condition \eqref{R:BC} at $y=0$.
On the other hand, the definitions of $\chi$ and $\varepsilon$ imply that $\eta_{x,\varepsilon}(y) \equiv 0$ near $y=L$, so it also satisfies Robin condition \eqref{R:BC} at $y=L$.
{\bf Case 3} can be shown in a similar way. Hence {\it Claim 1} follows.\smallskip

{\it Claim 2.} We have
\begin{align}\label{test2:SLND:R}
\eta_{x,\varepsilon}(x)=\varepsilon^{-1},
\quad
\|\eta_{x,\varepsilon}\|_{L^2([0,L])}^2 \le C\varepsilon^{-1},
\quad
\|\eta_{x,\varepsilon}''\|_{L^2([0,L])}^2 \le C\varepsilon^{-5},
\end{align}
where $C\in(0\,,\infty)$ is a constant that depends only on $\alpha, \beta, L$ and the test functions $\psi, \chi$, but is independent of $(x\,,\varepsilon)$.\\
{\it Proof of Claim 2.}
In {\bf Case 1}, the first identity is obvious; the other two estimates follow readily from scaling:
\[\textstyle
	\|\eta_{x,\varepsilon}\|^2_{L^2([0,L])} \le \varepsilon^{-1} \int_{-\infty}^\infty |\psi(y)|^2 \, \d y, \qquad \|\eta_{x,\varepsilon}''\|^2_{L^2([0,L])} \le \varepsilon^{-5} \int_{-\infty}^\infty |\psi''(y)|^2 \, \d y.
\]
In {\bf Case 2}, $x<\varepsilon/2$ together with the fact that $\chi \equiv 1$ on $[0\,,1/2]$ implies $\chi(\varepsilon^{-1}x)=1$, hence $\eta_{x,\varepsilon}(x)=\varepsilon^{-1}$. 
Next, since $\sup_{x \in [0,1/(4\alpha)]} \sup_{y \in [0,L]}(|g_x(y)|+|g_x'(y)|) < \infty$ and $g_x'' \equiv 0$, it follows that for $0<x<\varepsilon/2$,
\[
	|\eta_{x,\varepsilon}(y)| \lesssim \varepsilon^{-1} \chi(\varepsilon^{-1}y)\quad \text{and} \quad
	|\eta_{x,\varepsilon}''(y)|
	\lesssim
	\varepsilon^{-3}|\chi''(\varepsilon^{-1}y)|
	+
	\varepsilon^{-2}|\chi'(\varepsilon^{-1}y)|.
\]
These imply, respectively, that 
\begin{align*}
	&\textstyle\|\eta_{x,\varepsilon}\|_{L^2([0,L])}^2 \lesssim \varepsilon^{-1} \int_{-\infty}^\infty |\chi(y)|^2 \, \d y \quad \text{and}\\
	&\textstyle\|\eta_{x,\varepsilon}''\|_{L^2([0,L])}^2 \lesssim
	\varepsilon^{-5}\int_{-\infty}^\infty |\chi''(y)|^2\,\d y
	+
	\varepsilon^{-3}\int_{-\infty}^\infty |\chi'(y)|^2\,\d y
	\lesssim \varepsilon^{-5},
\end{align*}
where the implicit constants depend only on $\alpha, L, \chi$.
The proof for {\bf Case 3} is similar, with implicit constants depending only on $\beta, L, \chi$. This proves {\it Claim 2}.\smallskip

{\it Claim 3.} In all three cases,
\begin{align}\label{test3:SLND:R}
	\phi_{\varepsilon^2}(t-t_j)\eta_{x,\varepsilon}(x_j)=0 \quad \text{for $j=1,\dots,m$.}
\end{align}
{\it Proof of Claim 3.} If $\phi_{\varepsilon^2}(t-t_j) = 0$, there is nothing to prove. If $\phi_{\varepsilon^2}(t-t_j)\neq 0$, then $\sqrt{|t-t_j|/T}<\varepsilon$. 
This, together with \eqref{eps:SLND:R}, implies that, for each $j = 1,\dots, m$,
\[\textstyle
	\varepsilon
	\le
	\left( 1\wedge \tfrac{L}{2} \wedge \tfrac{1}{2a} \right)
	\left( \sqrt{\frac{|t-t_j|}{T}} \vee \frac{|x-x_j|}{L} \right),
\]
and hence
\begin{align}\label{|x-xj|>eps}
	\textstyle
	|x-x_j|
	\ge
	\left( 1\wedge \tfrac{L}{2} \wedge \tfrac{1}{2a} \right)^{-1} L\varepsilon
	\ge
	2\varepsilon.
\end{align}
In {\bf Case 1}, we have $\supp\eta_{x,\varepsilon}\subset [ x- \varepsilon / 2\,,x+\varepsilon / 2 ]$, hence \eqref{|x-xj|>eps} implies that $x_j\notin \supp\eta_{x,\varepsilon}$.
In {\bf Case 2}, we have $x<\varepsilon/2$ and $\supp\eta_{x,\varepsilon}\subset [0\,,3\varepsilon / 4]$. But \eqref{|x-xj|>eps} implies that $x_j \ge x+2\varepsilon > 3\varepsilon/4$, hence $x_j \notin \supp \eta_{x,\varepsilon}$. 
The same is true for {\bf Case 3}. In all three cases, we have $\eta_{x,\varepsilon}(x_j) = 0$.
	This proves {\it Claim 3}.\smallskip
	
With these test functions and properties in hand, we proceed by defining
\begin{align*}\textstyle
	I := \sum_{n=1}^\infty \int_{-\infty}^\infty \d \tau \left[ (\e^{-i\tau t} - \e^{-\lambda_n t})f_n(x) - \sum_{j=1}^m a_j (\e^{-i\tau t_j} - \e^{-\lambda_n t_j}) f_n(x_j) \right]\\
	\times \e^{i\tau t} \hat{\phi}_{\varepsilon^2}(\tau) \langle \eta_{x,\varepsilon}\,, f_n \rangle,
\end{align*}
where $\langle f\,,g \rangle = \int_0^L f(y)g(y) \d y$.
Using Fourier inversion to compute the $\d \tau$-integral and then using Lemma \ref{lem:evaluation} to evaluate the sum over $n$, we may simplify $I$ as follows:
\begin{align*}
	&\textstyle I = 2\pi \sum_{n=1}^\infty \left[ (\phi_{\varepsilon^2}(0) -\e^{-\lambda_n t}\phi_{\varepsilon^2}(t)) f_n(x)\right.\\
	&\textstyle\quad\left. - \sum_{j=1}^m a_j (\phi_{\varepsilon^2}(t-t_j)-\e^{-\lambda_n t_j}\phi_{\varepsilon^2}(0))f_n(x_j)) \right] \langle \eta_{x,\varepsilon}\,, f_n \rangle\\
	&\textstyle = 2\pi \left[(\phi_{\varepsilon^2}(0) -\e^{-\lambda_n t}\phi_{\varepsilon^2}(t))\eta_{x,\varepsilon}(x) - \sum_{j=1}^m a_j (\phi_{\varepsilon^2}(t-t_j)- \e^{-\lambda_n t_j}\phi_{\varepsilon^2}(t)) \eta_{x,\varepsilon}(x_j)\right].
\end{align*}
It follows from \eqref{eps:SLND:R} and \eqref{test3:SLND:R} that $\phi_{\varepsilon^2}(t)=0$ and $\phi_{\varepsilon^2}(t-t_j)\eta_{x,\varepsilon}(x_j) = 0$, and hence $I = 2\pi \phi_{\varepsilon^2}(0) \eta_{x,\varepsilon}(x) = 2\pi \varepsilon^{-3}$.
Therefore, Cauchy-Schwarz inequality yields
\begin{align}\label{cs:SLND:R}
	&\textstyle 4\pi^2 \varepsilon^{-6} = |I|^2 \le 2\pi\, \E\left[ \left( w(t\,,x) - \sum_{j=1}^m a_j w(t_j\,,x_j) \right)^2 \right] \times J,
\end{align}
where
\begin{align*}\textstyle
	J = \sum_{n=1}^\infty\int_{-\infty}^\infty \d \tau (\lambda_n^2 + \tau^2) |\hat{\phi}_{\varepsilon^2}(\tau)|^2 |\langle \eta_{x,\varepsilon}\,, f_n \rangle|^2.
\end{align*}
By the scaling property of Fourier transform, $\hat{\phi}_{\varepsilon^2}(\tau) = \hat{\phi}(\varepsilon^2\tau)$.
Since $\hat{\phi}$ is rapidly decreasing, it follows that
\begin{align*}\textstyle
	J \lesssim \varepsilon^{-2} \sum_{n=1}^\infty |\langle \eta_{x,\varepsilon}\,, \lambda_n f_n \rangle|^2 + \varepsilon^{-6} \sum_{n=1}^\infty |\langle \eta_{x,\varepsilon}\,, f_n \rangle|^2.
\end{align*}
In particular, we may use $-\tfrac12 f_n''=\lambda_n f_n$ and integration by parts twice to see that
\begin{align*}
	2\langle \eta_{x,\varepsilon}\,, \lambda_n f_n \rangle
	&\textstyle= -  \int_0^L \eta_{x,\varepsilon}(y)  f_n''(y) \d y\\
	&\textstyle=-[\eta_{x,\varepsilon}(y)f_n'(y)]_{y=0}^{y=L} + [ \eta_{x,\varepsilon}'(y)f_n(y)]_{y=0}^{y=L} - \int_0^L \eta_{x,\varepsilon}''(y) f_n(y) \d y\\
	&\textstyle=- \int_0^L \eta_{x,\varepsilon}''(y) f_n(y) \d y,
\end{align*}
where we have used the property that $\eta_{x,\varepsilon}$ and $f_n$ both satisfy the Robin condition \eqref{R:BC} in order to obtain the last equality.
The preceding display, together with Parseval's identity, \eqref{test2:SLND:R}, and a change of variable, implies that
\begin{align*}
	J 
	&\textstyle \lesssim \varepsilon^{-2} \sum_{n=1}^\infty |\langle \eta_{x,\varepsilon}''\,,f_n\rangle|^2 + \varepsilon^{-6} \sum_{n=1}^\infty |\langle \eta_{x,\varepsilon}\,,f_n\rangle|^2\\
	&\textstyle \le \varepsilon^{-2} \|\eta_{x,\varepsilon}''\|_{L^2([0,L])}^2 + \varepsilon^{-6} \|\eta_{x,\varepsilon}\|^2_{L^2([0,L])}\\
	&\lesssim \varepsilon^{-7},
\end{align*}
where the implicit constants depend only on $\alpha, \beta, L$ and the test functions $\phi, \psi, \chi$, but are independent of all other parameters.
Putting this back into \eqref{cs:SLND:R} and recalling \eqref{eps:SLND:R} yield
\begin{align*}\textstyle
	\E\left[ \left( w(t\,,x) - \sum_{j=1}^m a_j w(t_j\,,x_j) \right)^2 \right] \gtrsim\, \varepsilon\, \gtrsim \min\limits_{1\le j \le m}\rho^2((t\,,x)\,,(t_j\,,x_j)) \wedge \sqrt{t}.
\end{align*}
The proof is complete.
\end{proof}

To sum up, we have:

\begin{proposition}\label{pr:var:w-w:op}
Fix $T>0$. Then, under \eqref{D:BC},
\begin{align}
	\Var(w(t\,,x)-w(s\,,y)) \asymp \rho^2((t\,,x)\,,(s\,,y)) \wedge \sqrt{t \vee s} \wedge (f_1(x)\vee f_1(y))
\end{align}
uniformly for all $(t\,,x)\,,(s\,,y) \in [0\,,T]\times[0\,,L]$.
For any fixed $T>0$, under \eqref{N:BC} or \eqref{R:BC},
\begin{align}
	\Var(w(t\,,x)-w(s\,,y)) \asymp \rho^2((t\,,x)\,,(s\,,y)) \wedge \sqrt{t \vee s}
\end{align}
uniformly for all $(t\,,x)\,,(s\,,y) \in [0\,,T]\times[0\,,L]$.
\end{proposition}

\begin{proposition}\label{pr:SLND:DNR}
Fix $0<a<T$.
Fix $0<c<d<L$ under \eqref{D:BC}; and fix $0 \le c < d \le L$ under \eqref{N:BC} or \eqref{R:BC}.
Then, there exists $c_2>0$ such that
\begin{align}\label{w-w:LB}
	&\Var(w(t\,,x)-w(s\,,y)) \ge c_2 \rho^2((t\,,x)\,,(s\,,y)),\\
	\label{SLND:DNR}
	&\Var( w(t\,,x) \mid w(t_1\,,x_1)\,,\dots, w(t_n\,,x_n)) \ge c_2 \min_{1\le i \le n}\rho^2((t\,,x)\,,(t_i\,,x_i))
\end{align}
uniformly for all $n \in \N_+$ and $(s\,,y)\,,(t\,,x)\,,(t_1\,,x_1)\,,\dots,(t_n\,,x_n) \in [a\,,T]\times[c\,,d]$.
\end{proposition}


\subsection{A series representation}

Define $v = \{ v(t\,,x) \}_{t \ge 0, x \in [0,L]}$ by
\begin{align}\label{v}
	v(t\,,x) = \frac{1}{\sqrt{2\pi}} \sum_{n=1}^\infty f_n(x) \,\Re \int_{-\infty}^\infty \frac{\e^{-i\tau t}-\e^{-\lambda_n t}}{\lambda_n - i\tau} W_n(\d \tau),
\end{align}
where $W_n = W^{(1)}_n + iW^{(2)}_n$ and $\{W^{(1)}_n, W^{(2)}_n\}_{n \in \N_+}$ are i.i.d.~white noises on $\R$.
Then $v$ is a centered Gaussian random field.
The next lemma shows that $v$ has the same law as the solution $w$ to \eqref{she:add}. 

\begin{lemma}\label{lem:v=w}
The process $v$ has the same law as the solution $w$ to \eqref{she:add}.
\end{lemma}

\begin{proof}
Since $v$ and $w$ are both centered Gaussian processes, it suffices to show that they have the same covariance function. 
Indeed, for every $t\,,s\ge0$ and $x\,,y\in[0\,,L]$, by independence of $\{W_n\,, n\in \N_+\}$, Wiener isometry, and Plancherel's theorem,
\begin{align*}
	&\E[v(t\,,x)v(s\,,y)]
	= \frac{1}{2\pi} \sum_{n=1}^\infty f_n(x)f_n(y) \int_{-\infty}^\infty \left( \frac{\e^{-i\tau t}-\e^{-\lambda_n t}}{\lambda_n - i\tau} \right) \overline{\left( \frac{\e^{-i\tau s}-\e^{-\lambda_n s}}{\lambda_n - i\tau} \right)}\, \d \tau\\
	&= \sum_{n=1}^\infty f_n(x)f_n(y) \int_{-\infty}^\infty \left(\e^{-\lambda_n(t-r)}\1_{[0,t]}(r)\right) \left( \e^{-\lambda_n(s-r)}\1_{[0,s]}(r)\right) \d r\\
	&= \int_0^{t \wedge s} \d r \int_0^L \d z \,G_{t-r}(x\,,z) G_{s-r}(y\,,z) = \E[w(t\,,x)w(s\,,y)],
\end{align*}
where the last line follows from \eqref{G} and \eqref{w}.
Hence, $v$ and $w$ have the same law.
\end{proof}

For any Borel subset $A$ of $[0\,,\infty)$, $t \ge 0$, and $x\in[0\,,L]$, define the following truncated version of $v$:
\begin{align*}
	v(A\,,t\,,x) = \frac{1}{\sqrt{2\pi}}\sum_{n=1}^\infty f_n(x)\,\mathrm{Re}\int_{\{\tau \in \R: \sqrt{n}\vee |\tau|^{1/4} \in A\}} \frac{\e^{-i\tau t}-\e^{-\lambda_n t}}{\lambda_n - i\tau} W_n(\d \tau).
\end{align*}
We now verify that Assumption 2.1 of Lee and Xiao \cite{LX23} is satisfied.

\begin{lemma}\label{lem:hr}
If $A$ and $B$ are disjoint subsets of $[0\,,\infty)$, then $\{v(A\,,t\,,x)\}_{t\ge0,x\in[0,L]}$ and $\{v(B\,,t\,,x)\}_{t\ge0,x\in[0,L]}$ are independent.
Moreover, for any $T>0$, there exists $C>0$ such that for all $0 \le a < b \le \infty$, for all $(t\,,x)\,,(s\,,y)\in [0\,,T]\times[0\,,L]$,
\begin{align*}
	\|v(t\,,x) - v([a\,,b),t\,,x) - v(s\,,y) + v([a\,,b),s\,,y) \|_2
	\le C (a^3 |t-s| + a|x-y| + \tfrac{1}{b}).
\end{align*}
\end{lemma}

\begin{proof}
The first statement concerning independence is clear.
To show the second, we start with the following decomposition:
\begin{align*}
        &v(t\,,x) - v([a\,,b),t\,,x) - v(s\,,y) + v([a\,,b),s\,,y)\\
        &= \left[ v([0\,,a), t\,, x) - v([0\,,a), s\,, y) \right] + \left[ v([b\,,\infty), t\,, x) - v([b\,,\infty), s\,, y)\right].
\end{align*}
For the first component, Wiener isometry yields
\begin{align*}
    &\|v([0\,,a), t\,, x) - v([0\,,a), s\,, y)\|_2^2\\
    &= \frac{1}{2\pi} \sum_{1 \le n \le a^2}  \int_{|\tau| < a^4} \frac{\left|f_n(x) (e^{-i\tau t} - e^{-\lambda_n t}) - f_n(y) (e^{-i\tau s} - e^{-\lambda_n s}) \right|^2}{\lambda_n^2+\tau^2} \d\tau,
\end{align*}
where the summation reduces to one that only runs over $1 \le n \le a^2$ because the domain of integration $\{ \tau \in \R : \sqrt{n} \vee |\tau|^{1/4} \in [0\,,a) \}$ is empty when $n > a^2$.
By triangle inequality, mean value theorem, and \eqref{f_n:bd},
\begin{align*}
	&\left|f_n(x) (e^{-i\tau t} - e^{-\lambda_n t}) - f_n(y) (e^{-i\tau s} - e^{-\lambda_n s}) \right|\\
	&\le \left|f_n(x) - f_n(y)\right||e^{-i\tau t} - e^{-\lambda_n t}| + \left|f_n(y)\right|  \left|(e^{-i\tau t} - e^{-\lambda_n t}) - (e^{-i\tau s} - e^{-\lambda_n s})\right|\\
        &\lesssim n |x-y| +\left( |\tau| + \lambda_n\right)|t-s|.
\end{align*}
This together with \eqref{lambda} implies that
\begin{align*}
	&\|v([0\,,a), t\,, x) - v([0\,,a), s\,, y) \|_2^2\\
	&\lesssim  \sum_{1 \le n \le a^2} \left[ \int_\R \frac{n^2|x-y|^2}{\lambda_n^2 + \tau^2} \d \tau + \int_{|\tau|< a^4} \frac{(|\tau| + \lambda_n)^2}{\tau^2 + \lambda_n^2}|t-s|^2 \d \tau \right]\\
	&\lesssim \sum_{1 \le n \le a^2}\left[ \frac{n^2}{\lambda_n}  |x-y|^2 + a^4 |t-s|^2\right]
	\lesssim  a^2|x-y|^2 + a^6|t-s|^2. 
\end{align*}
For the other component, we use the property that $f_n(x) (e^{-i\tau t} - e^{-\lambda_n t}) - f_n(y) (e^{-i\tau s} - e^{-\lambda_n s})$ is bounded (see \eqref{f_n:bd}) and \eqref{lambda} to deduce that
\begin{align*}
	&\|v([b\,,\infty), t\,, x) - v([b\,,\infty), s\,, y) \|_2^2\\
	&= \frac{1}{2\pi} \sum_{n\ge b^2} \int_{\R} \frac{\left|f_n(x) (e^{-i\tau t} - e^{-\lambda_n t}) - f_n(y) (e^{-i\tau s} - e^{-\lambda_n s}) \right|^2}{\lambda_n^2+\tau^2} \d\tau\\
	& \quad + \frac{1}{2\pi} \sum_{1 \le n \le b^2}  \int_{|\tau| \ge b^4} \frac{\left|f_n(x) (e^{-i\tau t} - e^{-\lambda_n t}) - f_n(y) (e^{-i\tau s} - e^{-\lambda_n s}) \right|^2}{\lambda_n^2+\tau^2} \d\tau\\
	& \lesssim \sum_{n \ge b^2} \int_\R \frac{\d\tau}{\lambda_n^2 + \tau^2}  +  \sum_{1 \le n \le b^2} \int_{|\tau|\ge b^4} \frac{\d \tau}{\lambda_n^2 + \tau^2}\\
	&\lesssim \sum_{n \ge b^2} \lambda_n^{-1} +  \sum_{1 \le n \le b^2} b^{-4} \lesssim \int_{b^2}^\infty \frac{\d z}{z^2} + b^{-2} \lesssim b^{-2}.
\end{align*}
Combining both parts together, we complete the proof.
\end{proof}

\subsection{Spatio-temporal increments}

The theorem below establishes the exact local and uniform spatio-temporal moduli of continuity for the solution to \eqref{she:add}.

\begin{theorem}\label{th:w:lil:mc}
For any fixed point $z_0=(t_0\,,x_0)\in [0\,,\infty)\times[0\,,L]$, there exists a constant $K_0 = K_0(z_0) \in (0\,,\infty)$ such that
\begin{align}\label{w:lil}
	\lim_{\varepsilon\to0^+} \sup_{z \in B^*_\rho(z_0,\varepsilon)} \frac{|w(z)-w(z_0)|}{\rho(z\,,z_0)\sqrt{\log\log(1/\rho(z\,,z_0))}} = K_0
	\quad \text{a.s.}
\end{align}
For every fixed interval $I=[a\,,T]\times[c\,,d]$ with $0<a<T$ and $0< c<d<L$, there exists a constant $K = K(a\,,T\,,c\,,d) \in (0\,,\infty)$ such that
\begin{align}\label{w:mc}
	\lim_{\varepsilon\to0^+} \sup_{{z,z'\in I: 0<\rho(z,z')\le \varepsilon}} \frac{|w(z)-w(z')|}{\rho(z\,,z')\sqrt{\log(1/\rho(z\,,z'))}} = K \quad \text{a.s.}
\end{align}
and $\sqrt{12 c_2} \le K \le \sqrt{12 c_1}$, where $c_1$ is any constant satisfying \eqref{w-w} and $c_2$ is any constant satisfying \eqref{SLND:DNR}.
When $a=0$ or $0 \le c < d \le L$, \eqref{w:mc} still holds for a constant $K$ such that $\sqrt{12 c_2} \le K(a\,,T\,,c\,,d)<\infty$, where $c_2$ is any constant satisfying \eqref{SLND:DNR} for any interval $[T/2\,,T] \times [c',d']$ with $c < c' < d' < d$.
\end{theorem}

\begin{proof}
Suppose $t_0>0$, $a>0$, and assume either Dirichlet condition \eqref{D:BC} with $0<x_0<L$, $0<c<d<L$, or Neumann/Robin condition \eqref{N:BC}/\eqref{R:BC} with $0 \le x_0 \le L$, $0 \le c < d \le L$.
In these cases, thanks to SLND (Proposition \ref{pr:SLND:DNR}) and Lemma \ref{lem:hr}, the assumptions of Theorems 5.2 and 6.1 of Lee and Xiao \cite{LX23} are satisfied for $\{ w(t\,,x)\}_{(t\,,x) \in I}$, hence \eqref{w:lil} and \eqref{w:mc} follow directly from those two theorems.

It remains to deal with the following cases:
\begin{enumerate}
\item[(i)] $t_0=0$, $a=0$, $0<x_0<L$, $0<c<d<L$ with boundary condition \eqref{D:BC}, \eqref{N:BC} or \eqref{R:BC};
\item[(ii)] $t_0>0$, $a>0$, $x_0 \in \{0\,,L\}$, $c = 0$ or $d=L$ with boundary condition \eqref{D:BC};
\item[(iii)] $t_0=0$, $a=0$, $x_0 \in \{0\,,L\}$, $c=0$ or $d=L$ with boundary condition \eqref{D:BC}.
\end{enumerate}
These cases need to be treated with care because the variance bounds have a different form (see Proposition \ref{pr:var:w-w:op}) than in \cite{LX23}, thus we cannot directly apply the results in \cite{LX23}.

(i). Suppose $t_0=0$, $a=0$, $0<x_0<L$, $0<c<d<L$ with boundary condition \eqref{D:BC}, \eqref{N:BC} or \eqref{R:BC}.
Let 
\[
	\phi(z\,,z') = \rho(z\,,z')\sqrt{\log\log(1/\rho(z\,,z'))},
	\quad
	\psi(z\,,z') = \rho(z\,,z')\sqrt{\log(1/\rho(z\,,z'))}.
\]
Define the metric $d(z\,,z') = C_0^{-1/2} \|w(z)-w(z')\|_2$ for any $z\,,z'\in I$, where $C_0>0$ is the constant in Lemma \ref{lem:w-w}. By Lemma \ref{lem:w-w},
\[
	\lim_{\varepsilon\to0^+} \sup_{z,z'\in I: 0<\rho(z\,,z')\le \varepsilon} \frac{d(z\,,z')}{\phi(z\,,z')} = 0 \quad \text{and}\quad
	\lim_{\varepsilon\to0^+} \sup_{z,z'\in I: 0<\rho(z\,,z')\le \varepsilon} \frac{d(z\,,z')}{\psi(z\,,z')} = 0.
\]
This allows us to apply a zero-one law for Gaussian random fields \cite[Lemma 7.1.1]{MR} to deduce that \eqref{w:lil} and \eqref{w:mc} hold for some constants $K_0=K_0(z_0) \in [0\,,\infty]$ and $K=K(0\,,T\,,c\,,d)\in [0\,,\infty]$, respectively.
We aim to show that $0<K_0<\infty$ and $0< K<\infty$.

First, $K_0<\infty$ can be shown by the following argument using metric entropy and concentration of measure.
For any set $A \subset I$, consider the metric entropy $N(A\,,r)$, i.e., the smallest number of $d$-balls of radius $r$ needed to cover $A$.
Then, for any $\varepsilon>0$, Dudley's theorem \cite{Dudley} states that
\[
	\E\left[\sup_{z\in B_\rho(z_0,\varepsilon)} |w(z)| \right]  \lesssim \int_0^D \sqrt{\log N(B_\rho(z_0\,,\varepsilon)\,,r)}\, \d r,
\]
where $D$ is the $d$-diameter of $B_\rho(z_0\,,\varepsilon)$, which satisfies $D \lesssim \varepsilon$ by Lemma \ref{lem:w-w}.
To estimate $N(B_\rho(z_0\,,\varepsilon)\,,r)$ for $0 < r < \varepsilon$, we split $B_\rho(z_0\,,\varepsilon) = [0\,,\varepsilon^4] \times [x_0-\varepsilon^2\,,x_0+\varepsilon^2]$ into two parts: $([0\,,r^4]\times [x_0-\varepsilon^2\,,x_0+\varepsilon^2]) \cup ([r^4\,,\varepsilon^4]\times[x_0-\varepsilon^2\,,x_0+\varepsilon^2])$.
By Lemma \ref{lem:w-w}, the first part is covered by a single $d$-ball of radius $r$, and the second part is covered by $C\varepsilon^2(\varepsilon^4-r^4)/r^6$ many $d$-balls of radius $r$, hence $N(B_\rho(z_0\,,\varepsilon)\,,r) \le 1 + C\varepsilon^2(\varepsilon^4-r^4)/r^6 \lesssim (\varepsilon/r)^6$. 
It follows that there exist $C_1\,,C_2\,,C_3>0$ such that for all $\varepsilon \in (0\,,1)$,
\[
	\E\left[\sup_{z\in B_\rho(z_0,\varepsilon)} |w(z)| \right] \le C_1 \int_0^{C_1\varepsilon} \sqrt{\log(C_1\varepsilon/r)}\, \d r \le C_2 \varepsilon \int_0^\infty s^2 \e^{-s^2} \d s \le C_3 \varepsilon.
\]
Keeping in mind that $z_0 = (0\,,x_0)$ and $w(z_0)=0$, we have $\sup_{z \in B_\rho(z_0,\varepsilon)}\E|w(z)|^2 \le C_0 \varepsilon^2$ by Lemma \ref{lem:w-w}.
Let $C>0$ and $\varepsilon_n = \e^{-n}$.
We may apply Borell's inequality \cite{Borell} to see that for $n$ large,
\begin{align*}
	&\P\left\{ \sup_{z \in B_\rho(z_0,\varepsilon_n)}|w(z)| > C \varepsilon_n \sqrt{\log\log(1/\varepsilon_n)} \right\}\\
	&\le \P\left\{ \sup_{z \in B_\rho(z_0,\varepsilon_n)}|w(z)| - \E\left[\sup_{z\in B_\rho(z_0,\varepsilon)} |w(z)| \right] > (C/2) \varepsilon_n \sqrt{\log\log(1/\varepsilon_n)} \right\}\\
	&\le \exp\left( -\frac{((C/2) \varepsilon_n \sqrt{\log\log(1/\varepsilon_n)})^2}{2\sup_{z \in B_\rho(z_0,\varepsilon)}\E|w(z)|^2} \right) \le \exp\left( - \frac{C^2\log n}{8C_0} \right) = n^{-C^2/(8C_0)},
\end{align*}
which is summable, say, for $C= 4\sqrt{C_0}$.
Then, by Borel-Cantelli lemma and monotonicity,
\[
	\lim_{\varepsilon\to0^+} \sup_{z \in B^*_\rho(z_0,\varepsilon)} \frac{|w(z)|}{\rho(z\,,z_0)\sqrt{\log\log(1/\rho(z\,,z_0))}} \le C \quad \text{a.s.}
\]
This shows that \eqref{w:lil} holds with $K_0 \le C <\infty$. To show that $K_0>0$, since Lemma \ref{lem:w-w} implies that $\|w(t\,,x_0)-w(s\,,x_0)\|_2 \asymp |t-s|^{1/4}$ for all $t\,,s\in [0\,,1]$, we may apply Theorem 5.1 of Lee and Xiao \cite{LX23} to the process $\{w(t\,,x_0)\}_{t \in [0,1]}$ to find that
\[
	\lim_{\varepsilon\to0^+} \sup_{t \in (0,\varepsilon]} \frac{|w(t\,,x_0)|}{t^{1/4}\sqrt{\log\log(1/t^{1/4})}}  = K_2 \quad \text{a.s.}
\]
for some constant $K_2 \in (0\,,\infty)$. Clearly, the quantity in \eqref{w:lil} is no less than the above quantity, and hence $K_0\ge K_2>0$.

Next, we show that $K=K(0\,,T\,,c\,,d)\in (0\,,\infty)$.
It is possible to directly use the form of SLND in Lemmas \ref{lem:SLND:DN} and \ref{lem:SLND:R} and follow \cite{LX19, LX23} to prove that $K>0$.
Alternatively, we may simply use the $a>0$ case in the beginning of this proof to deduce that $K = K(0\,,T\,,c\,,d) \ge K(T/2\,,T\,,c',d') \ge \sqrt{12 c_2} > 0$, where $c_2$ is any constant satisfying \eqref{SLND:DNR} for any interval $[T/2\,,T\,,c',d']$ with $c<c'<d'<d$.
To show that $K<\infty$, we use again a metric entropy argument,
which shows that $N(I\,,r) \lesssim r^{-6}$.
Set $\varepsilon_n = \e^{-n}$.
By Theorem 1.3.5 of \cite{AT}, there exist $C_5\,,C_6\,,C_7\,,C_8>0$ such that, a.s., for all large $n$,
\begin{align*}
	\sup_{z,z'\in I: d(z,z') \le \varepsilon_n} |w(z)-w(z')| 
	&\le C_5 \int_0^{\varepsilon_n} \sqrt{\log N(I\,,r)} \, \d r
	\le C_6 \int_0^{\e^{-n}} \sqrt{\log(C_6/r)} \, \d r\\
	&\le C_7 \int_{\sqrt{\log(C_6\e^n)}}^\infty s^2 \e^{-s^2} \d s \le C_8 \e^{-n} \sqrt{\log(C_6\e^n)},
\end{align*}
where the last inequality follows from the fact that $\int_a^\infty s^2\e^{-s^2} \d s \lesssim a\e^{-a^2}$ as $a \to \infty$, and $C_5\,,C_6\,,C_7\,,C_8$ are universal constants that do not depend on $n$.
This together with Lemma \ref{lem:w-w} implies that, a.s.,
\begin{align*}
	&\lim_{n\to \infty}\sup_{\substack{z,z'\in I\\ \varepsilon_{n+1}\le\rho(z,z') \le \varepsilon_n}} \frac{|w(z)-w(z')|}{\psi(z\,,z')}
	\le \lim_{n\to \infty}\sup_{\substack{z,z'\in I\\ \varepsilon_{n+1}\le\rho(z,z') \le \varepsilon_n}} \frac{|w(z)-w(z')|}{\varepsilon_{n+1}\sqrt{\log(1/\varepsilon_{n+1})}}\\
	&\le \lim_{n\to \infty}\sup_{\substack{z,z'\in I\\ 0<d(z,z') \le \varepsilon_n}} \frac{|w(z)-w(z')|}{\e^{-n-1}\sqrt{n+1}}
	\le \lim_{n\to \infty}\sup_{\substack{z,z'\in I\\ 0<d(z,z') \le \varepsilon_n}} \frac{C_8 \e^{-n} \sqrt{\log(C_6\e^n)}}{\e^{-n-1}\sqrt{n+1}} \le C_8\e.
\end{align*}
This implies that $K\le C_8 \e<\infty$.

For (ii) and (iii), we treat only the particular case of $z_0 = (0\,,0)$, $a=0$, $c=0$, $d<L$ under Dirichlet condition \eqref{D:BC}, and leave the remaining cases to the reader.
The proofs of \eqref{w:lil} and \eqref{w:mc} for the particular case are similar to those in case (i) above, so we only mention the major differences.
Namely, we split $B_\rho(z_0\,,\varepsilon) = [0\,,\varepsilon^4] \times [0\,,\varepsilon^2]$ into two parts:
\[
	\{ (t\,,x) : 0 \le t \le r^4 \text{ or } 0 \le x \le r^2 \} \cup ([r^4\,,\varepsilon^4]\times[r^2\,,\varepsilon^2]).
\]
By Lemma \ref{lem:w-w}, the first part is covered by one $d$-ball of radius $r$, namely $B_d(z_0\,,r)$; and the second part is covered by $C(\varepsilon^4-r^4)(\varepsilon^2-r^2)/r^6$ many $d$-balls of radius $r$, where $d$ is as defined in case (i).
Hence, $N(B_\rho(z_0\,,\varepsilon)\,,r) \le 1 + C(\varepsilon^4-r^4)(\varepsilon^2-r^2)/r^6$. Again, this yields $N(B_\rho(z_0\,,\varepsilon)\,,r) \lesssim (\varepsilon/r)^6$.
Therefore, we can repeat the proof in case (i) to deduce \eqref{w:lil} with $K_0 \in (0\,,\infty)$.
A similar metric entropy argument shows that $N([0\,,T]\times[0\,,d]\,,r) \lesssim r^{-6}$, hence the same proof in case (i) leads to \eqref{w:mc} with $\sqrt{12 c_2} \le K(0\,,T\,,c\,,d)<\infty$, where $c_2$ is any constant satisfying \eqref{SLND:DNR} for any interval $[T/2\,,T\,,c',d']$ with $c<c'<d'<d$.
\end{proof}

The next result yields matching bounds on small-ball probabilities and a Chung-type law of the iterated logarithm for spatio-temporal increments of $w$.

\begin{theorem}\label{th:w:sb:chung}
For every fixed $z_0 = (t_0\,,x_0) \in [0\,,\infty) \times [0\,,L]$, there exist constants $0<c_0<c_1<\infty$ such that for all $0<\varepsilon<r< 1$,
\begin{align}\label{w:sb}
	\e^{-c_1(r/\varepsilon)^6} \le \P\left\{ \sup_{z\in B_\rho(z_0,r)}|w(z)-w(z_0)| \le \varepsilon \right\} \le \e^{-c_0(r/\varepsilon)^6}
\end{align}
and
\begin{align}\label{w:chung}
	\liminf_{\varepsilon \to 0^+} \frac{(\log\log(1/\varepsilon))^{1/6}}{\varepsilon}\sup_{z \in B_\rho(z_0,\varepsilon)} |w(z)-w(z_0)| = C_2
	\quad \text{a.s.}
\end{align}
where $C_2$ is a constant such that $c_0^{1/6}\le C_2\le c_1^{1/6}$.
\end{theorem}

\begin{proof}
Suppose $t_0>0$, and assume either Dirichlet condition \eqref{D:BC} with $0<x_0<L$ or Neumann/Robin condition \eqref{N:BC}/\eqref{R:BC} with $0 \le x_0 \le L$.
Thanks to SLND (Proposition \ref{pr:SLND:DNR}) and Lemma \ref{lem:hr}, we may apply Proposition 4.2 and Theorem 4.4 of Lee and Xiao \cite{LX23} to obtain \eqref{w:sb} and \eqref{w:chung}.

Now assume $t_0=0$, or Dirichlet condition with $x_0=0$ or $L$. 
Let $r \in (0\,,1]$.
Noting that $w(z_0)=0$,
we can use a metric entropy argument as in the proof of Theorem \ref{th:w:lil:mc} to show that there exists $C>0$ such that $N(B_\rho(z_0\,,r)\,,\varepsilon) \le \Psi_r(\varepsilon) := C (r/\varepsilon)^6$ for all $\varepsilon \in (0\,, r]$.
Then, by a small-ball probability estimate of Talagrand \cite[Lemma 2.2]{T95} (see also \cite[Lemma 3.4]{DLMX21} for a more precise statement), there exists a universal constant $K>0$ such that for all $\varepsilon \in (0\,,r)$,
\[
	\P\left\{ \sup_{z\in B_\rho(z_0,r)}|w(z)| \le \varepsilon \right\} \ge \exp\left( - \frac{\Psi_r(\varepsilon)}{K}\right)
	= \exp\left( - \frac{C}{K} \left( \frac{r}{\varepsilon} \right)^6 \right).
\]
Next, we apply SLND to establish a reverse inequality.
Let $0<\varepsilon < r <1$ and define a finite subset $F$ of $B_\rho(z_0\,,r) = [0\,,r^4]\times[x_0-r^2\,,x_0+r^2]$ by
\[
	F = B_\rho(z_0\,,r) \cap \{ (k_1 \varepsilon^4\,, k_2 \varepsilon^2) : k_1\,,k_2 \in \N_+ \}.
\]
Then $\# F \asymp (r/\varepsilon)^6$ and $\rho(z\,,z') \ge \varepsilon$ for any pair of distinct $z\,,z'\in F$.
Assign an order to the points in $F$ and label them as $z_1\,,z_2\,,\dots,z_n$.
Then, by conditioning and Anderson's shifted-ball inequality \cite{Anderson}, 
\begin{align*}
	&\P\left\{ \sup_{z \in B_\rho(z_0,r)}|w(z)| \le \varepsilon \right\} \le \P\left\{ \max_{1 \le i \le n} |w(z_i)| \le \varepsilon \right\}\\
	& = \E\left[ \1_{\big\{\max\limits_{1\le i \le n-1}|w(z_i)| \le \varepsilon\big\}} \P\left\{ |w(z_n)| \le \varepsilon \mid w(z_1)\,,\dots, w(z_{n-1}) \right\} \right]\\
	& \le \P\left\{ \max\limits_{1\le i \le n-1}|w(z_i)| \le \varepsilon \right\} \P\left\{ |Z| \le \frac{\varepsilon}{[\Var(w(z_n)\mid w(z_1)\,,\dots,w(z_{n-1}))]^{1/2}} \right\},
\end{align*}
where $Z$ has a standard normal distribution.
Thanks to SLND (Lemmas \ref{lem:SLND:DN} and \ref{lem:SLND:R}), there exists $c_2>0$ such that
\begin{align*}
	\P\left\{ |Z| \le \frac{\varepsilon}{[\Var(w(z_n)\mid w(z_1)\,,\dots,w(z_{n-1}))]^{1/2}} \right\}
	\le \P\left\{ |Z| \le c_2^{-1/2}\right\}.
\end{align*}
In fact, by Lemmas \ref{lem:SLND:DN} and \ref{lem:SLND:R}, $\Var(w(z_i)\mid w(z_1)\,,\dots,w(z_{i-1})) \ge c_2 \varepsilon^2$ for every $1 \le i \le n$.
Hence, by induction, we can find $c\,,c_0>0$ such that for all $0<\varepsilon<r<1$,
\begin{align*}
	\P\left\{ \sup_{z \in B_\rho(z_0,r)}|w(z)| \le \varepsilon \right\} \le \left( \P\left\{ |Z| \le c_2^{-1/2} \right\} \right)^n = \e^{-cn} \le \e^{-c_0(r/\varepsilon)^6}.
\end{align*}

Next, we aim to show \eqref{w:chung} under $t_0=0$ or Dirichlet condition with $x_0=0$ or $L$.
Again, in either case, $w(z_0)=0$.
Thanks to Lemma \ref{lem:hr} and a zero-one law of Lee and Xiao \cite[Lemma 3.1]{LX23}, \eqref{w:chung} holds for some constant $C_2 \in [0\,,\infty]$.
Let $\varepsilon_n = \e^{-n}$.
Thanks to the upper bound in \eqref{w:sb}, 
\begin{align*}
	\sum_{n=1}^\infty \P\left\{ \sup_{z\in B_\rho(z_0,\varepsilon_n)}|w(z)| \le C \varepsilon_n(\log\log(1/\varepsilon_n))^{-1/6} \right\}
	\le \sum_{n=1}^\infty n^{-c_0/C^6},
\end{align*}
which is convergent provided that $C$ is any fixed number so that $0<C<c_0^{1/6}$.
It follows by Borel-Cantelli lemma that $C_2 \ge C$.
Letting $C\uparrow c_0^{1/6}$ yields $C_2 \ge c_0^{1/6}$.
It remains to show that $C_2\le c_1^{1/6}$.
We follow the proof of \cite[Theorem 4.4]{LX23}.
Fix $\delta\in (0\,,1)$. For any $n\in\N$, let $\varepsilon_n=\exp(-(n^\delta + n^{1+\delta}))$ and $b_n=\exp(n^{1+\delta})$. 
Recall the Gaussian random field $v$ defined in \eqref{v}.
For any $z\in [0\,,\infty)\times[0\,,L]$, define $v_n(z)=v([b_n\,,b_{n+1})\,,z)$ and $\tilde{v}_n(z)=v([0\,,\infty)\setminus[b_n\,,b_{n+1})\,,z)$, so that $v(z)=v_n(z)+\tilde{v}_n(z)$. 
Write $h(\varepsilon)=\varepsilon(\log\log(1/\varepsilon))^{-1/6}$.
Since $v_n$ and $\tilde{v}_n$ are independent, we may apply conditionally Anderson's inequality \cite{Anderson}, Lemma \ref{lem:v=w}, and the lower bound in \eqref{w:sb} to deduce that
\begin{align*}
&\P\left\{\sup_{z\in B_\rho(z_0,\varepsilon_n)} |v_n(z)| \le C h(\varepsilon_n)\right\}
\ge \P\left\{\sup_{z\in B_\rho(z_0,\varepsilon_n)} |v_n(z)+\tilde{v}_n(z)| \le C h(\varepsilon_n)\right\}\\
&= \P\left\{\sup_{z\in B_\rho(z_0,\varepsilon_n)} |w(z)| \le C h(\varepsilon_n)\right\} \ge \exp\left(-c_1\left(\frac{\varepsilon_n}{C h(\varepsilon_n)} \right)^6\right)
\gtrsim n^{-(1+\delta)c_1/C^6}.
\end{align*}
Since $v_1\,,v_2\,,\dots$ are independent, we may take $C = ((1+\delta)c_1)^{1/6}$ and apply the second Borel-Cantelli lemma to see that
\begin{align}\label{v:chung}
\liminf_{n\to\infty} \sup_{z\in B_\rho(z_0,\varepsilon_n)} \frac{|v_n(z)|}{h(\varepsilon_n)} \le ((1+\delta)c_1)^{1/6} \quad \text{a.s.}
\end{align}
Thanks to Lemma \ref{lem:hr}, we can follow the proof of \cite[Theorem 4.4]{LX23} using a metric entropy method with a concentration inequality to show that
\begin{align}\label{v:tilde:chung}
\limsup_{n\to\infty} \sup_{z\in B_\rho(z_0,\varepsilon_n)} \frac{|\tilde{v}_n(z)|}{h(\varepsilon_n)}=0\quad \text{a.s.}
\end{align}
Combining \eqref{v:chung} and \eqref{v:tilde:chung} and letting $\delta \to 0^+$ shows that $C_2\le c_1^{1/6}$.
\end{proof}

\section{Linearization error}\label{s:E}

Recall the mild formulation \eqref{she:mild} of the SPDE \eqref{she} and the solution $w$ to the linear SPDE \eqref{she:add}.
For any $t\,,t'\in [0\,,\infty)$ and $x\,,x'\in [0\,,L]$, define
\begin{align}\begin{split}\label{E}
	\mathscr{E}(t\,,x\,; t',x')
	= u(t',x') - u(t\,,x) 
	&- [(G_{t'}\ast u_0)(x') - (G_t \ast u_0)(x)]\\
	&- \sigma(u(t\,,x)) (w(t',x') - w(t\,,x)).
\end{split}\end{align}
The random variable $\mathscr{E}(t\,,x\,; t',x')$ measures the linearization error of the spatio-temporal increments of the solution from $(t\,,x)$ to $(t',x')$.
In order to simplify the notation, we let
\begin{align}\begin{split}\label{u:tilde}
	\tilde{u}(t\,,x) := u(t\,,x) - (G_t \ast u_0)(x)
	&= \int_{(0,t) \times [0,L]} G_{t-s}(x\,,y) b(u(s\,,y))\, \d s\, \d y\\
	&\;\; + \int_{(0,t)\times [0,L]} G_{t-s}(x\,,y) \sigma(u(s\,,y))\xi(\d s \, \d y)
\end{split}\end{align}
so that
\[
	\mathscr{E}(t\,,x\,; t',x')
	= \tilde{u}(t',x') - \tilde{u}(t\,,x) 
	- \sigma(u(t\,,x)) (w(t',x') - w(t\,,x)).
\]

\subsection{Moment estimates}\label{S:moment:error}


\begin{proposition}\label{pr:linear}
There is a number $\zeta>1$ such that the following statement holds.
If $b$ and $\sigma$ are bounded, then for any $0<a<T$, there exists $C>0$ such that
\begin{align}\label{E:bd}
	\|\mathscr{E}(t\,,x\,;t',x')\|_k \le C k [\rho((t\,,x)\,,(t',x'))]^{\zeta}
\end{align}
uniformly for all $(t\,,x)\,,(t',x') \in I:=[a\,,T] \times [0\,,L]$ and $k \in [2\,,\infty)$.
This remains valid when $I=[0\,,T]\times[c\,,d]$ for fixed $T>0$ and $0\le c<d\le L$ if \eqref{G*u:I} holds.
\end{proposition}

The rest of Section \ref{S:moment:error} is devoted to proving Proposition \ref{pr:linear}.
We first establish some lemmas.

\begin{lemma}\label{lem:sup:u}
If $b$ and $\sigma$ are bounded, then for any $0<a<T$, there exists $C>0$ such that $\sup_{t\in [a,T], x \in [0,L]}\|u(t\,,x)\|_k \le C \sqrt{k}$ for all $k \in [2\,,\infty)$.
\end{lemma}

\begin{proof}
Write $u(t\,,x) = I_0+I_1+I_2$, where
\begin{gather*}
	I_0 \textstyle= (G_t\ast u_0)(x), \quad 
	I_1 \textstyle= \int_{(0,t)\times[0,L]} G_{t-s}(x\,,y) b(u(s\,,y)) \, \d s\, \d y,\\
	I_2 \textstyle= \int_{(0,t)\times[0,L]} G_{t-s}(x\,,y) \sigma(u(s\,,y)) \, \xi(\d s\, \d y).
\end{gather*}
First, it is easy to show that $I_0$ is bounded on $[a\,,T]\times[0\,,L]$ using \eqref{G}, $u_0 \in L^2([0\,,L])$, and Lemma \ref{lem:G}.
Next, by Minkowski's inequality, the boundedness of $b$, Cauchy-Schwarz inequality, and Lemma \ref{lem:int:G},
\begin{align*}
	\|I_1\|_k &\textstyle\le \int_0^t \d s \int_0^L \d y\, |G_{t-s}(x\,,y)| \|b(u(s\,,y))\|_k\lesssim \int_0^t \d s \int_0^L \d y\, |G_{s}(x\,,y)| \\
	&\textstyle \le \int_0^t \d s \sqrt{L} \left[ \int_0^L |G_s(x\,,y)|^2 \d y \right]^{1/2}
	\lesssim \int_0^t s^{-1/4} \d s \lesssim t^{3/4}.
\end{align*}
Finally, by the Burkholder-Davis-Gundy (BDG) inequality \cite[Proposition 4.4]{DK}, the boundedness of $\sigma$, and Lemma \ref{lem:int:G},
\begin{align*}
	\|I_2\|_k^2
	&\textstyle\le k \int_0^t \d s \int_0^L \d y \, [G_{t-s}(x\,,y)]^2 \|\sigma(u(s\,,y))\|_k^2\\
	&\textstyle \lesssim k \int_0^t \d s \int_0^L \d y \, [G_{s}(x\,,y)]^2
	\lesssim k \sqrt{t}.
\end{align*}
Combine the estimates to finish the proof.
\end{proof}

\begin{lemma}\label{lem:u-u:x}
If $b$ and $\sigma$ are bounded, then for any $T>0$, there is $C>0$ such that
\begin{align*}
	\|\tilde{u}(t\,,x')-\tilde{u}(t\,,x)\|_k \le C \sqrt{k}\, |x'-x|^{1/2}
\end{align*}
uniformly for all $k \in [2\,,\infty)$, $t\in [0\,,T]$ and $x\,,x'\in [0\,,L]$.
\end{lemma}

\begin{proof}
Write $\tilde{u}(t\,,x')-\tilde{u}(t\,,x) = I_1+I_2$, where
\begin{align*}
	I_1 &\textstyle= \int_{(0,t)\times[0,L]} [G_{t-s}(x',y)-G_{t-s}(x\,,y)] b(u(s\,,y)) \, \d s\, \d y,\\
	I_2 &\textstyle= \int_{(0,t)\times[0,L]} [G_{t-s}(x',y)-G_{t-s}(x\,,y)] \sigma(u(s\,,y)) \, \xi(\d s\, \d y).
\end{align*}
Thanks to Minkowski's inequality, the boundedness of $b$, Cauchy-Schwarz inequality, and Lemma \ref{lem:int:G:x},
\begin{align*}
	\|I_1\|_k&\textstyle
	\le \int_0^t \d s \int_0^L \d y\, |G_{t-s}(x',y)-G_{t-s}(x\,,y)| \|b(u(s\,,y))\|_k\\
	&\textstyle \lesssim \int_0^t \d s \int_0^L \d y\, |G_{s}(x',y)-G_{s}(x\,,y)|\\
	&\textstyle\lesssim \sqrt{tL} \left[ \int_0^t \d s \int_0^L \d y\, |G_{s}(x',y)-G_{s}(x\,,y)|^2 \right]^{1/2}
	\lesssim |x'-x|^{1/2}.
\end{align*}
By the BDG inequality \cite[Prop. 4.4]{DK}, the boundedness of $\sigma$, and Lemma \ref{lem:int:G:x},
\begin{align*}
	\|I_2\|_k^2
	&\textstyle\le k \int_0^t \d s \int_0^L \d y \, [G_{t-s}(x',y)-G_{t-s}(x\,,y)]^2 \|\sigma(u(s\,,y))\|_k^2\\
	&\textstyle \lesssim k \int_0^t \d s \int_0^L \d y \, [G_s(x',y)-G_{s}(x\,,y)]^2
	\lesssim k |x'-x|.
\end{align*}
The proof is complete.
\end{proof}

\begin{lemma}\label{lem:u-u:t}
If $b$ and $\sigma$ are bounded, then for any $T>0$, there is $C>0$ such that
\begin{align*}
	\|\tilde{u}(t',x)-\tilde{u}(t\,,x)\|_k \le C \sqrt{k}\, |t'-t|^{1/4}
\end{align*}
uniformly for all $k\in [2\,,\infty)$, $t\,,t'\in [0\,,T]$ and $x\in [0\,,L]$.
\end{lemma}

\begin{proof}
Suppose $t<t'$.
Write $\tilde{u}(t',x)-\tilde{u}(t\,,x) = I_1+I_2+I_3+I_4$, where
\begin{align*}
	I_1 &\textstyle= \int_0^t \d s \int_0^L \d y\, [G_{t'-s}(x\,,y)-G_{t-s}(x\,,y)] b(u(s\,,y)),\\
	I_2 &\textstyle= \int_t^{t'} \d s \int_0^L \d y \, G_{t'-s}(x\,,y) b(u(s\,,y)),\\
	I_3&\textstyle= \int_{(0,t)\times[0,L]} [G_{t'-s}(x\,,y)-G_{t-s}(x\,,y)] \sigma(u(s\,,y)) \, \xi(\d s \, \d y),\\
	I_4 &\textstyle= \int_{(t,t')\times[0,L]} G_{t'-s}(x\,,y) \sigma(u(s\,,y)) \, \xi(\d s \, \d y).
\end{align*}
Since $b$ is bounded, Minkowski's inequality, Cauchy-Schwarz inequality and Lemma \ref{lem:int:G:t} yield
$\|I_1\|_k \lesssim |t'-t|^{1/4}$ and $\|I_2\|_k \lesssim |t'-t|^{1/4}$.
Also, since $\sigma$ is bounded, it follows from the BDG inequality \cite[Prop. 4.4]{DK} and Lemma \ref{lem:int:G:t} that
\begin{align*}
	\|I_3\|_k^2 
	&\textstyle\le k \int_0^t \d s \int_0^L \d y \, [G_{t'-s}(x\,,y)-G_{t-s}(x\,,y)]^2
	\lesssim k (t'-t)^{1/2}
\end{align*}
and
\begin{align*}\textstyle
	\|I_4\|_k^2
	\le k \int_t^{t'} \d s \int_0^L \d y \, [G_{t'-s}(x\,,y)]^2
	\lesssim k (t'-t)^{1/2}.
\end{align*}
Combine the estimates to finish the proof.
\end{proof}

\begin{lemma}\label{lem:E:exp}
If $b$ and $\sigma$ are bounded, then for any $T>0$, there is $\gamma>0$ such that
\begin{align}\label{E:exp:u}
	\E\left[ \exp\left( \gamma \sup_{z,z'\in [0,T]\times[0,L]} \left|\frac{\tilde{u}(z)-\tilde{u}(z')}{\rho(z\,,z')\sqrt{\log_+(1/\rho(z\,,z'))}}\right|^2 \right)\right]<\infty
\end{align}
and
\begin{align}\label{E:exp:w}
	\E\left[ \exp\left( \gamma \sup_{z,z'\in [0,T]\times[0,L]} \left|\frac{w(z)-w(z')}{\rho(z\,,z')\sqrt{\log_+(1/\rho(z\,,z'))}}\right|^2 \right)\right]<\infty.
\end{align}
\end{lemma}

\begin{proof}
Thanks to Lemmas \ref{lem:w-w}, \ref{lem:u-u:x} and \ref{lem:u-u:t}, for any $T>0$, there is $C>0$ such that
\begin{align}\label{u-u}
	\|w(z)-w(z')\|_k \le C \sqrt{k} \, \rho(z\,,z') \quad \text{and}
	\quad
	\|\tilde{u}(z)-\tilde{u}(z')\|_k \le C\sqrt{k}\, \rho(z\,,z')
\end{align}
uniformly for all $k\in [2\,,\infty)$ and $z\,,z' \in [0\,,T]\times[0\,,L]$.
Therefore, \eqref{E:exp:w} and \eqref{E:exp:u} follow from \eqref{u-u} and an appeal to Dudley's metric entropy theorem \cite{Dudley} or the Garsia-Rodemich-Ramsey continuity lemma (see, e.g., \cite[Proposition A.1]{DKN07}).
This is standard, so we omit the details.
\end{proof}

\begin{lemma}\label{lem:linear:x}
If $b$ and $\sigma$ are bounded, then for any $0<a<T$, there exist $C > 0$ and $\epsilon_1 \in (0\,,L)$ such that
\begin{align}\label{E:bd:x}
	\|\mathscr{E}(t\,,x\,;t\,,x')\|_k \le C k |x'-x|^{19/28}
\end{align}
uniformly for all $k\in [2\,,\infty)$ and $(t\,,x)\,,(t\,,x') \in I:=[a\,,T]\times [0\,,L]$ with $|x'-x| \le \epsilon_1$. This remains valid when $I = [0\,,T]\times [c\,,d]$ for fixed $T>0$ and $0 \le c < d \le L$ if  \eqref{G*u:I} holds.
\end{lemma}

\begin{proof}
Let $(t\,,x)\,,(t\,,x') \in I = [a\,,T]\times [0\,,L]$.
Set $\varepsilon=x'-x$.
Write $\mathscr{E}(t\,,x\,;t\,,x')= J_1 + J_2$, where
\begin{align*}
	J_1 &= \int_{(0,t)\times[0,L]} [G_{t-s}(x+\varepsilon\,,y)-G_{t-s}(x\,,y)] b(u(s\,,y)) \, \d s \, \d y,\\
	J_2 &= \int_{(0,t)\times[0,L]} [G_{t-s}(x+\varepsilon\,,y) - G_{t-s}(x\,,y)] \sigma(u(s\,,y)) \, \xi(\d s\, \d y)\\
	& \quad - \sigma(u(t\,,x)) \int_{(0,t)\times[0,L]} [G_{t-s}(x+\varepsilon\,,y) - G_{t-s}(x\,,y)] \, \xi(\d s\, \d y).
\end{align*}
Since $b$ is bounded, we may use \eqref{G}, \eqref{lambda} and \eqref{f_n:bd} to see that for any $\gamma \in (0\,,1)$,
\begin{align*}
	\|J_1\|_k& \lesssim \int_0^t \d s \int_0^L \d y \, |G_{s}(x+\varepsilon\,,y) - G_{s}(x\,,y)| \\
	& \lesssim \int_0^t \d s \sum_{n=1}^\infty (\varepsilon n \wedge 1)\, \e^{-cn^2 s}
	\le \varepsilon^\gamma \int_0^t \d s \sum_{n=1}^\infty n^\gamma \e^{-cn^2 s}\\
	&\lesssim \varepsilon^\gamma \int_0^t \d s \int_0^\infty \d z \, z^\gamma \, \e^{-cz^2 s}
	\lesssim \varepsilon^\gamma \int_0^t s^{-(1+\gamma)/2} \lesssim \varepsilon^\gamma,
\end{align*}
where the implied constants depend on $\gamma$.

In order to estimate $J_2$, we use the idea of localization of heat kernel \cite{FKM15}.
Let $\delta \in (0\,,|\varepsilon|)$ and define
\begin{align*}
	&B = \{ (s\,,y) \in (0\,,t) \times [0\,,L] : t-\delta < s < t\,, |x-y| \le \sqrt{|\varepsilon|}\, \},\\
	&B^c = ((0\,,t) \times [0\,,L]) \setminus B.
\end{align*}
Suppose first $\delta < t$.
Then, we may write $J_2 = J_{2,1} + J_{2,2} + J_{2,3} + J_{2,4}$, where
\begin{align*}
	J_{2,1} &= \iint_B [G_{t-s}(x+\varepsilon\,,y) - G_{t-s}(x\,,y)] [\sigma(u(s\,,y)) - \sigma(u(t-\delta\,,x))] \, \xi(\d s\,\d y),\\
	J_{2,2} &= [\sigma(u(t-\delta\,,x)) - \sigma(u(t\,,x))] \iint_B [G_{t-s}(x+\varepsilon\,,y) - G_{t-s}(x\,,y)] \, \xi(\d s\,\d y),\\
	J_{2,3} &= \iint_{B^c} [G_{t-s}(x+\varepsilon\,,y) - G_{t-s}(x\,,y)]  \sigma(u(s\,,y)) \, \xi(\d s\,\d y),\\
	J_{2,4} &= -\sigma(u(t\,,x)) \iint_{B^c} [G_{t-s}(x+\varepsilon\,,y) - G_{t-s}(x\,,y)] \, \xi(\d s\,\d y).
\end{align*}
Here, we have used the equality
\begin{align*}
	&\sigma(u(t-\delta\,,x)) \iint_B [G_{t-s}(x+\varepsilon\,,y) - G_{t-s}(x\,,y)] \, \xi(\d s\,\d y)\\
	& = \iint_B [G_{t-s}(x+\varepsilon\,,y) - G_{t-s}(x\,,y)] \sigma(u(t-\delta\,,x)) \, \xi(\d s\,\d y),
\end{align*}
which holds because $u(t-\delta\,,x)$ is $\mathscr{F}_{t-\delta}$ measurable and the right-hand side is a well-defined Walsh integral of a predictable process \cite{Walsh}.
By the BDG inequality \cite[Prop. 4.4]{DK}, the Lipschitz continuity of $\sigma$, Lemmas \ref{lem:u-u:x}, \ref{lem:u-u:t}, and Lemma \ref{lem:G*u} (or \eqref{G*u:I} when $I=[0\,,T]\times [c\,,d]$), we have
\begin{align*}
	\|J_{2,1}\|_k^2 
	&\lesssim k \iint_B \d s \, \d y\, [G_{t-s}(x+\varepsilon\,,y) - G_{t-s}(x\,,y)]^2 \|u(s\,,y) - u(t-\delta\,,x)\|_k^2\\
	& \lesssim k^2 \int_{t-\delta}^t \d s \sqrt{s-(t-\delta)} \int_0^L \d y \, \1_{\{|x-y|\le\sqrt{|\varepsilon|}\}} [G_{t-s}(x+\varepsilon\,,y) - G_{t-s}(x\,,y)]^2\\
	& \quad + k^2 \int_{t-\delta}^t \d s \int_0^L \d y\, \1_{\{|x-y|\le\sqrt{|\varepsilon|}\}} [G_{t-s}(x+\varepsilon\,,y) - G_{t-s}(x\,,y)]^2 |x-y|\\
	& \lesssim k^2 \sqrt{|\varepsilon|} \iint_B \d s \, \d y \, [G_{t-s}(x+\varepsilon\,,y) - G_{t-s}(x\,,y)]^2\\
	& \le k^2 \sqrt{|\varepsilon|} \Var(w(t\,,x+\varepsilon)-w(t\,,x)) \lesssim k^2 |\varepsilon|^{3/2}.
\end{align*}
Similarly, by Cauchy-Schwarz inequality,
\begin{align*}
	\|J_{2,2}\|_k^2 
	&\textstyle \le \|u(t-\delta\,,x)-u(t\,,x)\|_{2k}^2 \cdot\|\iint_B [G_{t-s}(x+\varepsilon\,,y) - G_{t-s}(x\,,y)] \, \xi(\d s \, \d y)\|_{2k}^2\\
	& \lesssim k^2 \delta^{1/2} \Var(w(t\,,x+\varepsilon)-w(t\,,x)) \lesssim k^2 |\varepsilon| \delta^{1/2}.
\end{align*}
Next, by the BDG inequality \cite[Prop. 4.4]{DK} and the boundedness of $\sigma$, we have
\begin{align*}
	\|J_{2,3}\|_k^2 \lesssim k \iint_{B^c} \d s \, \d y\, [G_{t-s}(x+\varepsilon\,,y) - G_{t-s}(x\,,y)]^2.
\end{align*}
We estimate the integral by splitting $B^c$ into the union of $B_1$ and $B_2$, where
\begin{align*}
	&B_1 := (0\,,t-\delta] \times [0\,,L],\\
	&B_2 := (t-\delta\,,t) \times \{ y \in [0\,,L] : |x-y| >\sqrt{|\varepsilon|}\,\}.
\end{align*}
By Lemma \ref{lem:int:G:x},
\begin{align*}
	&\iint_{B_1}  \d s \, \d y\, [G_{t-s}(x+\varepsilon\,,y) - G_{t-s}(x\,,y)]^2\\
	&\lesssim \int_0^{t-\delta} \d s \int_0^\infty \d z \, (|\varepsilon|^2 z^2 \wedge 1) \, \e^{-cz^2(t-s)}
	\le |\varepsilon|^2 \int_\delta^{t} \d s \int_0^\infty \d z \, z^2 \e^{-cz^2s}\\
	& = |\varepsilon|^2 \int_\delta^t \frac{\d s}{s^{3/2}} \int_0^\infty \d z \, z^2 \e^{-cz^2} \lesssim |\varepsilon|^2 \int_\delta^\infty \frac{\d s}{s^{3/2}}
	\lesssim |\varepsilon|^{2} \delta^{-1/2}.
\end{align*}
Moreover, if $\epsilon_1>0$ is small enough, then $\sqrt{|\varepsilon|}-|\varepsilon| > \sqrt{|\varepsilon|}/2$ for $|\varepsilon| \le \epsilon_1$, so we may use Lemma \ref{lem:G} to deduce that
\begin{align*}
	&\iint_{B_2} \d s \, \d y\, [G_{t-s}(x+\varepsilon\,,y) - G_{t-s}(x\,,y)]^2\\
	& \lesssim \int_{t-\delta}^t \d s \int_0^L \d y\, \1_{\{|x-y|>\sqrt{|\varepsilon|}\}} \left[ \frac{(t-s)^2}{|x+\varepsilon-y|^6} + \frac{(t-s)^2}{|x-y|^6} \right]\\
	& \lesssim \int_{t-\delta}^t \d s \, (t-s)^2 \left[ \int_{\sqrt{|\varepsilon|}-|\varepsilon|}^\infty \frac{\d y}{y^6} + \int_{\sqrt{|\varepsilon|}}^\infty \frac{\d y}{y^6} \right]\\
	& \lesssim \int_{t-\delta}^t \d s \, (t-s)^2 \left[ \frac{1}{(\sqrt{|\varepsilon|}/2)^5} + \frac{1}{|\varepsilon|^{5/2}} \right]
	\lesssim |\varepsilon|^{-5/2} \delta^3.
\end{align*}
Hence, $\|J_{2,3}\|_k^2 \lesssim k (|\varepsilon|^{2}\delta^{-1/2} + |\varepsilon|^{-5/2}\delta^3)$.
Similarly, 
\begin{align*}
	\|J_{2,4}\|_k^2 \lesssim k \iint_{B^c} \d s \, \d y\, [G_{t-s}(x+\varepsilon\,,y) - G_{t-s}(x\,,y)]^2 \lesssim k (|\varepsilon|^{2}\delta^{-1/2} + |\varepsilon|^{-5/2}\delta^3).
\end{align*}
Combining the above estimates yields
\begin{align*}
	\|J_2\|_k 
	&\le \|J_{2,1}\|_k + \|J_{2,2}\|_k + \|J_{2,3}\|_k + \|J_{2,4}\|_k\\
	& \lesssim k \left[ |\varepsilon|^{3/4} +  |\varepsilon|^{1/2} \delta^{1/4} + |\varepsilon|\delta^{-1/4} + |\varepsilon|^{-5/4} \delta^{3/2} \right].
\end{align*}
Choose $\delta = |\varepsilon|^{9/7}$ to optimize this bound and deduce that if $t>\delta =|\varepsilon|^{9/7}$, then
\begin{align*}
	\|J_2\|_k \lesssim k \left[ |\varepsilon|^{3/4} + |\varepsilon|^{23/28} + |\varepsilon|^{19/28} + |\varepsilon|^{19/28} \right] \lesssim k |\varepsilon|^{19/28}.
\end{align*}
Combine the estimates for $J_1$ and $J_2$ to obtain the desired estimate \eqref{E:bd:x}.
Finally, if $t \le \delta= |\varepsilon|^{9/7}$, then the estimate for $J_1$ is still valid, whereas for $J_2$, by considering
\begin{align*}
	&B= \{(s\,,y) \in (0\,,t) \times [0\,,L] : |x-y| \le \sqrt{|\varepsilon|}\,\} \text{ and } B^c = B_1 \cup B_2,\\
	&\text{where }  B_1 = \varnothing \text{ and } B_2= \{(s\,,y) \in (0\,,t) \times [0\,,L] : |x-y| > \sqrt{|\varepsilon|}\,\},
\end{align*}
it is not hard to derive the same form of estimates for $J_{2,1}\,,\dots,J_{2,4}$.
Again, we obtain the desired estimate.
\end{proof}

\begin{lemma}\label{lem:linear:t}
If $b$ and $\sigma$ are bounded, then for any $0<a<T$, there is $C>0$ such that
\begin{align}\label{E:bd:t}
	\|\mathscr{E}(t\,,x\,;t',x)\|_k \le C k |t'-t|^{19/48}
\end{align}
uniformly for all $k\in [2\,,\infty)$ and $(t\,,x) \,, (t',x) \in I:=[a\,, T]\times [0\,,L]$ with $|t'-t|\le1$.
This remains valid when $I=[0\,,T]\times [c\,,d]$ for fixed $T>0$ and $0\le c < d \le L$ if \eqref{G*u:I} holds.
\end{lemma}

\begin{proof}
Let $(t\,,x) \,, (t',x) \in I=[a\,, T]\times [0\,,L]$ with $|t'-t|\le1$.
Suppose first $t\le t'$. Set $\varepsilon = t'-t$.
We use \eqref{she:mild} and \eqref{E} to write $\mathscr{E}(t\,,x\,;t',x) =I_1 + I_2 + I_3 + I_4$, where
\begin{align*}
	I_1 &= \int_{(t,t+\varepsilon)\times[0,L]} G_{t+\varepsilon-s}(x\,,y) b(u(s\,,y)) \, \d s \, \d y,\\
	I_2 &= \int_{(0,t) \times [0,L]} [G_{t+\varepsilon-s}(x\,,y) - G_{t-s}(x\,,y)] b(u(s\,,y)) \, \d s \, \d y,\\
	I_3 &= \int_{(t,t+\varepsilon)\times[0,L]} G_{t+\varepsilon-s}(x\,,y) [\sigma(u(s\,,y))-\sigma(u(t\,,x))]\, \xi(\d s\,\d y),\\
	I_4 &= \int_{(0,t)\times[0,L]} [G_{t+\varepsilon-s}(x\,,y) - G_{t-s}(x\,,y)] \sigma(u(s\,,y))\, \xi(\d s\, \d y)\\
	& \quad - \sigma(u(t\,,x)) \int_{(0,t)\times[0,L]} [G_{t+\varepsilon-s}(x\,,y) - G_{t-s}(x\,,y)]\, \xi(\d s\, \d y).
\end{align*}
Since $b$ is bounded, we may use Minkowski's inequality and Lemma \ref{lem:G} to see that
\begin{align*}
	\|I_1\|_k \le \int_t^{t+\varepsilon} \frac{\d s}{\sqrt{t+\varepsilon-s}} \lesssim \varepsilon^{1/2}.
\end{align*}
Similarly, we may use \eqref{G}, \eqref{lambda}, and the elementary inequality $\e^{-s}-\e^{-t}\le \e^{-s}((t-s)\wedge 1)$ for $0<s<t$ to deduce the following:
\begin{align*}
	\|I_2\|_k &\le \int_0^t \d s \int_0^L \d y\, |G_{t+\varepsilon-s}(x\,,y) - G_{t-s}(x\,,y)|\\
	& \lesssim \int_0^t \d s\sum_{n=1}^\infty |\e^{-\lambda_n (t+\varepsilon-s)} - \e^{-\lambda_n (t-s)}|
	\lesssim \int_0^t \d s \int_0^\infty \d z \,  (\varepsilon z \wedge 1)\, \e^{-cz^2(t-s)} \\
	& \lesssim \varepsilon^{1/2} \int_0^t \d s \int_0^\infty \d z \, \sqrt{z} \, \e^{-cz^2 s} \lesssim \varepsilon^{1/2} \int_0^t s^{-3/4} \d s \lesssim \varepsilon^{1/2}.
\end{align*}

In order to estimate $I_3$ and $I_4$, we use again the idea of localization of heat kernel. Let $c \in [0\,,1/2]$.
By the BDG inequality \cite[Prop. 4.4]{DK}, the Lipschitz continuity of $\sigma$, Lemmas \ref{lem:sup:u}, \ref{lem:u-u:x}, \ref{lem:u-u:t}, and Lemma \ref{lem:G*u} (or \eqref{G*u:I} when $I=[0\,,T]\times [c\,,d]$),
\begin{align*}
	\|I_3\|_k^2 
	&\lesssim k^2 \int_t^{t+\varepsilon} \d s \int_0^L \d y \, \1_{\{|x-y| \le (t+\varepsilon-s)^{1/2-c}\}} G^2_{t+\varepsilon-s}(x\,,y) \sqrt{s-t}\\
	& + k^2 \int_t^{t+\varepsilon} \d s \int_0^L \d y \, \1_{\{|x-y| \le (t+\varepsilon-s)^{1/2-c}\}} G^2_{t+\varepsilon-s}(x\,,y) |x-y|\\
	& + k^2 \int_t^{t+\varepsilon} \d s \int_0^L \d y \, \1_{\{|x-y| > (t+\varepsilon-s)^{1/2-c}\}} G^2_{t+\varepsilon-s}(x\,,y) =: k^2 [ I_{3,1} + I_{3,2} + I_{3,3} ].
\end{align*}
Thanks to Parseval's identity, \eqref{lambda}, and \eqref{f_n:bd}, we have
\begin{align*}
	I_{3,1} & = \int_0^{\varepsilon} \d s \sqrt{\varepsilon-s} \int_0^L \d y \, G_{s}^2(x\,,y) = \int_0^{\varepsilon}\d s \sqrt{\varepsilon-s} \sum_{n=1}^\infty \e^{-2\lambda_n s} |f_n(x)|^2\\
	& \lesssim \sqrt{\varepsilon}\int_0^\varepsilon\d s  \sum_{n=1}^\infty \e^{-cn^2 s}
	\lesssim \sqrt{\varepsilon} \int_0^\varepsilon \d s \int_0^\infty \d z \, \e^{-cz^2 s}
	\lesssim \sqrt{\varepsilon}\int_0^\varepsilon \frac{\d s}{\sqrt{s}} \lesssim \varepsilon.
\end{align*}
By similar computations,
\begin{align*}
	I_{3,2} &\lesssim \int_0^{\varepsilon} \d s \, s^{1/2-c} \int_0^L \d y \, G_s^2(x\,,y) \lesssim \varepsilon^{1-c}.
\end{align*}
By Lemma \ref{lem:G}, if $|x-y| > (t+\varepsilon-s)^{1/2-c}$, then
\begin{align}\label{G:bd:local}
	&|G_{t+\varepsilon-s}(x\,,y)| \lesssim \frac{t+\varepsilon-s}{|x-y|^3}= \frac{1}{|x-y|} \cdot \frac{t+\varepsilon-s}{|x-y|^2} \le \frac{(t+\varepsilon-s)^{2c}}{|x-y|}
\end{align}
and hence
\begin{align*}
	I_{3,3} \lesssim \int_t^{t+\varepsilon} \d s \, (t+\varepsilon-s)^{4c} \int_{(t+\varepsilon-s)^{1/2-c}}^\infty \frac{\d y}{y^2}
	\lesssim \int_t^{t+\varepsilon} \frac{\d s}{(t+\varepsilon-s)^{1/2-5c}} \lesssim \varepsilon^{1/2+5c}.
\end{align*}
Choose $c = 1/12$ and combine the estimates to find that $\|I_3\|_k \lesssim k^2 \varepsilon^{11/24}$.

To estimate $I_4$, let $\delta = \varepsilon^b$, where $b\in (0\,,1)$, let $\gamma \in [0\,,1/2]$, and define
\begin{align*}
	&A = \{ (s\,,y) \in (0\,,t) \times [0\,,L] : t-\delta < s < t \,, |x-y| \le (t+\varepsilon-s)^{1/2-\gamma}\},\\
	&A^c = ((0\,,t) \times [0\,,L])\setminus A.
\end{align*}
Then, we may write $I_4 = I_{4,1} + I_{4,2}+ I_{4,3} + I_{4,4}$, where
\begin{align*}
	I_{4,1} &= \iint_{A} [G_{t+\varepsilon-s}(x\,,y) - G_{t-s}(x\,,y)] [\sigma(u(s\,,y))-\sigma(u(t-\delta\,,x))] \, \xi(\d s\, \d y),\\
	I_{4,2} &= [\sigma(u(t-\delta\,,x)) - \sigma(u(t\,,x))] \iint_A [G_{t+\varepsilon-s}(x\,,y) - G_{t-s}(x\,,y)] \, \xi(\d s\, \d y),\\
	I_{4,3} &= \iint_{A^c} [G_{t+\varepsilon-s}(x\,,y) - G_{t-s}(x\,,y)] \sigma(u(s\,,y)) \, \xi(\d s\,\d y),\\
	I_{4,4}& = - \sigma(u(t\,,x)) \iint_{A^c} [G_{t+\varepsilon-s}(x\,,y) - G_{t-s}(x\,,y)] \, \xi(\d s\,\d y).
\end{align*}
Suppose that $\delta < t$.
By the BDG inequality \cite[Prop. 4.4]{DK}, the Lipschitz continuity of $\sigma$, Lemmas \ref{lem:u-u:x}, \ref{lem:u-u:t}, and Lemma \ref{lem:G*u} (or \eqref{G*u:I} when $I=[0\,,T]\times [c\,,d]$),
\begin{align*}
	&\|I_{4,1}\|_k^2
	\lesssim k \iint_A \d s\, \d y \,[G_{t+\varepsilon-s}(x\,,y) - G_{t-s}(x\,,y)]^2 \|u(s\,,y)-u(t-\delta\,,x)\|_k^2\\
	& \lesssim k^2 \int_{t-\delta}^t \d s \sqrt{s-(t-\delta)} \int_0^L \d y \,\1_{\{|x-y|\le(t+\varepsilon-s)^{1/2-\gamma}\}}[G_{t+\varepsilon-s}(x\,,y) - G_{t-s}(x\,,y)]^2  \\
	& \quad + k^2 \int_{t-\delta}^t \d s \int_0^L \d y \,\1_{\{|x-y|\le(t+\varepsilon-s)^{1/2-\gamma}\}}[G_{t+\varepsilon-s}(x\,,y) - G_{t-s}(x\,,y)]^2 |x-y| \\
	& \le k^2 (\sqrt{\delta}+(\varepsilon+\delta)^{1/2-\gamma}) \iint_A \d s\, \d y \,[G_{t+\varepsilon-s}(x\,,y) - G_{t-s}(x\,,y)]^2\\
	&\lesssim k^2 \delta^{1/2-\gamma} \Var(w(t+\varepsilon\,,x)-w(t\,,x)) \lesssim k^2 \delta^{1/2-\gamma} \varepsilon^{1/2}.
\end{align*}
By Cauchy-Schwarz inequality,
\begin{align*}
	\|I_{4,2}\|_k^2
	&\textstyle\lesssim \|u(t-\delta\,,x)-u(t\,,x)\|_{2k}^2 \cdot \|\iint_A [G_{t+\varepsilon-s}(x\,,y) - G_{t-s}(x\,,y)] \, \xi(\d s\,\d y)\|_{2k}^2\\
	&\lesssim k^2 \delta^{1/2} \Var(w(t+\varepsilon\,,x)-w(t\,,x)) \lesssim k^2 \delta^{1/2} \varepsilon^{1/2}.
\end{align*}
Next, by the BDG inequality \cite[Prop. 4.4]{DK} and the boundedness of $\sigma$,
\begin{align*}
	\|I_{4,3}\|_k^2 \lesssim k \iint_{A^c} \d s\, \d y \,[G_{t+\varepsilon-s}(x\,,y) - G_{t-s}(x\,,y)]^2.
\end{align*}
Split $A^c$ into the union of $A_1$ and $A_2$, where
\begin{align*}
	&A_1:=(0\,,t-\delta]\times[0\,,L],\\
	&A_2:=(t-\delta\,,t) \times \{ y \in [0\,,L] : |x-y| > (t+\varepsilon-s)^{1/2-\gamma} \}.
\end{align*}
By Lemma \ref{lem:int:G:t},
\begin{align*}
	&\iint_{A_1} \d s \, \d y\, [G_{t+\varepsilon-s}(x\,,y) - G_{t-s}(x\,,y)]^2\\
	& \lesssim \int_0^{t-\delta} \d s \int_0^\infty \d z \, (\varepsilon^2 z^4 \wedge 1)\, \e^{-cz^2(t-s)}\\
	&\le \varepsilon^2 \int_{\delta}^t \d s \int_0^\infty \d z \, z^4\, \e^{-cz^2 s}
	\lesssim \varepsilon^2 \int_\delta^\infty \frac{\d s}{s^{5/2}} \lesssim \varepsilon^2 \delta^{-3/2}.
\end{align*}
Using $|G_{t+\varepsilon-s}(x\,,y) - G_{t-s}(x\,,y)| \le |G_{t+\varepsilon-s}(x\,,y)| + |G_{t-s}(x\,,y)|$ and a similar bound to the one in \eqref{G:bd:local}, we have
\begin{align*}
	&\iint_{A_2} \d s \, \d y\, [G_{t+\varepsilon-s}(x\,,y) - G_{t-s}(x\,,y)]^2\\
	& \lesssim \int_{t-\delta}^t \d s \int_0^L \d y \, \1_{\{|x-y|>(t+\varepsilon-s)^{1/2-\gamma}\}} \frac{(t+\varepsilon-s)^{4\gamma}}{|x-y|^2} \\
	& \lesssim \delta^{4\gamma} \int_{t-\delta}^t \d s \int_{(t+\varepsilon-s)^{1/2-\gamma}}^\infty \frac{\d y}{y^2} \lesssim \delta^{4\gamma} \int_{t-\delta}^t \frac{\d s}{(t+\varepsilon-s)^{1/2-\gamma}}\\
	& \lesssim \delta^{4\gamma} (\varepsilon+\delta)^{1/2+\gamma} \le \delta^{1/2 + 5\gamma}.
\end{align*}
Hence, $\|I_{4,3}\|_k^2 \lesssim k [\varepsilon^2\delta^{-3/2} +\delta^{1/2+5\gamma}]$.
Similarly, by the boundedness of $\sigma$,
\begin{align*}
	\|I_{4,4}\|_k^2 \lesssim k \iint_{A^c} \d s\, \d y \,[G_{t+\varepsilon-s}(x\,,y) - G_{t-s}(x\,,y)]^2 \lesssim k \left[\varepsilon^2\delta^{-3/2} + \delta^{1/2+5\gamma}\right].
\end{align*}
It is not hard to check that $I_{4,1}\,,\dots,I_{4,4}$ have the same form of estimates when $t<\delta$. Therefore,
\begin{align*}
	\|I_4\|_k &\le \|I_{4,1}\|_k + \|I_{4,2}\|_k + \|I_{4,3}\|_k + \|I_{4,4}\|_k\\
	&\lesssim k \left[ \delta^{1/4-\gamma/2}\varepsilon^{1/4} + \delta^{1/4}\varepsilon^{1/4} + \delta^{-3/4} \varepsilon + \delta^{1/4+5\gamma/2} \right].
\end{align*}
Recall that $\delta = \varepsilon^b$.
Choose $b=3/4$ and $\gamma = 1/9$ to obtain
\begin{align*}
	\|I_4\|_k \lesssim k \left[ \varepsilon^{19/48} + \varepsilon^{7/16} +   \varepsilon^{7/16} +  \varepsilon^{19/48} \right] \lesssim k \varepsilon^{19/48}
\end{align*}
uniformly for all $k \in [2\,,\infty)$, $x \in [0\,,L]$ and $t \le t'$ in $I$.
Combine the estimates for $I_1\,,\dots,I_4$ to obtain the desired estimate \eqref{E:bd:t}.

Finally, to prove the desired estimate for $t'<t$, note that this is the same as proving that $\mathscr{E}(t',x\,; t\,,x)$ satisfies the desired estimate for $t<t'$. But this can be shown by observing that
\begin{align*}
	\mathscr{E}(t',x\,; t\,,x)
	= - \mathscr{E}(t\,,x\,; t',x) + [\sigma(u(t\,,x)) - \sigma(u(t+\varepsilon\,,x))] [w(t+\varepsilon\,,x) - w(t\,,x)],
\end{align*}
applying the estimate for $\mathscr{E}(t\,,x\,; t',x)$ from the first part of this proof,
and using Cauchy-Schwarz inequality, Lipschitz continuity of $\sigma$, and Lemma \ref{lem:u-u:t}, which yields
\begin{align*}
	&\|[\sigma(u(t\,,x)) - \sigma(u(t+\varepsilon\,,x))] [w(t+\varepsilon\,,x) - w(t\,,x)]\|_k\\
	& \lesssim \|u(t\,,x)-u(t+\varepsilon\,,x)\|_{2k} \cdot \|w(t+\varepsilon\,,x) - w(t\,,x)\|_{2k}
	\lesssim k \varepsilon^{1/2}.
\end{align*}
This completes the proof.
\end{proof}

\begin{proof}[Proof of Proposition \ref{pr:linear}]
Thanks to Lemma \ref{lem:sup:u}, it suffices to show \eqref{E:bd} uniformly for all $k \in [2\,,\infty)$ and $(t\,,x)\,,(t',x')\in I$ with $\rho((t\,,x)\,,(t',x')) \le \epsilon_0$, where $\epsilon_0>0$ is a small but fixed number.
Observe that
\begin{align}\begin{split}\label{E+E}
	\mathscr{E}(t\,,x\,;t',x') &= \mathscr{E}(t\,,x';t',x')+\mathscr{E}(t\,,x\,;t\,,x')\\
	& \quad + (\sigma(u(t\,,x'))-\sigma(u(t\,,x))) (w(t',x')-w(t\,,x')).
\end{split}\end{align}
Also, by Cauchy-Schwarz inequality and Lemmas \ref{lem:u-u:x} and \ref{lem:u-u:t},
\[
	\|(\sigma(u(t\,,x'))-\sigma(u(t\,,x))) (w(t',x')-w(t\,,x'))\|_k 
	\lesssim k [\rho((t\,,x)\,,(t',x'))]^2.
\]
This and Lemmas \ref{lem:linear:x} and \ref{lem:linear:t} conclude the proof since $\min\{19/14\,, 19/12\}>1$.
\end{proof}

\subsection{Tail probability and almost sure bounds}

\begin{lemma}\label{lem:EexpE}
Let $\zeta>1$ be the number given by Proposition \ref{pr:linear}.
If $b$ and $\sigma$ are bounded, then for any $0<a<T$, there is $\gamma_1>0$ such that
\begin{align*}
	\sup_{z,z'\in I}\E\left[ \exp\left( \gamma_1\frac{|\mathscr{E}(z\,;z')|}{[\rho(z\,,z')]^{\zeta}}\right)\right] < \infty,
\end{align*}
where $I=[a\,,T]\times [0\,,L]$ (or $I=[0\,,T]\times[c\,,d]$ with $0\le c<d\le L$ if \eqref{G*u:I} holds).
\end{lemma}

\begin{proof}
Thanks to Proposition \ref{pr:linear}, the series expansion of the  exponential function, and Stirling's formula, there exists $C>0$ such that for all $z\,,z'\in I$,
\begin{align*}
	\E\left[ \exp\left( \gamma_1\frac{|\mathscr{E}(z\,;z')|}{[\rho(z\,,z')]^{\zeta}}\right)\right]
	=\sum_{k=0}^\infty \frac{\gamma_1^k}{k!} \frac{\|\mathscr{E}(z\,;z')\|_k^k}{[\rho(z\,,z')]^{k\zeta}}
	\le \sum_{k=0}^\infty \gamma_1^k C^k.
\end{align*}
The last quantity remains bounded provided $\gamma_1>0$ is small enough.
\end{proof}

\begin{proposition}\label{pr:P:supE}
Let $\zeta>1$ be given by Proposition \ref{pr:linear}.
If $b$ and $\sigma$ are bounded, then for any fixed $0<a<T$ and $p \in (0\,,\zeta]$, there exists $C>0$ such that
\begin{align}\label{P:supE:bd}
	\P\left\{ \sup_{\substack{z,z'\in I: \rho(z,z') \le \varepsilon}} |\mathscr{E}(z\,;z')| > h \varepsilon^p \right\}
	\le C \varepsilon^{-6(p+\zeta)} \exp\left( - \frac{h \wedge h^2}{C \varepsilon^{\zeta-p}\log_+(1/\varepsilon)} \right)
\end{align}
uniformly for all $\varepsilon \in (0\,,1]$ and $h>0$, where $I=[a\,,T]\times [0\,,L]$ (or $I=[0\,,T]\times[c\,,d]$ with $0\le c<d\le L$ if \eqref{G*u:I} holds).
\end{proposition}

\begin{proof}
Define $L_\sigma = \sup_{u,v\in \R}|\sigma(u)-\sigma(v)|/|u-v|$ and $M_\sigma = \sup_{u \in \R}|\sigma(u)|$.
Let $h>0$ and $\varepsilon\in (0\,,1]$. 
The proof uses an interpolation argument.
Let $\delta\in (0\,,\varepsilon]$ be a number to be determined, and define
\begin{align*}
	J=\left\{ (t\,,x) \in I \, : \, \exists\, k_1\,,k_2 \in \N_+\,, t = k_1 \delta^{4} \text{ and } x = k_2 \delta^{2} \right\}.
\end{align*}
Let $A$ denote the event appearing on the left-hand side of \eqref{P:supE:bd}.
Consider the events $B_0$ and $B_1$ defined by
\begin{gather*}
	B_0=\left\{ \max_{q,q' \in J: \rho(q,q') \le 3\varepsilon}|\mathscr{E}(q\,;q')| > \frac{h\varepsilon^p}{2} \right\}\quad\text{and}\quad
	B_1=B_2 \cap B_3 \cap B_4,
\end{gather*}
where
\begin{align*}
	B_2&=\left\{ \forall q \in J\,, \sup_{q':\rho(q,q')\le \delta}|\tilde{u}(q)-\tilde{u}(q')| \le \frac{(\sqrt{h} \wedge h)\varepsilon^p}{2(2+2M_\sigma+L_\sigma)} \right\},\\
	B_3&=\left\{ \forall q \in J\,, \sup_{q': \rho(q,q') \le \delta} |w(q)-w(q')| \le \frac{(\sqrt{h} \wedge h)\varepsilon^p}{2(2+2M_\sigma+L_\sigma)} \right\},\\
	B_4&=\left\{ \sup_{z,z'\in I: \rho(z,z') \le \varepsilon} |w(z)-w(z')| \le \sqrt{h} \wedge h \right\}
\end{align*}
Suppose that $A$ and $B_1$ both occur.
Then, in particular, there exist $z\,,z'\in I$ with $\rho(z\,,z')\le \varepsilon$ such that $|\mathscr{E}(z\,;z')|>h\varepsilon^p$.
For any $q\,,q'\in J$,
\begin{align*}
	\mathscr{E}(q\,;q') &= \mathscr{E}(z\,;z') + \tilde{u}(q) - \tilde{u}(z) - \tilde{u}(q') + \tilde{u}(z') - \sigma(u(q)) (w(q')-w(z'))\\
	&\quad +\sigma(u(q))(w(q)-w(z)) - [\sigma(u(q)) - \sigma(u(z))] (w(z')-w(z)),
\end{align*}
so triangle inequality implies that
\begin{align*}
	|\mathscr{E}(q\,;q')| &\ge |\mathscr{E}(z\,;z')| - |\tilde{u}(q) - \tilde{u}(z)| - |\tilde{u}(q') - \tilde{u}(z')| - M_\sigma |w(q')-w(z')| \\
	& \quad - M_\sigma |w(q)-w(z)| - L_\sigma |u(q)-u(z)| |w(z')-w(z)|.
\end{align*}
Now, if we take $q \in J$ to be the closest point to $z$ and $q' \in J$ to be the closest point to $z'$, then $\rho(q\,,q') \le \rho(q\,,z) + \rho(z\,,z') + \rho(z',q') \le \delta+\varepsilon+\delta \le 3 \varepsilon$, and since $B_1$ occurs, it follows that
\begin{align*}
	|\mathscr{E}(q\,;q')| \ge h\varepsilon^p - \frac{(2+2M_\sigma+L_\sigma) h \varepsilon^p}{2(2+2M_\sigma+L_\sigma)} = \frac{h\varepsilon^p}{2}.
\end{align*}
This shows that $A \cap B_1 \subset B_0$, hence
\begin{align*}
	\P\{A\} = \P\{A \cap B_1 \} + \P\{A \cap B_1^c\}
	\le \P\{B_0\} + \P\{B_1^c\}.
\end{align*}
Set $\delta=\varepsilon^r$, where $r \in [p\,,\zeta]$.
Then, by a union bound, Chebyshev's inequality, Lemma \ref{lem:EexpE},  and $\# J \lesssim \delta^{-6}$, there exists $C_1>0$ such that
\begin{align*}
	\P\{B_0\} &\le (\# J)^2 \sup_{q,q'\in J: \rho(q,q')\le3\varepsilon} \P\left\{ \frac{|\mathscr{E}(q\,;q')|}{[\rho(q\,,q')]^{\zeta}} > \frac{h\varepsilon^p}{2(3\varepsilon)^{\zeta}} \right\}\\
	& \le C_1 \delta^{-12} \exp\left( - \frac{h \varepsilon^p}{C_1 \varepsilon^{\zeta}} \right) = C_1 \varepsilon^{-12r} \exp\left( -\frac{h}{C_1 \varepsilon^{\zeta-p}} \right).
\end{align*}
Similarly, thanks to Lemma \ref{lem:E:exp}, there exists $C_2>0$ such that
\begin{align*}
	&\P\{B_1^c\} 
	\le \P\{B_2^c\} + \P\{B_3^c\} + \P\{B_4^c\}\\
	& \lesssim \delta^{-6} \exp\left( - \frac{(h \wedge h^2) \varepsilon^{2p}}{C_2 \delta^2\log_+(\frac{1}{\delta})} \right) + \delta^{-6} \exp\left( - \frac{(h \wedge h^2) \varepsilon^{2p}}{C_2 \delta^2\log_+(\frac{1}{\delta})} \right) + \exp\left( - \frac{h \wedge h^2}{C_2\varepsilon^2\log_+(\frac{1}{\varepsilon})} \right)\\
	& \lesssim \varepsilon^{-6r} \exp\left( - \frac{h \wedge h^2}{C_2 \zeta \varepsilon^{2(r-p)}\log_+(1/\varepsilon)} \right) + \exp\left( - \frac{h \wedge h^2}{C_2\varepsilon^2\log_+(1/\varepsilon)} \right).
\end{align*}
We may optimize by choosing $r = (p+\zeta)/2$ so that $2(r-p)=\zeta-p$.
Then, combining the last two displays, we see that there exists $C>0$ such that
\begin{align*}
	\P\{A\} \le C \varepsilon^{-12 r} \exp\left( - \frac{h \wedge h^2}{C\varepsilon^{\zeta-p}\log_+(1/\varepsilon)} \right).
\end{align*}
This completes the proof of \eqref{P:supE:bd}.
\end{proof}

\begin{proposition}\label{pr:E:as}
Let $\zeta>1$ be the number given by Proposition \ref{pr:linear}.
Regardless of whether or not $b$ and $\sigma$ are bounded, for any fixed $p \in (0\,,\zeta)$ and fixed $0<a<T$,
\begin{align*}
	\lim_{\varepsilon\to0^+}\sup_{z,z'\in I: 0<\rho(z,z')\le\varepsilon} \frac{|\mathscr{E}(z\,;z')|}{[\rho(z\,,z')]^p} = 0 \qquad \text{a.s.}
\end{align*}
where $I=[a\,,T]\times [0\,,L]$ (or $I=[0\,,T]\times[c\,,d]$ with $0\le c<d\le L$ if \eqref{G*u:I} holds).
\end{proposition}

\begin{proof}
We prove the proposition using a truncation and stopping time argument.
Fix $p \in (0\,,\zeta)$ and $T>0$.
For each $N>0$, define $b_N\,, \sigma_N : \R \to \R$ by
\[
	b_N(x) = \begin{cases}
	b(N) & \text{if $x>N$,}\\
	b(x) & \text{if $-N\le x \le N$,}\\
	b(-N) & \text{if $x<-N$,}
	\end{cases}\quad 
	\sigma_N(x) = \begin{cases}
	\sigma(N) & \text{if $x>N$,}\\
	\sigma(x) & \text{if $-N\le x \le N$,}\\
	\sigma(-N) & \text{if $x<-N$.}
	\end{cases}
\]
Define $u_N$ as the solution to \eqref{she} but with $b$ and $\sigma$ replaced by $b_N$ and $\sigma_N$, respectively.
Define $\mathscr{E}_N$ the same as $\mathscr{E}$ in \eqref{E} but with $u$ replaced by $u_N$.
Let
\[\textstyle
	\tau_N = \inf\{ t \ge 0 : \sup_{x \in [0,L]}|u_N(t\,,x)| > N \}
\]
with $\inf \varnothing = \infty$.
Then $\tau_N$ is a stopping time with respect to the filtration $\{\mathscr{F}_t\}_{t\ge 0}$ generated by the noise $\xi$.
Uniqueness of the solution to \eqref{she} implies that
\begin{align}\label{u_N=u}
	\P\{ u_N(t\,,x) = u(t\,,x) \text{ for all $t<\tau_N$ and $x \in [0\,,L]$} \} = 1.
\end{align}
Fix $N>0$ and $\delta \in (0\,,1)$.
Proposition \ref{pr:P:supE} implies that for any $n \in \N_+$,
\begin{align*}
	\P\left\{ \sup_{\substack{z,z'\in I: 2^{-n-1} \le \rho(z,z') \le 2^{-n}}} |\mathscr{E}_N(z\,;z')| > \delta 2^{-p n} \right\} \le C 2^{6(p+\zeta)n} \exp\left( - \frac{\delta^2 2^{(\zeta-p)n}}{Cn} \right).
\end{align*}
It follows by the Borel-Cantelli lemma that
\begin{align*}
	\lim_{n\to\infty} \sup_{\substack{z,z'\in I: 0<\rho(z,z')\le 2^{-n}}} \frac{|\mathscr{E}_N(z\,;z')|}{[\rho(z\,,z')]^p} \le \delta 2^p \qquad \text{a.s.}
\end{align*}
By monotonicity, this implies that
\begin{align*}
	\lim_{\varepsilon\to0^+} \sup_{\substack{z,z'\in I: 0<\rho(z,z')\le\varepsilon}} \frac{|\mathscr{E}_N(z\,;z')|}{[\rho(z\,,z')]^p} \le \delta 2^p \qquad \text{a.s.}
\end{align*}
Letting $\delta \to 0^+$ yields
\begin{align*}
	\lim_{\varepsilon\to0^+} \sup_{\substack{z,z'\in I: 0<\rho(z,z')\le\varepsilon}} \frac{|\mathscr{E}_N(z\,;z')|}{[\rho(z\,,z')]^p}= 0 \qquad \text{a.s.}
\end{align*}
Thanks to \eqref{u_N=u}, for every $N>0$, we have
\begin{align*}
	\P\left\{ \lim_{\varepsilon\to0^+} \sup_{\substack{z,z'\in I: 0<\rho(z,z')\le\varepsilon}} \frac{|\mathscr{E}(z\,;z')|}{[\rho(z\,,z')]^p}= 0 \right\} \ge \P\{\tau_N>T\}.
\end{align*}
Finally, we may finish the proof by letting $N\to\infty$ because the a.s. continuity of $u$ (see Lemma \ref{lem:E:exp}) together with \eqref{u_N=u} implies that $\lim_{N\to\infty}\P\{\tau_N>T\} =1$.
\end{proof}

\section{Proofs of the main results}\label{s:pf}

\subsection{Proof of Theorem \ref{th:she:lil}}

\begin{proof}
Recall the linearization error $\mathscr{E}(t\,,x\,;t',x')$ defined in \eqref{E}.
By triangle inequality, for any $z\,,z' \in [0\,,\infty)\times [0\,,L]$,
\begin{align}\begin{split}\label{u:w:tri}
	&|\sigma(u(z))| |w(z')-w(z)| - |(G\ast u_0)(z') - (G \ast u_0)(z)| - |\mathscr{E}(z\,;z')|\\
	& \le  |u(z')-u(z)|\\
	&\le |\sigma(u(z))| |w(z')-w(z)| + |(G\ast u_0)(z') - (G \ast u_0)(z)| + |\mathscr{E}(z\,;z')|.
\end{split}\end{align}
Fix $z_0 = (t_0\,,x_0) \in (0\,,\infty)\times[0\,,L]$ and write 
\[
	\phi(z\,,z')=\rho(z\,,z')\sqrt{\log\log(1/\rho(z\,,z'))}.
\]
Thanks to Lemma \ref{lem:G*u}, there exists $K_0>0$ such that for all $z = (t\,,x) \in B_\rho(z_0\,,\varepsilon)$,
\begin{align}\label{G*u:z-z}
	|(G\ast u_0)(z) - (G \ast u_0)(z_0)| \le K_0(|t-t_0|+|x-x_0|) \le K_0(\varepsilon^4 + \varepsilon^2)
\end{align}
and hence
\begin{align}\label{u:lil:G*u}
	\lim_{\varepsilon\to0^+} \sup_{z\in B^*_\rho(z_0,\varepsilon)} \frac{|(G\ast u_0)(z) - (G \ast u_0)(z_0)|}{\phi(z\,,z_0)} = 0.
\end{align}
By Proposition \ref{pr:E:as},
\[
	\lim_{\varepsilon\to 0^+} \sup_{z\in B^*_\rho(z_0,\varepsilon)}\frac{|\mathscr{E}(z_0\,;z)|}{\phi(z\,,z_0)}=0 \quad \text{a.s.}
\]
It follows from \eqref{u:w:tri} and the last two displays that, a.s.,
\begin{align*}
	\lim_{\varepsilon\to0^+} \sup_{z\in B^*_\rho(z_0,\varepsilon)} \frac{|u(z)-u(z_0)|}{\phi(z\,,z_0)}
	= |\sigma(u(z_0))| \lim_{\varepsilon\to0^+} \sup_{z\in B^*_\rho(z_0,\varepsilon)} \frac{|w(z)-w(z_0)|}{\phi(z\,,z_0)}.
\end{align*}
Owing to \eqref{w:lil} in Theorem \ref{th:w:lil:mc}, the right-hand side is equal to $|\sigma(u(z_0))| K_0$ a.s.

Finally, when $t_0=0$, \eqref{u:lil:G*u} still holds under the additional assumption \eqref{G*u:z_0}.
Moreover, Proposition \ref{pr:E:as} and \eqref{w:lil} in Theorem \ref{th:w:lil:mc} continue to hold when $t_0=0$.
This again shows \eqref{u:lil} and completes the proof of Theorem \ref{th:she:lil}.
\end{proof}

\subsection{Proof of Theorem \ref{th:she:mc}}

\begin{proof}
Fix $I=[a\,,T]\times[c\,,d]$ as in the statement of the theorem.
Write 
\[
	\psi(z\,,z') = \rho(z\,,z')\sqrt{\log(1/\rho(z\,,z'))}.
\]
By the polarity condition, $\sigma(u(z)) \ne 0$ for all $z\in I$.
But since $u$ is a.s.~continuous on the compact set $I$, it follows that $\Delta:=\inf_{z \in I}|\sigma(u(z))|$ is an a.s.~strictly positive random variable.
With this in mind, we begin with \eqref{u:w:tri}, which implies
\begin{align*}
	&|w(z')-w(z)| - \tfrac{1}{\Delta}|(G\ast u_0)(z') - (G \ast u_0)(z)|-\tfrac{1}{\Delta}|\mathscr{E}(z\,;z')|\\
	&\le \frac{|u(z')-u(z)|}{|\sigma(u(z))|}\\
	&\le |w(z')-w(z)| + \tfrac{1}{\Delta}|(G\ast u_0)(z') - (G \ast u_0)(z)|+\tfrac{1}{\Delta}|\mathscr{E}(z\,;z')|.
\end{align*}
By Lemma \ref{lem:G*u},
\begin{align}\label{u:mc:G*u}
	\lim_{\varepsilon\to0^+} \sup_{z,z'\in I: 0< \rho(z,z')\le\varepsilon} \frac{|(G\ast u_0)(z') - (G \ast u_0)(z)|}{\psi(z\,,z')} = 0.
\end{align}
We may apply Proposition \ref{pr:E:as} to see that
\begin{align*}
	\lim_{\varepsilon \to 0^+} \sup_{{z,z'\in I:0<\rho(z,z')\le\varepsilon}}\frac{|\mathscr{E}(z\,;z')|}{\psi(z\,,z')} =0 \quad \text{a.s.}
\end{align*}
Applying the last two displays to \eqref{u:w:tri} yields
\begin{align*}
	\lim_{\varepsilon\to0^+} \sup_{\substack{z,z'\in I\\ 0< \rho(z,z')\le\varepsilon}} \frac{|u(z')-u(z)|}{|\sigma(u(z))| \psi(z\,,z')}
	= \lim_{\varepsilon\to0^+} \sup_{\substack{z,z'\in I\\ 0< \rho(z,z')\le\varepsilon}} \frac{|w(z')-w(z)|}{\psi(z\,,z')} \quad \text{a.s.}
\end{align*}
Thanks to \eqref{w:mc} in Theorem \ref{th:w:lil:mc}, the right-hand side above is equal to $K$ a.s.

Finally, when $a=0$, \eqref{u:mc:G*u} still holds under the additional assumption \eqref{G*u:I}.
Moreover, Proposition \ref{pr:E:as} and \eqref{w:mc} in Theorem \ref{th:w:lil:mc} continue to hold when $a=0$.
This shows \eqref{u:mc} and completes the proof of Theorem \ref{th:she:mc}.
\end{proof}

\subsection{Proof of Corollary \ref{cor:she:ex}}

\begin{proof}
Fix $I=[a\,,T]\times[c\,,d]$, where $0<a<T$ and $0\le c<d \le L$.
Suppose $\theta>K$.
If on an event of positive probability, $F(\theta)$ is nonempty and contains a random point $z$, then on this event,
\[
	\lim_{\varepsilon\to0^+} \sup_{z,z'\in I: 0<\rho(z,z')\le\varepsilon} \frac{|u(z')-u(z)|}{|\sigma(u(z))| \rho(z\,,z')\sqrt{\log(1/\rho(z\,,z'))}} \ge \theta.
\]
This is a contradiction to \eqref{u:mc}. Hence, $F(\theta) = \varnothing$ a.s.

Suppose $0<\theta \le K$.
Theorem \ref{th:she:lil} implies that for every fixed $z\in I$, 
\begin{align*}
	\P\left\{ \lim_{\varepsilon\to0^+} \sup_{z' \in B_\rho^*(z,\varepsilon)} \frac{|u(z')-u(z)|}{\rho(z\,,z')\sqrt{\log(1/\rho(z\,,z'))}} = 0  \right\} = 1.
\end{align*}
By Fubini's theorem and the preceding, the expectation of the Lebesgue measure of $F(\theta)$ is
\begin{align*}
	&\E\left[\int_I \1_{F(\theta)} \d z \right] = \int_I \P\left\{ z \in F(\theta) \right\} \d z\\
	&= \int_I \P\left\{ \lim_{\varepsilon\to0^+} \sup_{z' \in B_\rho^*(z,\varepsilon)} \frac{|u(z')-u(z)|}{\rho(z\,,z')\sqrt{\log(1/\rho(z\,,z'))}} \ge \theta |\sigma(u(z))|  \right\} \d z = 0.
\end{align*}
Hence, $F(\theta)$ has Lebesgue measure 0 a.s.

Set $K' = \sqrt{12 c_2}$, where $c_2$ is the constant in \eqref{SLND:DNR}.
It is clear that for any rectangle $J \subset I$, \eqref{SLND:DNR} still holds on $J$ with the same constant $c_2$.
The proof of Theorem \ref{th:she:mc} and \eqref{w:mc} show that for any such rectangle $J$,
\begin{align}\label{u:mc:J}
	\lim_{\varepsilon\to0^+} \sup_{z,z'\in J: 0<\rho(z,z')\le\varepsilon} \frac{|u(z)-u(z')|}{|\sigma(u(z))| \rho(z\,,z')\sqrt{\log(1/\rho(z\,,z'))}} \ge K' \quad \text{a.s.}
\end{align}
and $K'\le K$.
For any $z\,,z'\in I$, let $J(z\,,z')$ denote the unique closed rectangle that contains $z$ and $z'$ as vertices.
Suppose $0<\theta<K'$.
In order to prove the last assertion of Corollary \ref{cor:she:ex}, we adapt the argument of \cite{OT74} to show that for any open rectangle $I'$ with rational vertices with $I' \cap I \ne \varnothing$, $\P\{F(\theta) \cap I' \ne \varnothing\} = 1$.
To show this, let $\Omega_0$ be the intersection of the events \eqref{u:mc:J} over all rectangles $J$ in $I$ with rational vertices, which satisfies $\P\{\Omega_0\}=1$.
On $\Omega_0$, there exist rational points $z_1\,,z_1' \in I' \cap I$ such that $\rho(z_1\,,z_1') \le 2^{-1}$ and
\begin{align*}
	\frac{|u(z_1)-u(z_1')|}{|\sigma(u(z_1'))|} > \theta \rho(z_1\,,z_1') \sqrt{\log(1/\rho(z_1\,,z_1'))}.
\end{align*}
Since $u$ and $\sigma$ are continuous, we may choose a rational $z_1^* \in J(z_1\,,z_1')$ such that $\rho(z_1\,,z_1^*) \le 2^{-1}$ and for all $z\in J(z_1\,,z_1^*)$,
\begin{align*}
	\frac{|u(z_1)-u(z)|}{|\sigma(u(z))|}
	> \theta \rho(z_1\,,z_1') \sqrt{\log(1/\rho(z_1\,,z_1'))}
	\ge \theta \rho(z_1\,,z) \sqrt{\log(1/\rho(z_1\,,z))}
\end{align*}
where the second inequality holds because $x \mapsto x \sqrt{\log(1/x)}$ is increasing on $[0\,,2^{-1}]$.
Next, since $J(z_1\,,z_1^*)$ is a rectangle with rational vertices, we can iterate the above procedure to find that, on $\Omega_0$, there are rational points $z_n\,,z_n'\,,z_n^* \in I' \cap I$, $n \in \N_+$ such that $\rho(z_n\,,z_n^*) \le 2^{-n}$,
\begin{align}\label{nested}
	J(z_n\,,z_n^*) \subset J(z_n\,,z_n') \subset J(z_{n-1}\,,z_{n-1}^*) \quad \text{for each $n \ge 2$}
\end{align}
and
\begin{align}\label{lower:fcn}
	\frac{|u(z_n)-u(z)|}{|\sigma(u(z))|}
	> \theta \rho(z_n\,,z) \sqrt{\log(1/\rho(z_n\,,z))}
	\quad\text{for all $z \in J(z_n\,,z_n^*)$.}
\end{align}
In particular, the nested property \eqref{nested} implies that $\bigcap_{n\in\N_+} J(z_n\,,z_n^*)$ is nonempty and contains a point $z_0$ which, thanks to \eqref{lower:fcn}, satisfies
\begin{align*}
	\frac{|u(z_n)-u(z_0)|}{|\sigma(u(z_0))|}
	> \theta \rho(z_n\,,z_0) \sqrt{\log(1/\rho(z_n\,,z_0))} \quad \text{for all $n \in \N_+$.}
\end{align*}
That is, $z_0 \in F(\theta) \cap I'$. 
This proves the claim, and hence $F(\theta)$ is dense in $I$.
\end{proof}

\subsection{Proof of Theorem \ref{th:she:sb}}

\begin{proof}
Suppose $b$ and $\sigma$ are bounded.
Let $\phi$ be as in the statement of the theorem.
Fix $z_0 \in [0\,,\infty)\times[0\,,L]$.
Let $m_\sigma = \inf_{x \in \R}|\sigma(x)|$ and $M_\sigma = \sup_{x \in \R}|\sigma(x)|$.
Let $\zeta>1$ be given by Proposition \ref{pr:linear}.
Choose and fix a number $p \in (1\,,\zeta)$.
Thanks to \eqref{phi}, we can find $\varepsilon_1 \in (0\,,1]$ such that
\begin{align}\label{eps:phi}
	\varepsilon \le \tfrac{1}{4K_0}(\phi(\varepsilon))^{-1/6} \quad
	\text{and} \quad \varepsilon^p \le \tfrac14 \varepsilon (\phi(\varepsilon))^{-1/6} \quad \text{for all $\varepsilon \in (0\,,\varepsilon_1]$,}
\end{align}
where $K_0$ is the constant in \eqref{G*u:z-z}.

Suppose first that $t_0>0$ and $m_\sigma>0$.
Recall the linearization error $\mathscr{E}$ defined in \eqref{E}.
For any $\varepsilon\in (0\,,\varepsilon_1]$ and $z \in B_\rho(z_0\,,\varepsilon)$,
if $|u(z)-u(z_0)| \le \varepsilon(\phi(\varepsilon))^{-1/6}$ and $|\mathscr{E}(z_0\,;z)| \le \varepsilon^p$, then
\begin{align}\begin{split}\label{u:sb:w-w}
	|w(z)-w(z_0)|
	&\le |\sigma(u(z_0))|^{-1} \left( |\tilde{u}(z)-\tilde{u}(z_0)| + |\mathscr{E}(z_0\,;z)| \right)\\
	& \le m_\sigma^{-1} \left( \varepsilon(\phi(\varepsilon))^{-1/6} + |(G\ast u_0)(z)-(G\ast u_0)(z_0)| + \varepsilon^p \right)\\
	&\le m_\sigma^{-1}\left( \tfrac54 \varepsilon(\phi(\varepsilon))^{-1/6} + K_0 \varepsilon^2\right)
	\le K_1 \varepsilon(\phi(\varepsilon))^{-1/6},
\end{split}\end{align}
where $K_1 = 3m_\sigma^{-1}/2$ and the last two lines follow from \eqref{G*u:z-z} and \eqref{eps:phi}.
It follows from the preceding, \eqref{w:sb}, and Proposition \ref{pr:P:supE} that for all $\varepsilon\in (0\,,\varepsilon_1]$,
\begin{align*}
	&\textstyle\P\left\{ \sup_{z \in B_\rho(z_0,\varepsilon)}|u(z)-u(z_0)| \le \varepsilon (\phi(\varepsilon))^{-1/6} \right\}\\
	&\textstyle\le \P\left\{ \sup_{z \in B_\rho(z_0,\varepsilon)}|w(z)-w(z_0)| \le \frac{K_1\varepsilon}{\phi(\varepsilon)^{1/6}}  \right\}
	+ \P\left\{ \sup_{z \in B_\rho(z_0,\varepsilon)}|\mathscr{E}(z_0\,;z)| > \varepsilon^p\right\}\\
	&\textstyle \le \exp\left( - c_0 K_1^{-6} \phi(\varepsilon) \right) + C \varepsilon^{-6(1+\zeta)}\exp\left( - \frac{1}{C\varepsilon^{\zeta-p}\log_+(1/\varepsilon)} \right).
\end{align*}
Take $C_0 = c_0K_1^{-6}/2$.
By \eqref{phi}, we can find $\varepsilon_2 \in (0\,,\varepsilon_1)$ such that for all $\varepsilon \in (0\,,\varepsilon_2)$,
\begin{align*}\textstyle
	\P\left\{ \sup_{z \in B_\rho(z_0,\varepsilon)}|u(z)-u(z_0)| \le \varepsilon(\phi(\varepsilon))^{-1/6} \right\}
	\le \e^{-C_0 \phi(\varepsilon)}.
\end{align*}
Next, let $K_2 = 1/(4M_\sigma)$. For $\varepsilon \in (0\,,\varepsilon_2)$, $z \in B_\rho(z_0\,,\varepsilon)$, if $|w(z)-w(z_0)|\le K_2 \varepsilon(\phi(\varepsilon))^{-1/6}$ and $|\mathscr{E}(z_0\,;z)| \le \varepsilon^p$, then by \eqref{E} and \eqref{G*u:z-z},
\begin{align}\begin{split}\label{u:sb:u-u}
	|u(z)-u(z_0)|
	&\le \varepsilon^p + M_\sigma K_2 \varepsilon(\phi(\varepsilon))^{-1/6} + 2K_0 \varepsilon^2\\
	&\le \varepsilon (\phi(\varepsilon))^{-1/6}.
\end{split}\end{align}
Hence, we can obtain in a similar way a reverse inequality for the small-ball probabilities for $\varepsilon \in (0\,,\varepsilon_2)$ using \eqref{w:sb} and Proposition \ref{pr:P:supE}:
\begin{align*}
	&\textstyle\exp\left( - c_1 K_2^{-6} \phi(\varepsilon) \right)\le \P\left\{ \sup_{z \in B_\rho(z_0,\varepsilon)}|w(z)-w(z_0)| \le K_2 \varepsilon(\phi(\varepsilon))^{-1/6} \right\}\\
	&\textstyle\le \P\left\{ \sup_{z \in B_\rho(z_0,\varepsilon)}|u(z)-u(z_0)| \le \frac{\varepsilon}{\phi(\varepsilon)^{1/6}} \right\}
	+ \P\left\{ \sup_{z \in B_\rho(z_0,\varepsilon)}|\mathscr{E}(z_0\,;z)| > \varepsilon^p \right\}\\
	&\textstyle \le \P\left\{ \sup_{z \in B_\rho(z_0,\varepsilon)}|u(z)-u(z_0)| \le \frac{\varepsilon}{\phi(\varepsilon)^{1/6}} \right\}
	+ C \varepsilon^{-6(1+\zeta)} \exp\left( - \frac{1}{C\varepsilon^{\zeta-p}\log_+(1/\varepsilon)} \right).
\end{align*}
Let $C_1 = 2c_1 K_2^{-6}$.
Thanks to \eqref{phi} again, we may choose another small number $\varepsilon_0 \in (0\,,\varepsilon_2)$ to ensure that for all $\varepsilon \in (0\,,\varepsilon_0)$,
\begin{align*}
	\textstyle\P\left\{ \sup_{z \in B_\rho(z_0,\varepsilon)}|u(z)-u(z_0)| \le \varepsilon(\phi(\varepsilon))^{-1/6} \right\}
	\ge \e^{-C_1\phi(\varepsilon)}.
\end{align*}
This proves \eqref{u:sb}.

Finally, if $t_0=0$, then under \eqref{phi}, \eqref{G*u:chung} and $\sigma(u_0(x_0)) \ne 0$,
similarly to the argument in \eqref{u:sb:w-w} above, 
we can find $\tilde K_1>0$ such that if $|u(z)-u(z_0)| \le \varepsilon (\phi(\varepsilon))^{-1/6}$ and $|\mathscr{E}(z_0\,;z)| \le \varepsilon^p$, then
\begin{align*}
	 |w(z)-w(z_0)| \le \tilde K_1 |\sigma(u_0(x_0))|^{-1} \varepsilon (\phi(\varepsilon))^{-1/6}.
\end{align*}
Hence, by \eqref{w:sb} and Proposition \ref{pr:P:supE}, for all $\varepsilon>0$ small,
\begin{align*}
	&\textstyle\P\left\{ \sup_{z \in B_\rho(z_0,\varepsilon)}|u(z)-u(z_0)| \le \varepsilon (\phi(\varepsilon))^{-1/6} \right\}\\
	&\textstyle\le \P\left\{ \sup_{z \in B_\rho(z_0,\varepsilon)}|w(z)-w(z_0)| \le \frac{\tilde K_1\varepsilon}{|\sigma(u_0(x_0))|\phi(\varepsilon)^{1/6}}  \right\}
	+ \P\left\{ \sup_{z \in B_\rho(z_0,\varepsilon)}|\mathscr{E}(z_0\,;z)| > \varepsilon^p\right\}\\
	&\textstyle \le \exp\left( - c_0 \tilde K_1^{-6} |\sigma(u_0(x_0))|^6 \phi(\varepsilon) \right) + C \varepsilon^{-6(1+\zeta)}\exp\left( - \frac{1}{C\varepsilon^{\zeta-p}\log_+(1/\varepsilon)} \right).
\end{align*}
This leads to the upper bound in \eqref{u:sb:t=0} for some constant $C_0>0$. 
Similarly to the argument in \eqref{u:sb:u-u} above, we can find $\tilde K_2>0$ such that if $|w(z)-w(z_0)| \le \tilde K_2 |\sigma(u_0(x_0))|^{-1} \varepsilon (\phi(\varepsilon))^{-1/6}$ and $|\mathscr{E}(z_0\,; z)| \le \varepsilon^p$, then
\begin{align*}
	|u(z)-u(z_0)| \le \varepsilon^p + K_2' \varepsilon (\phi(\varepsilon))^{-1/6} + \varepsilon^q \le \varepsilon (\phi(\varepsilon))^{-1/6}.
\end{align*}
Hence, by \eqref{w:sb} and Proposition \ref{pr:P:supE}, for all $\varepsilon > 0$ small,
\begin{align*}
	&\textstyle\exp\left( - c_1 \tilde K_2^{-6} |\sigma(u_0(x_0))|^6 \phi(\varepsilon) \right)\\
	&\textstyle\le \P\left\{ \sup_{z \in B_\rho(z_0,\varepsilon)}|w(z)-w(z_0)| \le \tilde K_2 |\sigma(u_0(x_0))|^{-1} \varepsilon(\phi(\varepsilon))^{-1/6} \right\}\\
	&\textstyle \le \P\left\{ \sup_{z \in B_\rho(z_0,\varepsilon)}|u(z)-u(z_0)| \le \frac{\varepsilon}{\phi(\varepsilon)^{1/6}} \right\}
	+ C \varepsilon^{-6(1+\zeta)} \exp\left( - \frac{1}{C\varepsilon^{\zeta-p}\log_+(1/\varepsilon)} \right).
\end{align*}
This yields the lower bound in \eqref{u:sb:t=0}.
\end{proof}

\subsection{Proof of Theorem \ref{th:she:chung}}

\begin{proof}
Fix $z_0 \in (0\,,\infty) \times [0\,,L]$ and write
$\varphi(\varepsilon) = \varepsilon^{-1}(\log\log(1/\varepsilon))^{1/6}$.
By \eqref{G*u:z-z},
\begin{align}\label{u:chung:G*u}
	\liminf_{\varepsilon\to0^+} \varphi(\varepsilon) \sup_{z\in B_\rho(z_0,\varepsilon)} |(G \ast u_0)(z) - (G \ast u_0)(z_0)| = 0.
\end{align}
By Proposition \ref{pr:E:as},
\begin{align*}
	\liminf_{\varepsilon\to0^+} \varphi(\varepsilon) \sup_{z\in B_\rho(z_0,\varepsilon)} |\mathscr{E}(z_0\,;z)| = 0 \quad \text{a.s.}
\end{align*}
The last two displays applied to \eqref{u:w:tri} yield
\begin{align*}
	&\liminf_{\varepsilon\to0^+} \varphi(\varepsilon) \sup_{z\in B_\rho(z_0,\varepsilon)} |u(z)-u(z_0)|\\
	&= |\sigma(u(z_0))| \liminf_{\varepsilon\to0^+} \varphi(\varepsilon) \sup_{z\in B_\rho(z_0,\varepsilon)} |w(z)-w(z_0)| = |\sigma(u(z_0))| C_2 \;\;\text{a.s.}
\end{align*}
where the last equality is due to \eqref{w:chung} in Theorem \ref{th:w:sb:chung}.

Finally, when $t_0=0$, \eqref{u:chung:G*u} still holds under the additional assumption \eqref{G*u:chung}.
Also, Proposition \ref{pr:E:as} and \eqref{w:chung} in Theorem \ref{th:w:sb:chung} continue to hold when $t_0=0$.
This leads to the same conclusion and concludes the proof of Theorem \ref{th:she:chung}.
\end{proof}

\section{Proofs for the open KPZ equation}\label{s:kpz}

\subsection{Proof of Theorem \ref{th:kpz:lil:mc}}

\begin{proof}
Fix $z_0\in [0\,,\infty)\times[0\,,1]$ and $\epsilon_0\in (0\,,1)$ such that $B_\rho(z_0\,,\epsilon_0) \subset [0\,,\infty) \times [0\,,1]$.
The random field $u$ is the solution to \eqref{she} with $b=0$ and $\sigma(u)=u$.
Since $u$ is continuous and strictly positive (see \cite[Proposition 2.7]{CS18} and \cite[Proposition 4.2]{P19}), this implies that $\sigma^{-1}\{0\} = \{0\}$ is polar for $u$ and 
$\Delta_0:= \inf_{z \in B_\rho(z_0,\epsilon_0)}u(z)$ is a strictly positive random variable. 
We adopt the idea of \cite{FKM15} to argue as follows. 
By Taylor expansion, for any $u\,,\bar{u}>0$,
\[
	\log{\bar{u}} = \log{u} + \frac{\bar{u}-u}{u} - \frac{(\bar{u}-u)^2}{2v^2},
\]
where $v = v(u\,,\bar{u})$ takes values between $u$ and $\bar{u}$. Applying this with $h(z) = \log u(z)$ and using \eqref{E} yield the following:
\begin{align}\label{h:tri:ub}
	&|h(z)-h(z_0)|
	\le \frac{|u(z)-u(z_0)|}{u(z_0)} + \frac{|u(z)-u(z_0)|^2}{2\Delta_0^2}\\
	& \le |w(z)-w(z_0)| + \frac{|(G\ast u_0)(z)-(G\ast u_0)(z_0)|}{u(z_0)} + \frac{|\mathscr{E}(z_0\,;z)|}{u(z_0)} + \frac{|u(z)-u(z_0)|^2}{2\Delta_0^2}.\notag
\end{align}
Similarly,
\begin{align}\label{h:tri:lb}
	&|h(z)-h(z_0)|
	\ge \frac{|u(z)-u(z_0)|}{u(z_0)} - \frac{|u(z)-u(z_0)|^2}{2\Delta_0^2}\\
	& \ge |w(z)-w(z_0)| - \frac{|(G\ast u_0)(z)-(G\ast u_0)(z_0)|}{u(z_0)} - \frac{|\mathscr{E}(z_0\,;z)|}{u(z_0)} - \frac{|u(z)-u(z_0)|^2}{2\Delta_0^2}.\notag
\end{align}
Let $\phi(z\,,z_0) = \rho(z\,,z_0)\sqrt{\log\log(1/\rho(z\,,z_0))}$.
Then, by Lemma \ref{lem:G*u} (or \eqref{G*u:z_0} when $t_0=0$), Proposition \ref{pr:E:as}, and Theorem \ref{th:she:lil}, respectively, we have
\begin{gather*}
	\lim_{\varepsilon\to0^+}\sup_{z\in B^*_\rho(z_0,\varepsilon)} \frac{|(G\ast u_0)(z)-(G\ast u_0)(z_0)|}{u(z_0) \phi(z\,,z_0)} = 0,\\
	\lim_{\varepsilon\to0^+}\sup_{z\in B^*_\rho(z_0,\varepsilon)} \frac{|\mathscr{E}(z_0\,;z)|}{u(z_0) \phi(z\,,z_0)} = 0 \quad \text{a.s.,}\\
	\lim_{\varepsilon\to0^+}\sup_{z\in B^*_\rho(z_0,\varepsilon)} \frac{|u(z)-u(z_0)|^2}{2 \Delta_0^2 \phi(z\,,z_0)} = 0 \quad \text{a.s.}
\end{gather*}
These together with \eqref{w:lil} imply that, a.s.,
\begin{align*}
	\lim_{\varepsilon\to0^+}\sup_{z\in B^*_\rho(z_0,\varepsilon)} \frac{|h(z)-h(z_0)|}{\phi(z\,,z_0)}
	= \lim_{\varepsilon\to0^+} \sup_{z\in B^*_\rho(z_0,\varepsilon)} \frac{|w(z)-w(z_0)|}{\phi(z\,,z_0)} = K_0.
\end{align*}
This proves \eqref{kpz:lil}.

We now turn to the proof of \eqref{kpz:mc}. 
Fix $I=[a\,,T]\times[c\,,d]$. We may use the same argument as in the first part of this proof to show that $\Delta := \inf_{z\in I} u(z)$ is a strictly positive random variable, and for all $z\,,z' \in I$,
\begin{align*}
	&|w(z')-w(z)| - \frac{|(G\ast u_0)(z')-(G\ast u_0)(z)|}{\Delta} - \frac{|\mathscr{E}(z\,;z')|}{\Delta} - \frac{|u(z')-u(z)|^2}{2\Delta^2}\\
	&\le |h(z')-h(z)|\\
	&\le |w(z')-w(z)| + \frac{|(G\ast u_0)(z')-(G\ast u_0)(z)|}{\Delta} + \frac{|\mathscr{E}(z\,;z')|}{\Delta} + \frac{|u(z')-u(z)|^2}{2\Delta^2}.
\end{align*}
Let $\psi(z\,,z') = \rho(z\,,z')\sqrt{\log(1/\rho(z\,,z'))}$.
Then, by Lemma \ref{lem:G*u} (or \eqref{G*u:I} when $a=0$), Proposition \ref{pr:E:as}, and Theorem \ref{th:she:mc} (recalling that $\sigma^{-1}\{0\}$ is polar for $u$),
\begin{gather*}
	\lim_{\varepsilon\to0^+}\sup_{z,z'\in I: 0<\rho(z,z')\le\varepsilon} \frac{|(G\ast u_0)(z')-(G\ast u_0)(z)|}{\Delta \psi(z\,,z')} = 0,\\
	\lim_{\varepsilon\to0^+}\sup_{z,z'\in I: 0<\rho(z,z')\le\varepsilon}\frac{|\mathscr{E}(z\,;z')|}{\Delta \psi(z\,,z')} = 0 \quad \text{a.s.,}\\
	\lim_{\varepsilon\to0^+} \sup_{z,z'\in I: 0<\rho(z,z')\le\varepsilon}\frac{|u(z')-u(z)|^2}{2\Delta^2 \psi(z\,,z')} = 0 \quad \text{a.s.}
\end{gather*}
The above and \eqref{w:mc} together imply that, a.s.,
\begin{align*}
	\lim_{\varepsilon\to0^+} \sup_{z,z'\in I: 0<\rho(z,z')\le\varepsilon} \frac{|h(z')-h(z)|}{\psi(z\,,z')}
	= \lim_{\varepsilon\to0^+} \sup_{z,z'\in I: 0<\rho(z,z')\le\varepsilon} \frac{|w(z')-w(z)|}{\psi(z\,,z')} = K_1.
\end{align*}
This proves \eqref{kpz:mc} and hence completes the proof of Theorem \ref{th:kpz:lil:mc}.
\end{proof}

\subsection{Proof of Corollary \ref{cor:kpz:ex}}

\begin{proof}
The proof is the same as that of Corollary \ref{cor:she:ex} and is therefore omitted.
\end{proof}

\subsection{Proof of Theorem \ref{th:kpz:chung}}

\begin{proof}
Write $\varphi(\varepsilon)=\varepsilon^{-1}(\log\log(1/\varepsilon))^{1/6}$.
By Lemma \ref{lem:G*u} (or \eqref{G*u:chung} when $t_0=0$), Proposition \ref{pr:E:as}, and Theorem \ref{th:she:lil}, we have
\begin{gather*}
	\limsup_{\varepsilon\to0^+} \varphi(\varepsilon) \sup_{z\in B_\rho(z_0,\varepsilon)} |(G\ast u_0)(z)-(G\ast u_0)(z_0)| = 0,\\
	\limsup_{\varepsilon\to0^+} \varphi(\varepsilon) \sup_{z\in B_\rho(z_0,\varepsilon)} |\mathscr{E}(z_0\,;z)| = 0 \quad \text{a.s.,}\\
	\limsup_{\varepsilon\to0^+} \varphi(\varepsilon) \sup_{z\in B_\rho(z_0,\varepsilon)}|u(z)-u(z_0)|^2 = 0 \quad \text{a.s.}
\end{gather*}
Applying the preceding to \eqref{h:tri:ub} and \eqref{h:tri:lb} yields
\begin{align*}
	\liminf_{\varepsilon\to0^+} \varphi(\varepsilon) \sup_{z\in B_\rho(z_0,\varepsilon)} |h(z)-h(z_0)|
	=\liminf_{\varepsilon\to0^+} \varphi(\varepsilon) \sup_{z\in B_\rho(z_0,\varepsilon)}|w(z)-w(z_0)| = C_2
\end{align*}
a.s., where the last equality follows from \eqref{w:chung} in Theorem \ref{th:w:sb:chung}.
\end{proof}

{\bf Acknowledgments.}
C.Y. Lee was supported in part by the Shenzhen Peacock grant 2025TC0013.
The authors thank Professor Davar Khoshnevisan for his comments on the open problems. The authors also thank two anonymous referees for their careful reading and helpful comments which have led to some improvements in the paper.

\bibliography{spde12}
\bibliographystyle{amsplain}

\end{document}